\documentclass{amsart}[11pt]

\usepackage[hidelinks]{hyperref}
\usepackage{amssymb, amsmath, amsthm, amsfonts}
\usepackage{verbatim}
\usepackage{bbm}
\usepackage{enumerate}
\usepackage{dsfont}
\usepackage{upgreek}
\usepackage[mathscr]{eucal}
\usepackage[dvipsnames]{xcolor}
\usepackage{tikz}
\usetikzlibrary{arrows.meta,calc}
\usepackage{kotex}
\usepackage{mathrsfs}
\usepackage{enumitem}

\usepackage{relsize}

\def\wt#1{\widetilde{#1}}

\newcommand{\N}{\mathbb{N}}
\newcommand{\R}{\mathbb{R}}
\newcommand{\Z}{\mathbb{Z}}

\newcommand{\la}{\lambda}

\newcommand{\supp}{\operatorname{supp}}
\newcommand{\dist}{\operatorname{dist}}

\newtheorem{thm}{Theorem}[section]
\newtheorem{defn}[thm]{Definition}
\newtheorem{prop}[thm]{Proposition}
\newtheorem{cor}[thm]{Corollary}
\newtheorem{lem}[thm]{Lemma}
\numberwithin{equation}{section}
\newtheorem{rem}{Remark}

\newcommand{\cA}{\mathcal A}

\newcommand\fa{{\mathbf n}}
\newcommand{\fb}{{\mathbf n^{\prime}}}
\newcommand\fr{{\mathfrak r}}
\newcommand{\fn}{\mathfrak n}
\newcommand{\fR}{\mathfrak R}
\newcommand{\fT}{\mathfrak T}
\newcommand{\fs}{\mathfrak s}
\newcommand{\fS}{\mathfrak D}

\newcommand{\ep}{{\varepsilon}}
\newcommand{\inp}[2]{\langle #1, #2\rangle}

\newcommand{\mmt}{\sqrt{\mu\mu'}}

\newcommand{\smu}{\sqrt\mu}
\newcommand{\smp}{\sqrt{\mu'}}
\newcommand{\mup}{{\mu'}}

\newcommand{\fP}{\mathcal O}

\newcommand{\cP}{\mathcal P}
\newcommand{\cR}{\mathcal R}

\newcommand{\Be}{\begin{equation}}
\newcommand{\Ee}{\end{equation}}

\newcommand{\cD}{\mathcal D}
\newcommand{\cL}{\mathcal L}

\newcommand{\cB}{\mathcal B}

\newcommand{\epz}{\ep_{\mspace{-1.5mu}\mathsmaller{
0
}}}
\newcommand{\cz}{c_{\mspace{-1.5mu}\mathsmaller{
0
}}}

\newcommand{\zc}{{{\mspace{-1.5mu}\mathsmaller{
0
}}}}

\newcommand{\sprp}{{\!\perp}}

\newcommand{\dels}{{\delta_{\mspace{-1.5mu} \ast}}}
\newcommand{\delt}{{\delta_{\mspace{-2.5mu}\mathsmaller\vartriangle}}}

\begin{document}

\author[Eunhee Jeong]{Eunhee Jeong}
\address[Jeong]{Department of Mathematics Education, and  Institute of Pure and Applied Mathematics, Jeonbuk National University, Jeonju 54896, Republic of Korea}
\email{eunhee@jbnu.ac.kr}

\author[Sanghyuk Lee]{Sanghyuk Lee}
\address[Lee]{Department of Mathematical Sciences and RIM, Seoul National University, Seoul 08826, Republic of Korea}
\email{shklee@snu.ac.kr}

\author[Jaehyeon Ryu]{Jaehyeon Ryu}
\address[Ryu]{Department of Mathematics, Ewha Womans University, Seoul
03760, Republic of Korea}
\email{jhryu67@ewha.ac.kr}

\keywords{Hermite functions, Spectral projection}
\makeatletter
\@namedef{subjclassname@2020}{\textup{2020} Mathematics Subject Classification}
\makeatother
\subjclass[2020]{42B99  (primary);  42C10 (secondary)}

\title[Endpoint Eigenfunction Bound]{The Endpoint Eigenfunction Bound for the Hermite Operator in Two Dimensions}

\begin{abstract}
In this paper, we establish the optimal \(L^2(\mathbb{R}^2)\to L^{10/3}(\mathbb{R}^2)\) endpoint estimate for the spectral projection operator associated with the Hermite operator on \(\mathbb{R}^2\). This completes a long-standing line of inquiry into sharp eigenfunction bounds for the Hermite operator, developed through the works of Thangavelu, Karadzhov, Koch--Tataru, and others. In higher dimensions \(d\ge 3\), the corresponding endpoint estimates
\[
    L^2(\mathbb{R}^d)\to L^{\frac{2(d+3)}{d+1}}(\mathbb{R}^d)
\]
were recently established by the present authors. Together with these earlier results, the present work fully resolves the problem of optimal \(L^2\to L^q\) eigenfunction bounds for Hermite spectral projections in all dimensions. Although our approach builds on our previous method, we overcome its limitations through a multiscale decomposition in space and time relative to the degeneracy set, combined with an asymmetric refinement on the input side and almost orthogonality.
\end{abstract}

\maketitle

\section{Introduction}
Let \(H=-\Delta+|x|^2\) be the Hermite operator on \(\mathbb{R}^d\), whose spectrum is \(2\mathbb{N}_0+d\). Here \(\mathbb{N}_0$ denotes $\mathbb{N}\cup\{0\}\). A standard orthonormal eigenbasis is given by tensor products of one-dimensional Hermite functions, each of which is a Hermite polynomial multiplied by a Gaussian factor, up to normalization. Throughout this article, \(\lambda\) denotes an eigenvalue of \(H\), i.e., \(\lambda\in 2\mathbb{N}_0+d\), and \(\Pi_\lambda\) denotes the spectral projection onto the corresponding eigenspace.  

\subsection{Background and the main result} Various authors have studied the problem of determining optimal bounds for the operator norm
\[
    \|\Pi_\lambda\|_{2\to q}, \quad 2\le q\le \infty
\]
in terms of \(\lambda\), where \(\|T\|_{s\to r}\) denotes the operator norm of \(T\) from \(L^s(\mathbb{R}^d)\) to \(L^r(\mathbb{R}^d)\). Such bounds yield estimates for the \(L^q\) norms of \(L^2\)-normalized eigenfunctions. They also have applications to \(L^p\) convergence of Hermite Bochner--Riesz means \cite{AW65, K94, Th98a, CLWY22, LR22}; see also \cite{CLSY20, KR07, So98, SZ98} for related results. Moreover, bounds on \(\|\Pi_\lambda\|_{2\to q}\) have been shown to have applications to the strong unique continuation property for parabolic differential inequalities. We refer interested readers to \cite{E00, EV01, JLR24a, KT09} and references therein.
  

When $d=1$, the optimal $L^2(\R^d)$--$L^q(\R^d)$ bounds on $\Pi_\lambda$ are relatively straightforward. These bounds follow from the known optimal $L^q(\R)$ bounds on one-dimensional Hermite functions. In fact, 
$\|\Pi_\lambda\|_{2\to q}\sim \lambda^{\max\{-1/12-1/(6q),-1/4+1/(2q)\}}$ for $q\neq 4$. At the critical case $q=4$, it is known that  $\|\Pi_\lambda\|_{2\to 4}\sim \lambda^{-1/8}(\log \lambda)^{1/4}$. 
An early source for the matching upper and lower bounds at the
critical exponent is Markett \cite[Lemma~1]{M84};
see also \cite{GS94}.

In higher dimensions \(d\ge 2\), the works of Thangavelu \cite{Th93}, Karadzhov \cite{K94}, and Koch--Tataru \cite{KT05} established the optimal bounds for \(\|\Pi_\lambda\|_{2\to q}\), \(2\le q\le \infty\), except at the endpoint
\[
    q_e (d)=\frac{2(d+3)}{d+1}.   
\]
More precisely, they proved that
\[
    \|\Pi_\lambda\|_{2\to q}\sim \lambda^{\beta(q)}
\]
for \(q\in [2,\infty]\setminus\{q_e (d)\}\), where
\[
    \beta(q)
    =
    \max\left\{
        -\frac{1}{2}+\frac{d}{2}\delta(q),
        -\frac{1}{6}+\frac{d}{6}\delta(q),
        -\frac{1}{2}\delta(q)
    \right\},
    \qquad
    \delta(q):=\frac{1}{2}-\frac{1}{q}.
\]
The lower bounds   $\|\Pi_\lambda\|_{2\to q}\gtrsim  \lambda^{\beta(q)}$ are comparatively easier and can be obtained by constructing suitable example functions (see \cite{KT05}). For the upper bounds, from the viewpoint of interpolation, the critical cases are \(q=\infty\), \(q=2d/(d-2)\), and \(q=q_e (d)\), since $\|\Pi_\lambda\|_{2\to 2}=1$; when \(d=2\), the first two cases coincide at \(q=\infty\). The case \(q=\infty\) follows from the kernel estimate recorded by Thangavelu \cite[Lemma 3.2.2]{Th93}, while the case \(q=2d/(d-2)\) for \(d\ge3\) is due to Karadzhov \cite{K94}. The bounds in the remaining ranges \(q\in(q_e (d),2d/(d-2))\) and \(q\in(2,q_e (d))\) are due to Koch--Tataru \cite{KT05}, where the upper endpoint of the first interval is understood as \(\infty\) when \(d=2\). Their result gives 
the bound  with logarithmic loss  at the remaining endpoint \(q=q_e (d)\), 
\[
\|\Pi_\lambda\|_{2\to q_e(d)}
\lesssim
\lambda^{-\frac{1}{2(d+3)}}(\log\lambda)^{1/q_e(d)},
\]
which follows rather directly from the local estimate \eqref{e:KTlocal} below; see also \cite{KTZ07}. 

Although this estimate appears analogous to the one-dimensional endpoint result, it is natural to expect that the logarithmic factor can be removed in dimensions \(d\ge2\), in contrast to the case \(d=1\), where \(q_e (1)=4\). This expectation is supported by the fact that there is no eigenfunction of \(H\) which concentrates in an equidistributed manner in a neighborhood of the sphere (see \cite[pp.~375--377]{KT05})
\[ \sqrt{\lambda}\mathbb S^{d-1}:=\{x\in\mathbb{R}^d: |x|=\sqrt{\lambda}\}.\]

The optimal upper bound at the missing endpoint \(q=q_e (d)\) was recently obtained by the authors \cite{JLR24b} for \(d\ge3\). The key idea there is to exploit a weak interaction gain arising from asymmetric localization on the input and output sides near the sphere $\sqrt\lambda \mathbb S^{d-1}$. However, when \(d=2\), this approach does not provide the necessary gain due to the weaker dispersive effect. Consequently, the optimal upper bound for \(\|\Pi_\lambda\|_{2\to q_e (2)}\) in two dimensions remained open.

The main objective of the present paper is to prove the missing two-dimensional endpoint estimate, thereby completing the optimal \(L^2\)--\(L^q\) bounds for Hermite spectral projections in all dimensions \(d\ge2\).  Note that $q_e (2)=10/3.$

\begin{thm}\label{thm:main}   Let $\lambda\in 2\N_0+2$ and $q_e=10/3$. 
   There exists a constant $C>0$ such that
  \Be
  \label{2-d}
        \|\Pi_\lambda\|_{L^2(\R^2)\to L^{q_e }(\R^2)}\le C\lambda^{-\frac{1}{10}}
    \Ee
\end{thm}

Combined with the preceding results, this completes the study of
$L^q$ eigenfunction bounds for the Hermite operator in all dimensions.

\subsection{Asymmetric localization}  
The \(L^2\)--\(L^q\) bounds for \(\Pi_\lambda\) have also been studied in local settings. Thangavelu \cite{Th98a} obtained optimal \(L^2(\mathbb{R}^d)\)--\(L^q(B)\) bounds for \(\Pi_\lambda\), where \(B\) is a fixed compact subset of \(\mathbb{R}^d\) (see also \cite{WZ25} for more refined recent results and \cite{OR24} for related work on localization along smooth curves). In his work, it was observed that the behavior of \(\Pi_\lambda\) near the sphere
\(\sqrt{\lambda}\mathbb{S}^{d-1}\)   
plays a crucial role in determining the global bound for \(\Pi_\lambda\). Later, Koch--Tataru \cite{KT05} proved the sharp estimates  for  \(\Pi_\lambda\)  restricted to  annular regions near the sphere \(\sqrt{\lambda}\mathbb{S}^{d-1}\). More precisely, set
\[
    A_{\lambda,\mu} =
    \begin{cases}
    \big\{
        x\in \mathbb{R}^d:
        \mu < 1-|\lambda^{-\frac{1}{2}}x| \le 2\mu
    \big\},
    \quad   &\mu> \lambda^{-\frac{2}{3}},
    \\[5pt]
    \big\{
        x\in \mathbb{R}^d:
         |1-|\lambda^{-\frac{1}{2}}x| |\le 2\mu
    \big\},
    \quad   & \lambda^{-\frac{2}{3}} \ge \mu 
    \end{cases} 
\]
for  $ \mu\in 2^{-\mathbb{N}} \cap [2^{-1}\lambda^{-2/3} ,2^{-1}]$.  
They  obtained
\begin{align}\label{e:KTlocal}
    \|\chi_{A_{\lambda,\mu}}\Pi_\lambda\|_{L^2(\mathbb{R}^d)\to L^q(\mathbb{R}^d)}
  \lesssim 
    \begin{cases}
        \lambda^{-\frac{1}{2}\delta(q)}
        \mu^{\frac{1}{4}-\frac{d+3}{4}\delta(q)},
        & 2\le q\le \frac{2(d+1)}{d-1}, \\[3pt]
        (\lambda\mu)^{\frac{d}{2}\delta(q)-\frac{1}{2}},
        & \frac{2(d+1)}{d-1}\le q\le \infty.  
    \end{cases}
\end{align}
The estimate \eqref{e:KTlocal} was a key ingredient in their result on the Hermite spectral projection operator in \cite{KT05}.

In order to prove the estimate \eqref{2-d}, we primarily follow the approach of \cite{JLR24b}, where an improvement over the previously known bounds on $\Pi_\lambda$ was exploited under asymmetric localization near $\sqrt{\lambda}\mathbb S^{d-1}$.  The following is our main estimate, which proves the endpoint eigenfunction bounds in two dimensions.

\begin{thm}\label{thm:decay}
    Let $\lambda\in 2\N_0+2$ and  $\mu,\mu'\in 2^{-\N}$.   
        Suppose $C_\ast\lambda^{-2/3}< \mu'\le\mu\le \cz$, where $C_\ast>0$ is sufficiently large and $\cz>0$ is sufficiently small. Then, there exists a constant $C>0$, independent of $\lambda,\mu,\mu'$, such that
    \begin{align}
    \label{cap}
         \|\chi_{A_{\lambda,\mu}}\Pi_\lambda\chi_{A_{\lambda,\mu'}}\|_{q'_e \to q_e }\le C\lambda^{-\frac{1}{5}}\biggl(\frac{\mu'}{\mu}\biggr)^{\frac{1}{20}}.
    \end{align}
\end{thm}

Theorem \ref{thm:main} follows from Theorem \ref{thm:decay} by the same
$TT^\ast$ argument and dyadic annular summation as in
\cite[pp.~1315--1316 and Section~2.7]{JLR24b}.

 This improvement was first observed in \cite{JLR24b} for $d\ge 3$. In the equivalent $T^\ast T$ formulation, comparison with \eqref{e:KTlocal} shows that \eqref{cap} gains precisely the additional factor $(\mu'/\mu)^{1/20}$; this geometric decay in the relative annular scale makes the dyadic summation converge. This demonstrates that the contribution from $\chi_{A_{\lambda,\mu}}\Pi_\lambda\chi_{A_{\lambda,\mu'}}$ becomes weaker as $\mu'/\mu$ gets smaller.

\medskip

We conclude the introduction with an outline of the proof. 

\subsection{Strategy of the proof}

As in \cite{JLR24b}, we use the Hermite--Schr\"odinger propagator to
reduce Theorem~\ref{thm:decay} to an estimate for a localized
oscillatory integral operator. After rescaling and applying a sectorial
Whitney-type decomposition, it is enough to consider a pair of
asymmetrically localized rectangles $\cB$ and $\cB'$ and to prove
Theorem~\ref{thm:core}. The main point is to obtain an additional gain
expressed as a positive power of the asymmetric ratio $\mu'/\mu$.

The geometry of the resulting oscillatory integral is governed by
\[
    \cD(x,y)=1+\langle x,y\rangle^2-|x|^2-|y|^2
    \quad\text{and}\quad
    S(x,y)=\cos^{-1}\langle x,y\rangle .
\]
Identity~\eqref{i:ptcP-Q} shows that $\cD$ plays the role of a
discriminant for the stationary-point equation in the time variable.
When $\cD>0$, there are two nondegenerate stationary points near
$S(x,y)$; when $\cD<0$, there is no stationary point; and when
$\cD=0$, the first two time derivatives of the phase vanish at
$t=S(x,y)$, producing a cubic degeneracy. The time decomposition used
in \cite{JLR24b} alone does not provide the gain required at the
two-dimensional endpoint. The additional ingredient is a multiscale
decomposition involving both the spatial and time variables, together
with an asymmetric refinement on the input side.
Together with the diagonal and almost-orthogonality estimates,
this scheme yields the factor \((\mu'/\mu)^{1/20}\) after interpolation.

\medskip
\noindent
\emph{Separation from the symmetric time.}
At the initial scale $\delt\sim\mu\mu'$, we first separate the part
supported away from $t=S(x,y)$. As described in
Section~\ref{sec:prune}, this contribution is decomposed dyadically
according to $|t-S(x,y)|$ and further localized on the input side. The
kernel and $L^2$ estimates in \eqref{l0} and \eqref{l2}, together with
the Cotlar--Stein argument developed in
Sections~\ref{subsec:away-contribution}--\ref{subsec:away-orthogonality},
yield Proposition~\ref{prop:estfE}. Thus, the main problem is reduced
to the operator supported near the symmetric time; see the
decomposition \eqref{d:top}--\eqref{e:decompm-1st}.

\medskip
\noindent
\emph{Recursive space--time decomposition.}
The bottom scale is $\dels\sim\lambda^{-2/3}\mu$, as defined in
\eqref{d:top&bottom}. Starting from $\delt$ and descending through the
dyadic scales $\dels\leq\delta\leq\delt$, we divide the spatial support
into the regions $\fS_\delta^+$, $\fS_\delta^-$, and
$\fS_\delta^\circ$ introduced in \eqref{fsd}--\eqref{fsdc}. In the
region where $|\cD(x,y)|\lesssim\delta$, the time support is divided
once more: the part at distance comparable to
$\delta^{1/2}\mu^{-1/2}$ from $S(x,y)$ is extracted, while the part
closer to $S(x,y)$ is passed to the next dyadic scale. The resulting
recursive decomposition is given in
\eqref{i:decomp-main}--\eqref{iterat}.

At each scale, the spatial regions are covered by pairs of rectangles
essentially tangent to the level curves of $\cD$. Their normal and
tangential dimensions are specified in \eqref{scale}, and their
directions are determined by the normalized gradients of $\cD$
introduced in \eqref{normal}--\eqref{normal'}.
The stability of these directions is established in
Lemmas~\ref{lem:locD-strong}--\ref{lem:locfab} and is used in Section~\ref{sec:decomp} to construct the required covering and prove the incidence bounds.
This gives the localized
operator decompositions \eqref{pdk} and \eqref{pds}.

\medskip
\noindent
\emph{Asymmetric refinement and the gain in $\mu'/\mu$.}
The gain in the asymmetric ratio is already visible at the level of
the fully refined localized pieces. In
Section~\ref{subsec:input-refinement}, each input rectangle $\fr'$ is
subdivided in its tangential direction into rectangles $\fs$ whose
tangential length is smaller by the factor $(\mu'/\mu)^{1/2}$; see
\eqref{s'}--\eqref{pdts'}. We denote the corresponding operators by
$\fP_{\delta,\fr,\fs}^{\kappa}$,
$\kappa\in\{+,-,\circ\}$.

For each such piece, Proposition~\ref{prop:L2est} gives the diagonal
estimate
\[
    \|\fP_{\delta,\fr,\fs}^{\kappa}\|_{2\to2}
    \lesssim
    \lambda^{-1}\delta^{\frac{1}{2}}(\mu\mu')^{-\frac{1}{4}};
\]
see \eqref{e:fP-dkap}. On the other hand, the corresponding
$L^1\to L^\infty$ estimate is recorded in \eqref{e:bds-1infty}, as a
consequence of the kernel estimates \eqref{e:ptbd-fPcn},
\eqref{p+circ}, and \eqref{i:exp-fP+}. Interpolation at $q_e=10/3$
gives
\[
    \|\fP_{\delta,\fr,\fs}^{\kappa}\|_{q_e'\to q_e}
    \lesssim
    \lambda^{-\frac{4}{5}}
    \delta^{\frac{1}{5}}\mu^{-\frac{1}{4}}(\mu')^{-\frac{3}{20}}.
\]
Since $\delta\leq\delt\sim\mu\mu'$, this already contains the required
improvement in $\mu'/\mu$ for a single fully refined piece. Thus, the
asymmetric gain originates in the comparison between the
$L^2\to L^2$ and $L^1\to L^\infty$ estimates after the additional
input-side refinement. The bottom-scale piece is treated in the same
way using \eqref{e:fP-ast}, \eqref{e:ptbd-fP*}, and
\eqref{e:bds-10}, together with the lower bound on $\mu'$ in
\eqref{mm}.

\medskip
\noindent
\emph{Estimates for the localized pieces.}
Section~\ref{sec:individual} treats the fully refined pieces according
to the behavior of the phase in the time variable. The pieces
corresponding to $\cD<0$, as well as the time-away pieces extracted
from $|\cD|\lesssim\delta$, are nonstationary and are estimated by the
van der Corput lemma using the lower bound \eqref{e:1stlb-cP}; see
\eqref{e:ptbd-fPcn}. At the bottom scale, the cubic behavior of the
phase gives \eqref{e:fP-ast} and \eqref{e:ptbd-fP*}.

For the pieces with $\cD>0$, stationary phase at the two nondegenerate
critical points gives the reduction in
Section~\ref{subsec:stationary-phase}. The mixed Hessian of each
reduced phase has rank one by Lemma~\ref{lem:mixedH} and is
nondegenerate between the transverse directions determined by the
tangential rectangles; see
\eqref{e:mixedH-unscaled}--\eqref{e:mixednunu'}. After the anisotropic
rescaling in \eqref{LL}--\eqref{i:G->wtG} and freezing the tangential
variables, a one-dimensional H\"ormander-type estimate gives the
required $L^2$ bound and completes the proof of
Proposition~\ref{prop:L2est}.

\medskip
\noindent
\emph{Almost orthogonality and summation.}
It remains to sum the fully refined pieces without losing the gain
already present in their individual estimates.
Proposition~\ref{prop:alortho-1} gives off-diagonal decay between
pieces associated with separated tangential rectangles, while
Proposition~\ref{prop:alortho-2} gives additional almost orthogonality
between separated tangential subrectangles $\fs$ within a fixed input
rectangle. These estimates are obtained by comparing the spatial
gradients of the phases and repeatedly integrating by parts. Combined
with the incidence properties from Section~\ref{sec:decomp}, they
allow the Cotlar--Stein lemma to be applied in
Section~\ref{subsec:global-L2}.

More precisely, Proposition~\ref{prop:main-components1-0} gives the
global $L^1\to L^\infty$ estimates, and
Proposition~\ref{prop:main-components2-2} follows by combining
Propositions~\ref{prop:alortho-1}, \ref{prop:alortho-2}, and
\ref{prop:L2est}. Interpolation of these two global bounds and dyadic
summation, split at $\delta=(\mu')^2$ in
Section~\ref{subsec:proof-main-p}, prove Proposition~\ref{main-p} and
hence Theorem~\ref{thm:core}.

\subsection*{Organization of the paper}
Section~\ref{sec:prelim} reduces the main theorem to the localized
oscillatory integral estimate and disposes of the contribution away
from the symmetric time. Section~\ref{sec:geometry} develops the
geometry of the level sets of $\cD$, and Section~\ref{sec:decomp}
constructs the recursive space--time decomposition and the
tangential rectangle coverings. Section~\ref{sec:alortho} establishes
the almost-orthogonality estimates, Section~\ref{sec:individual} proves
the bounds for the fully refined localized pieces, and
Section~\ref{sec:conclude} combines these estimates to complete the
proof.

\subsection*{Notation} For positive quantities $A$ and $B$, $A\lesssim B$ means that there exists a constant $C>0$, independent of $A$ and $B$, such that $A\le CB$. Moreover, we write $A\sim B$ if and only if $A\lesssim B$ and $B\lesssim A$. We write $A\ll B$ if $A\lesssim B$ and the implicit constant is sufficiently small. We also denote $A\lesssim_N B$ if the implicit constant depends on a parameter $N$.  
%
  We denote $\partial_x=(\partial_{x_1},\partial_{x_2})$ and
$\partial_y=(\partial_{y_1},\partial_{y_2})$. More generally, let
$w$ denote either $x$ or $y$. Let $R$ be a rectangle with side
lengths $L_1$ and $L_2$ in the orthonormal directions
$\nu_1,\nu_2\in\mathbb S^1$, respectively. We say that a smooth
function $\chi=\chi(w)$ is adapted to $R$ if $\supp\chi\subset R$ and
\[
\big|(\inp{\nu_1}{\partial_w})^N
      (\inp{\nu_2}{\partial_w})^M\chi(w)\big|
\lesssim_{N,M}L_1^{-N}L_2^{-M}
\]
for all nonnegative integers $N$ and $M$, with constants independent
of $R$, $L_1$, and $L_2$.
 \section{Reduction to a localized oscillatory integral}\label{sec:prelim}
In this section, we carry out preparatory reductions toward the proof of  Theorem \ref{thm:decay}. The objective is to isolate the essential part, which will be addressed in the remaining sections.  We begin by recalling the steps of reduction from \cite{JLR24b}. 

We first fix \(0<\varepsilon_0\ll1\), and then choose
\(C_\ast\) sufficiently large, depending only on
\(\varepsilon_0\). These constants remain fixed throughout the
proof. All implicit constants are uniform in
\(\lambda,\mu,\mu'\), although they may depend on
\(\varepsilon_0\) and \(C_\ast\).

\subsection{Further localization via a Whitney-type decomposition}
Under the hypotheses of Theorem~\ref{thm:decay}, the Koch--Tataru
estimate \eqref{e:KTlocal} shows that \eqref{cap} already holds when
\(\mu'>\epz\mu\); see \cite[Sections~2--3]{JLR24b}. Thus, taking
\(c_0=\epz\), it suffices to assume that
\Be
\label{mm}
C_\ast\lambda^{-\frac{2}{3}}< \mu'\le \epz  \mu,    \quad \mu  \le \epz
 \Ee
where \(\epz\) and \(C_\ast\) are fixed as above. Let
\[  \cP(x,y,t) = \frac{t}{2} + \frac{(|x|^2+|y|^2)\cos t-2\inp xy}{2\sin t}. \]
By the kernel representation of the projection operator $\Pi_\lambda$ (see \cite[Lemma 2.1]{JLR24b}) and the scaling $(x,y)\to \sqrt\lambda(x,y)$, the estimate \eqref{cap} 
is equivalent to 
 \begin{align}
    \label{cap0}
         \|\chi_{A_{\mu}}\mathcal O_\lambda\chi_{A_{\mu'}}\|_{q'_e \to q_e }\le C\lambda^{-\frac{4}{5}}\biggl(\frac{\mu'}{\mu}\biggr)^{\frac{1}{20}},
    \end{align}
where $A_{\mu}$ denotes $A_{1,\mu}$, and the operator $\mathcal O_\lambda$ is given by its kernel 
\[   \mathcal O_\lambda(x,y)= \int_{-\pi}^\pi \frac{e^{i\lambda \cP(x,y,t)}}{\sin t}  dt. \]
Indeed, in dimension two the dilation contributes
\(\lambda^{1+1/q_e-1/q_e'}=\lambda^{3/5}\) to the operator norm, so
the factor \(\lambda^{-4/5}\) in \eqref{cap0} corresponds to
\(\lambda^{-1/5}\) in \eqref{cap}.
 
Throughout this paper, to simplify notation, we identify an operator with its kernel as long as no confusion can arise. Thus, for an operator \(T\), we write \(T(x,y)\) for its kernel, so that
\[
    Tf(x)=\int T(x,y)f(y)\,dy.
\]
Conversely, when a kernel \(T(x,y)\) defines the associated operator $T$ through the right-hand side, we denote this operator again by \(T\).

Let $\psi\in C_c^\infty((1/2, 2))$ such that  $\sum_{j\in \Z}\psi(2^jt) = 1$ for $t>0$.
We briefly review the reduction steps in \cite[Sections 2--3]{JLR24b}. These reductions are dimension-independent and apply without change
in the present two-dimensional setting. Indeed, after dyadically decomposing in $t$ away from the singularities $\pm \pi,0$ and exploiting the symmetry of the phase function, it is sufficient to prove \eqref{cap0} with
 $\mathcal O_\lambda(x,y)$ replaced by
\[   \int \psi(2^jt)\,\frac{e^{i\lambda \cP(x,y,t)}}{\sin t}  dt\] 
for $2^{-j}\sim \mu^{1/2}$ after discarding the acceptable minor contributions.  
In what follows, by $\psi_\mu$ we denote a smooth function such that 
 \[\supp \psi_{\mu} \subset I_\mu:= [ C_\ast^{-1} \smu, C_\ast\smu ]\]
 for a large constant $C_\ast>0$ and $(d/dt)^N\psi_{\mu}
    =O(\mu^{-N/2})$ for any nonnegative integer $N$.  Thus, we may replace the above integral with 
    \begin{align}\label{d:Plapsij}
      \mathcal O_\la^\mu (x,y) = \int \psi_\mu(t)\,\frac{e^{i\lambda \cP(x,y,t)}}{\smu}  dt. \footnotemark 
    \end{align}
 \footnotetext{More precisely, $\psi_\mu= \sqrt \mu \psi(2^j t)/\sin t$ with $2^{-j}\sim \smu$.} 
 

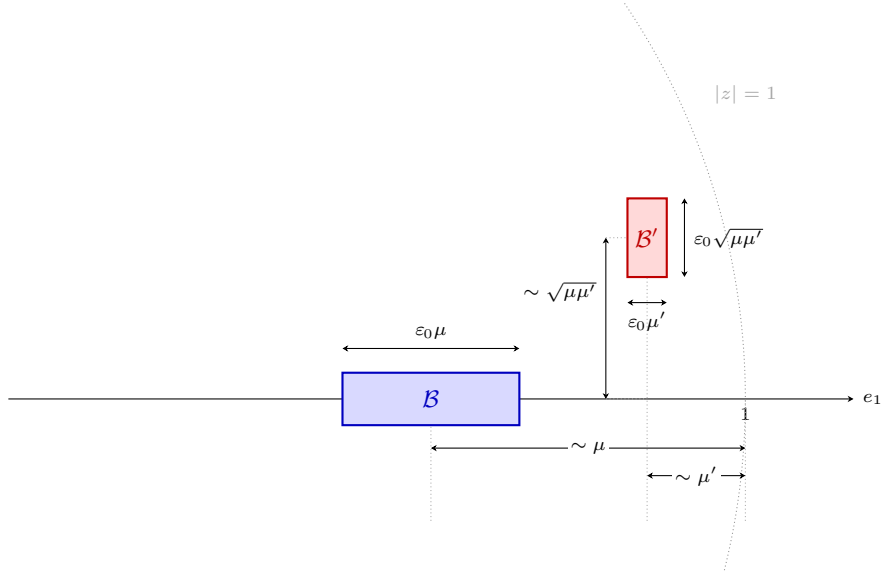
\begin{figure}[t]
\centering
\begin{tikzpicture}[
    >={Stealth[length=2.2pt,width=2.4pt]},
    scale=2.6,
    dim/.style={<->,thin},
    guide/.style={gray!65,densely dotted},
    every node/.style={font=\scriptsize}
]

\draw[->] (-0.15,0) -- (4.15,0)
    node[right] {$e_1$};

\draw[gray!85,densely dotted]
    plot[domain=-14:34,samples=80,variable=\t]
    ({3.6*cos(\t)},{3.6*sin(\t)});
\node[gray!70] at (3.60,1.55) {$|z|=1$};
\node[below] at (3.6,0) {$1$};

\filldraw[
    fill=blue!15,
    draw=blue!75!black,
    thick
]
    (1.55,-0.133333) rectangle (2.45,0.133333);

\node[blue!75!black,font=\small] at (2.00,0)
    {$\mathcal B$};

\filldraw[
    fill=red!15,
    draw=red!75!black,
    thick
]
    (3.00,0.62) rectangle (3.20,1.02);

\node[red!75!black,font=\small] at (3.10,0.82)
    {$\mathcal B'$};

\draw[guide] (2.00,-0.62) -- (2.00,-0.133333);
\draw[guide] (3.10,-0.62) -- (3.10,0.62);
\draw[guide] (3.60,-0.62) -- (3.60,0);

\draw[dim] (2.00,-0.25) -- (3.60,-0.25)
    node[midway,fill=white,inner sep=1pt]
    {$\sim\mu$};

\draw[dim] (3.10,-0.39) -- (3.60,-0.39)
    node[midway,fill=white,inner sep=1pt]
    {$\sim\mu'$};

\draw[guide] (2.89,0) -- (3.10,0);
\draw[guide] (2.89,0.82) -- (3.00,0.82);

\draw[dim] (2.89,0) -- (2.89,0.82)
    node[pos=2/3,left]
    {$\sim\sqrt{\mu\mu'}$};

\draw[dim] (1.55,0.256667) -- (2.45,0.256667)
    node[midway,above] {$\epz\mu$};

\draw[dim] (3.00,0.49) -- (3.20,0.49)
    node[midway,below] {$\epz\mu'$};

\draw[dim] (3.29,0.62) -- (3.29,1.02)
    node[midway,right]
    {$\epz \sqrt{\mu\mu'}$};

\end{tikzpicture}

\caption{A typical configuration of $\mathcal B$ and $\mathcal B'$
when $\mu'\ll\mu$. The figure is schematic and not drawn to scale.}
\label{fig:typical-B-Bprime}
\end{figure} 

For a sufficiently small number $\epz>0$, let $\cB$ and $\cB'$ be rectangles of dimensions $\epz\mu\times \epz(\mu\mu')^{1/2}$ and $\epz\mu'\times \epz(\mu\mu')^{1/2}$, whose sides are in the directions of $x_1$, $x_2$, respectively, such that 
\Be\begin{aligned}
\label{bb}
\mathcal B &\subset\, \big\{x: 1-x_1\sim \mu, \,|x_2| \ll   \mmt   \big \}, \\  \mathcal B' &\subset\, \big\{y: 1-y_1\sim {\mu'}, \,  y_2 \sim   \mmt\big\}. 
\end{aligned}
\Ee
Let $\wt\chi_{\cB}$ and $\wt\chi_{\cB'}$ be nonnegative smooth cutoff functions, adapted to $\cB$ and $\cB'$, respectively, such that 
\Be\label{suppBB}  \supp\wt\chi_{\cB} \Subset \cB, \quad \supp\wt\chi_{\cB'} \Subset \cB',\Ee
$\partial_1^k \partial_2^l\wt \chi_{\cB}=O(\mu^{-k} \times (\mu\mu')^{-l/2})$,  and $\partial_1^k \partial_2^l \wt\chi_{\cB'}=O((\mu')^{-k} \times (\mu\mu')^{-l/2})$. 

Now, after a sectorial Whitney-type decomposition of $A_\mu\times A_{\mu'}$ and a simultaneous orthogonal transformation (a rotation and a reflection), and discarding acceptable contributions again, it is enough to prove the estimate 
 \Be \label{main-est} \|\wt\chi_{\cB}  \mathcal O_\la^\mu \wt\chi_{\cB'}\|_{q_e '\to q_e }\lesssim \lambda^{-\frac{4}{5}}\biggl(\frac{\mu'}{\mu}\biggr)^{\frac{1}{20}},\Ee
under the assumption that  
\begin{align}
 \label{r:cond-cDh}
    &|\cD(x,y)|\le C_0\epz\mu\mu',\quad (x,y)\in \cB\times \cB'
\end{align}
where \(C_0>0\) is fixed. The complementary pieces with
\(|\cD|\gtrsim\epz\mu\mu'\) are controlled by the same kernel
estimates, in particular \cite[(3.3), p.~1329]{JLR24b}. Thus, we are
reduced to the following theorem.

\begin{thm}\label{thm:core}
Let \(\mu,\mu'\in2^{-\N}\) satisfy \eqref{mm}. Suppose that
\(\cB,\cB'\) and the adapted cutoffs
\(\wt\chi_{\cB},\wt\chi_{\cB'}\) are as above and that
\eqref{r:cond-cDh} holds. Then \eqref{main-est} holds uniformly over
cutoffs with fixed adaptation constants.
\end{thm}

The remainder of this paper is devoted to the proof of Theorem \ref{thm:core}.  
Throughout the paper, we always assume that 
\[  (x,y)\in \cB\times \cB'\]
even if it is not made explicit.  


\subsection{Decomposition away from the symmetric point} \label{sec:prune}
In this section, we single out the main part before carrying out the decomposition in the spatial variables in Section \ref{sec:decomp}. As explained in the introduction, we decompose the integral in $t$ away from the symmetric point $S(x,y)$, with the bottom scale comparable to $(\mup)^{1/2}$.  
This will allow us to localize tightly in $t$ in the integral defining the kernel. Moreover, this will serve as a gentle introduction to the more refined orthogonality argument that appears later in the paper.

Recalling \eqref{d:Plapsij}, we collect some properties of the phase function $\cP$. From \eqref{bb} and \eqref{mm}, note that 
\Be
\label{inp} 1-\inp xy\sim \mu, \quad   \forall (x,y)\in \cB\times\cB'.\Ee
Thus, after enlarging $C_\ast$ if necessary, we have a smooth map $S : \cB\times \cB' \to I_\mu$  
 such that 
\[ 
 \cos S (x,y) = \inp xy.
 \]
 A computation shows that  
\begin{align}\label{i:ptcP-Q}
    \partial_t \cP(x,y,t) = \frac{ \cD(x,y)-(\cos t - \cos S (x,y))^2}{2\sin^2 t}. 
\end{align}
Thanks to \eqref{r:cond-cDh}, if $|t-S(x,y)|$ is large enough, $\partial_t \cP$ has a good lower bound. 
  More precisely,  by \eqref{i:ptcP-Q} we have
\begin{align}\label{e:simpleobv}
    |\partial_t\cP(x,y,t)|\gtrsim \mu^{-1} |\cos t-\cos S (x,y)|^2 
\end{align}
 for $t\in I_\mu$ satisfying 
 \Be
\label{tsxy}
|t-S (x,y)|\ge 2C_\ast{\mu^{-\frac{1}{2}}}|\cD(x,y)|^{\frac{1}{2}}.\footnote{This is easy to see from the mean value theorem and the support condition of $\psi_\mu$.} 
\Ee  
Thus, if \eqref{tsxy} holds, we have a good lower bound \eqref{e:simpleobv} as long as $|\cD(x,y)|$ is controlled appropriately.

\newcommand{\smup}{\sqrt\mup}

Let us set
\begin{align}\label{d:top}
 \delt = 2^{\lfloor\log_2(C_0\epz\mu\mu')\rfloor}.
\end{align}
Hence, $| \cD(x,y)|\le 2\delt $ for $(x,y)\in \cB\times \cB'$.   Let $\eta\in C^\infty_c((-2, 2))$ satisfy $0\le\eta\le1$ and $\eta (t)=1$ for $|t|\le 1$.  Thus, $1-\eta(t)=0$ if $|t|\le 1 $.  
  We set 
\begin{align}
\label{psi-tri}
 \psi_{\vartriangle}(x,y,t)&=   \psi_\mu(t)   \eta\Big(\frac{t-S (x,y)}{2^2 C_\ast{\delt}^{\frac{1}{2}}\mu^{-\frac{1}{2}}}\Big) , 
 \\
 \label{psi-tri0}
  \wt\psi_{\vartriangle}(x,y,t)&=   \psi_\mu(t)  \Big( 1- \eta\Big(\frac{t-S (x,y)}{2^2 C_\ast{\delt}^{\frac{1}{2}}\mu^{-\frac{1}{2}}}\Big)  \Big),
 \end{align}
 which decompose the integral into the parts near and away from the symmetric point. 

To simplify notation, for a given function $h$ supported in $\cB\times\cB'\times I_\mu$, we denote 
\Be 
\label{ola}
  \mathcal O_\lambda[h](x,y)=  \int h(x,y,t) \,\frac{e^{i\lambda \cP(x,y,t)}}{\smu}   dt .
  \Ee
Thus, recalling \eqref{d:Plapsij},  we have 
\begin{align}\label{e:decompm-1st} 
\mathcal O_\lambda^\mu =   \mathcal O_\lambda[\psi_\vartriangle]+  \mathcal O_\lambda[\widetilde\psi_\vartriangle]. 
\end{align}

Since $| \cD(x,y)|\le 2\delt $ for $(x,y)\in \cB\times \cB'$, \eqref{tsxy} is satisfied provided that $\wt\psi_{\vartriangle}(x,y,t)\neq 0$ for  $(x,y)\in \cB\times \cB'$. 
So, we have \eqref{e:simpleobv}, which allows us to handle the contribution from $\mathcal O_\lambda[\widetilde\psi_\vartriangle]$ in a relatively easier manner. 
In fact, we have the following.

    \begin{prop}\label{prop:estfE} Let $\widetilde\psi_\vartriangle$ be given by \eqref{psi-tri0}. Then, we have 
    \begin{align}\label{e:sumfElk}
     \Big\| \wt\chi_{\cB} \mathcal O_\lambda[\widetilde\psi_\vartriangle]\wt\chi_{\cB'}\Big\|_{L^{q_e'}\to L^{q_e}} \lesssim \lambda^{-\frac{4}{5}}\biggl(\frac{\mu'}{\mu}\biggr)^{\frac{1}{20}}.
\end{align}
\end{prop}

Proposition~\ref{prop:estfE} treats the contribution away from the
symmetric time. The central estimate is the corresponding bound for
the near-time operator, stated as Proposition~\ref{main-p} at the
beginning of Section~\ref{sec:geometry}. Together, these two
propositions prove Theorem~\ref{thm:core}.  Obtaining the desired bound for $\mathcal O_\lambda[\psi_\vartriangle]$ (see \eqref{e:main-fP}) is considerably more involved. We carry it out through the remaining sections.

To prove Proposition \ref{prop:estfE}, we make an additional decomposition 
\Be
\label{psi-tri-l} \widetilde\psi_\vartriangle^\ell(x,y,t)= \widetilde\psi_\vartriangle(x,y,t)  \psi(2^\ell|t-S (x,y)|).
\Ee
Hence, $ \mathcal O_\lambda[\widetilde\psi_\vartriangle]=\sum_{\ell }  \mathcal O_\lambda[\widetilde\psi_\vartriangle^\ell]$.  
Note that $S(x,y)\in I_\mu$ and $\supp \widetilde\psi_\vartriangle(x,y,\cdot)\subset I_\mu$ for $(x,y)\in \cB\times\cB'$. Since $\delt\sim \mu\mu'$, it is clear that $\widetilde\psi_\vartriangle^\ell\neq 0$ 
only if 
\Be \label{l-range}  \sqrt{\mu'} \lesssim 2^{-\ell} \lesssim \sqrt\mu.\Ee

The proof of \eqref{e:sumfElk} is rather straightforward once we obtain the estimates
    \begin{align}
   \label{l0}  \| \mathcal O_\lambda[\widetilde\psi_\vartriangle^\ell] \|_{L^1\to L^\infty}&\lesssim \lambda^{-\frac{1}{2}}2^{\frac{\ell}{2}}\mu^{-\frac{1}{2}},
        \\[2pt]
             \label{l2}    \| \mathcal O_\lambda[\widetilde\psi_\vartriangle^\ell] \|_{L^2\to L^2}&\lesssim \lambda^{-1}2^{\frac{3\ell}{2}}\mu^{\frac{1}{4}}\mu'.
    \end{align}
These estimates hold provided that $2^{-\ell}\ge C_1 \smp$. Indeed, interpolation gives 
\begin{align*}
 \Big\|  \mathcal O_\lambda[\widetilde\psi_\vartriangle^\ell] \Big\|_{L^{q_e'} \to L^{q_e}}\lesssim \lambda^{-\frac{4}{5}}2^{\frac{11}{10}\ell} \mu^{-\frac{1}{20}}(\mu')^{\frac{3}{5}}. 
\end{align*}
Summation over $\ell: 2^{-\ell}\gtrsim \smp $  yields    \eqref{e:sumfElk}.

The estimate \eqref{l0} follows from the oscillatory integral estimate  
\Be
\label{osc-est}    \Big| \int   
\widetilde\psi_\vartriangle^\ell(x,y,t)
e^{i\lambda \cP(x,y,t)}    dt\Big|\lesssim    \lambda^{-1} 2^{2\ell}.   
\Ee 
Indeed, since the integral is $O(2^{-\ell})$, taking geometric mean gives \eqref{l0}. To see \eqref{osc-est},  we use the 
van der Corput  lemma (\cite[Corollary in p.334]{St93}).  To this end, recall that  \eqref{e:simpleobv} holds whenever $\wt\psi_{\vartriangle}(x,y,t)\neq 0$ for  $(x,y)\in \cB\times \cB'$. Thus, from this and \eqref{psi-tri-l}, we have 
\begin{align}\label{e:lbptcP-E*}
    |\partial_t \cP(x,y,t)| & \gtrsim 2^{-2\ell}. 
\end{align}
provided that $\widetilde\psi_\vartriangle^\ell(x,y,t)\neq 0$.  A computation gives
   \[
        \partial_t^2\cP(x,y,t) = -\inp xy \frac{\cR(x,y,\cos t)}{\sin^3 t}
   \]
    where $\cR(x,y,\tau) = \tau^2-\inp xy^{-1}(|x|^2+|y|^2)\tau + 1$.  Since $\cR(x,y,\cdot)$ is quadratic,  $\partial_t^2\cP(x,y,\dot)$ changes sign at most twice on $(0,\pi/2)$ (in fact, it changes sign only once). Thus, the support of $\widetilde\psi_\vartriangle^\ell(x,y,\cdot)$ is included in a union of two disjoint intervals where $\partial_t\cP(x,y,\dot)$ is monotone. Since 
     $\sup_{x,y,t}  |\widetilde\psi_\vartriangle^\ell(x,y,t)|+ \sup_{x,y} \| \partial_t \widetilde\psi_\vartriangle^\ell(x,y,\cdot)\|_1\lesssim 1$, combining \eqref{e:lbptcP-E*} and   the van der Corput  lemma, we obtain \eqref{osc-est} as desired.

\medskip

The rest of this section is devoted to the proof of the estimate  \eqref{l2}.

\subsection{Contribution away from the symmetric point}\label{subsec:away-contribution} We start by making a further decomposition of the operator $\wt\chi_{\cB} \mathcal O_\lambda[\widetilde\psi_\vartriangle^\ell] \wt\chi_{\cB'}$ on the input side at 
scale $2^{-\ell}(\mu')^{1/2}$.  

Let $\zeta\in C_c^\infty(-1,1)$ be a nonnegative function such that $\sum_{n} \zeta(\cdot-n)=1$. 
Fixing $\ell$ satisfying \eqref{l-range},  we set 
\Be 
\label{en}    \mathcal E_n (x,y)= \wt\chi_{\cB}(x)\, \mathcal O_\lambda[\widetilde\psi_\vartriangle^\ell](x,y)\, (\wt\chi_{\cB'})_n(y)
\Ee
for $n\in \mathbb Z$,  where 
\Be 
\label{Bn} (\wt\chi_{\cB'})_n(y)=   \wt\chi_{\cB'}(y)   \zeta\Big( \frac{y_2- 2^{-\ell}(\mu')^{\frac{1}{2}}n}{2^{-\ell}(\mu')^{\frac{1}{2}}}  \Big).
\Ee 
Consequently, we have 
\[     \wt\chi_{\cB} \mathcal O_\lambda[\widetilde\psi_\vartriangle^\ell]  \wt\chi_{\cB'}=  \sum_n   \mathcal E_n .\]
There are only finitely many nonzero summands. In fact, $\mathcal E_n\neq 0$
only if
\[
|2^{-\ell}(\mu')^{\frac{1}{2}}n|\lesssim (\mu\mu')^{\frac{1}{2}}.
\]

In order to obtain bounds for $\sum_n \mathcal E_n$, we make use of almost orthogonality. More precisely, we utilize     
 the Cotlar--Stein lemma in the form of the following lemma. See \cite[Lemma 18.6.5]{H83}, where 
the case $M=M_\ast$ was proved. The stated two-space form follows from the usual Cotlar--Stein lemma by applying it on \(H_1\oplus H_2\) to the corresponding off-diagonal block operators.

\begin{lem} \label{lem:cotlarstein} Let $H_1$, $H_2$ be Hilbert spaces.
    Let $\{T_j\}_{j=1}^N$ be a family of bounded operators from $H_1$ to $H_2$. Let $T_k^\ast$ denote the adjoint operator of $T_k$.  Suppose that 
    \[   \sup_j \sum_{k=1}^N \|T_j^*T_k\|^{\frac{1}{2}}_{H_1\to H_1}\le M, \quad    \sup_j \sum_{k=1}^N \|T_jT_k^*\|^{\frac{1}{2}}_{H_2\to H_2}\le M_\ast  \]
     Then, we have $
        \big\|\sum_{j=1}^N T_j\big\|_{H_1\to H_2}\le \sqrt{MM_\ast}.
    $
\end{lem}

  In what follows, we also use the next lemma, which is a simple consequence of 
Schur's test or Riesz--Thorin interpolation. 

\begin{lem}
\label{schur}  Let $A, A'\subset \R^2$ and $L>0$.  Suppose $|T(x,y)|\le L \chi_{A\times A'}(x, y)$. Then, $\|T\|_{2\to 2}\le L(|A||A'|)^{1/2}$. 
\end{lem}

We first obtain a bound on $\|\mathcal E_n\|_{2\to 2}$. From the estimate \eqref{osc-est}, it follows that $|\mathcal E_n(x,y)|\lesssim \lambda^{-1} 2^{2\ell} \mu^{-1/2}\chi_{\cB}(x)  (\wt\chi_{\cB'})_n(y)$. Thus, Lemma \ref{schur} gives 
\Be\label{en-} \|\mathcal E_n \|_{L^2\to L^2}\lesssim \lambda^{-1}2^{\frac{3\ell}{2}}\mu^{\frac{1}{4}}\mu'. \Ee
Note $\mathcal E_j \mathcal E_k^*=0$ if $|j-k|\ge 2$. Since $\|\mathcal E_j\mathcal E_k^*\|_{2\to 2}\le \|\mathcal E_j\|_{2\to 2}\|\mathcal E_k\|_{2\to 2},$ it follows that 
\[ \textstyle  \sup_j \sum_{k=1}^N \|\mathcal E_j\mathcal E_k^*\|^{\frac{1}{2}}_{2\to 2}\lesssim  \lambda^{-1}2^{\frac{3\ell}{2}}\mu^{\frac{1}{4}}\mu'.\] 

Therefore, to prove  \eqref{l2} we also have  to show
\Be 
\label{diff}\textstyle \sup_j \sum_{k=1}^N \|\mathcal E_j^*\mathcal E_k\|^{\frac{1}{2}}_{2\to 2}\lesssim  \lambda^{-1}2^{\frac{3\ell}{2}}\mu^{\frac{1}{4}}\mu', \Ee
since Lemma \ref{lem:cotlarstein}, combined with these estimates, yields \eqref{l2}.  
On the other hand, \eqref{diff} is a consequence of the following lemma.   

\begin{lem}\label{prop:alortho-easy}  There is a constant $C$ such that
    \begin{align}
    \label{ejek}
      \|\mathcal E_j^*\mathcal E_{k}\|_{L^2\to L^2} \lesssim_N  2^{-3\ell}\mu^{\frac{1}{2}}(\mu')^2 |\la 2^{-\ell}\mu'(j-k)|^{-N}
    \end{align}
for any $N\in \mathbb N$ provided that $|j-k|\ge C$. 
\end{lem}

Indeed, after splitting the sum into two cases $|j-k|\ge C_2$ and $|j-k|< C_2$, we apply the above estimate with $N=6$ and \eqref{en-} to get 
\begin{align*}
  \textstyle   \sum_{k=1}^N \|\mathcal E_j^*\mathcal E_k\|^{\frac{1}{2}}_{2\to 2} \lesssim  ( \lambda^{-2} \mu'^{-3}+1) \lambda^{-1}2^{\frac{3\ell}{2}}\mu^{\frac{1}{4}} \mu'. 
\end{align*}
This gives the desired  estimate  \eqref{diff}, since  $\lambda^{-2/3}\lesssim \mu'$.

To complete the proof of \eqref{l2}, it remains to show Lemma \ref{prop:alortho-easy}.

\subsection{Proof of Lemma \ref{prop:alortho-easy}}\label{subsec:away-orthogonality} We prove Lemma \ref{prop:alortho-easy} by obtaining the estimate for the kernel of $\mathcal E_j^*\mathcal E_{k}$, which is given as an oscillatory integral. To this end, we need a couple of lemmas. 

\begin{lem}\label{lem:lbdiffcP-easy}  
    Let $x\in \cB$, $y,w\in \cB'$, and $2^{-\ell}\gtrsim \smp$. Assume that  $t\in \supp \widetilde\psi_\vartriangle^\ell(x,y,\cdot)$ and $s\in \supp \widetilde\psi_\vartriangle^\ell(x,w,\cdot)$. 
   Then, there are constants  $C$ and $C'$ such that 
   \Be
   \label{px22}
        \big|\partial_{x_2}\big(\cP(x,y,t) - \cP(x,w,s)\big)\big| \ge C\mu^{-\frac{1}{2}}(|y_2-w_2| - C'2^{-\ell}\smup)
   \Ee
    \end{lem}

\begin{proof} First observe that 
\Be \label{ss}   |S(x,y)- S(x,w)|  \le C \mu' \mu^{-\frac{1}{2}}\Ee
for $x\in \cB$ and $y, w\in \cB'$. To see this, note $\cos S(x,y)-\cos S(x,w)= \inp x{y-w}$. 
Since $x\in \cB$ and $y, w\in \cB'$, $\inp x{y-w}=O(\mu')$ by \eqref{bb}. Moreover, since $S(x,y), S(x,w) \sim \mu^{1/2}$, 
the mean value theorem gives  \eqref{ss}. 

Now, we  note that
   \Be
   \label{px2}
        \partial_{x_2}\big(\cP(x,y,t) - \cP(x,w,s)\big) = x_2\Big(\frac{\cos t}{\sin t} - \frac{\cos s}{\sin s}\Big) + \frac{w_2}{\sin s} - \frac{y_2}{\sin t}.
    \Ee
   Via an elementary trigonometric identity, the right-hand side is rewritten as  $I+ I\!I$, where 
    \begin{align*}
     I :=   \frac{w_2-y_2}{\sin s}, \qquad  I\!I:= \frac{x_2\sin(s-t) + y_2(\sin t - \sin s)}{\sin t \sin s}.
    \end{align*}
Since $|t- S(x,y)|\lesssim 2^{-\ell}$ and $|s- S(x,w)|\lesssim 2^{-\ell}$, using \eqref{ss}, we see   
    $|s-t|, |\sin s-\sin t|\lesssim  2^{-\ell} + \mu' \mu^{-1/2} $. From \eqref{bb} we also have $|x_2|,\ |y_2|\lesssim (\mu\mu')^{1/2}$.
   Since $t, s\sim \mu^{1/2}$ and $2^{-\ell}\gtrsim \smp$,  it follows that 
    \[  |I\!I|\lesssim  (2^{-\ell} + \mu' \mu^{-\frac{1}{2}}) \biggl(\frac{\mu'}{\mu}\biggr)^{\frac{1}{2}}\lesssim 2^{-\ell}  \biggl(\frac{\mu'}{\mu}\biggr)^{\frac{1}{2}}. \] 
   On the other hand,  we have $|I|\gtrsim \mu^{-1/2} |y_2- w_2|$. Therefore, the inequality \eqref{px22} follows. 
   \end{proof}

In order to prove Lemma \ref{prop:alortho-easy}, we need to control the derivatives $\partial_{x_2}^\beta S$. For later use, we do this in a more general setting than what is needed here.

\begin{lem}\label{lem:bdsS*}
    Let $(x,y)\in \cB\times\cB'$, and $\nu,\nu'\in \mathbb S^1$. Suppose  that 
    \[ |\inp{\nu}{e_1}|,\ |\inp{\nu'}{e_2}|\le c\biggl(\frac{\mu'}{\mu}\biggr)^{\frac{1}{2}}\] for some constant $c>0$. Then, for any $N\in\N$,  we have
    \begin{align}
    \label{nu}
        |\inp{\nu}{\partial_x}^N S (x,y)|& \lesssim_N  \mu^{\frac{1}{2}}\biggl(\frac{\mu^{\frac{3}{2}}}{\smp}\biggr)^{-N},
        \\
   \label{nup}      |\inp{\nu'}{\partial_y}^N S (x,y)|& \lesssim_N  \mu^{\frac{1}{2}-N}.
    \end{align}
\end{lem}

\begin{proof} 
    We first show \eqref{nu}  by induction on $N$. 
    Since  $\cos S (x,y)=\inp xy$, note that 
    \begin{align}\label{i:vpxS*}
       \inp\nu y= \inp{\nu}{\partial_x}\cos S (x,y) = -\sin S (x,y) \inp{\nu}{\partial_x}S (x,y).
    \end{align}
    Since $y\in \cB'$, $|y_2|\lesssim (\mu\mu')^{1/2}$. Thus, by the assumption on $\nu$, 
    $ \inp\nu y=  O((\mu'/\mu)^{1/2})$. The first equality gives 
    \begin{align}\label{e:pxcosS*}
        |\inp{\nu}{\partial_x}^N\cos S (x,y)|\lesssim \begin{cases}
            \bigl(\frac{\mu'}{\mu}\bigr)^{\frac{1}{2}} & N=1 \\
           \quad\!\! 0 & N\ge 2.
        \end{cases}
    \end{align} 
    Combining this with \eqref{i:vpxS*}, we have  \eqref{nu} for $N=1$. 
        We also used $\sin S(x,y) \sim \mu^{1/2}$. 
    
    Assume that the  bound \eqref{nu} holds for $1\le N\le k$. Since 
   $
    \inp{\nu}{\partial_x}^{k+1}\cos S (x,y) = 0,  
    $
    applying $\inp{\nu}{\partial_x}^k$ to both sides of the second equality in \eqref{i:vpxS*} gives 
    \begin{align}\label{i:pxS*-E*k}
        \sin S (x,y) \inp{\nu}{\partial_x}^{k+1}S (x,y) + E_k(x,y) = 0
    \end{align}
    where 
    \begin{align}\label{i:defEk*}
        E_k(x,y) = \sum_{j=1}^k C_j \inp{\nu}{\partial_x}^j (\sin S (x,y)) \inp{\nu}{\partial_x}^{k+1-j}S (x,y)
    \end{align}
    with suitable constants $C_j$. Since $\inp{\nu}{\partial_x} \sin S (x,y) = \inp{\nu}{\partial_x}S (x,y)\cos S (x,y)$, by our induction assumption and \eqref{e:pxcosS*} we have 
    \begin{align*}
        |\inp{\nu}{\partial_x}^j \sin S (x,y)|\lesssim_j \mu^{\frac{1}{2}}(\mu^{-\frac{3}{2}}\smp)^j
    \end{align*}
    for $1\le j\le k$. Substituting this into \eqref{i:defEk*} and using the induction hypothesis yields $|E_k(x,y)|\lesssim \mu (\mu^{-3/2}\smp)^{k+1}$. From this and \eqref{i:pxS*-E*k}, \eqref{nu} with $N=k+1$ follows, since $\sin S(x,y) \sim \mu^{1/2}$.

    We show the bounds on $\inp{\nu'}{\partial_y}^N S$, whose proof is almost identical to the previous one. Since $|\inp{\nu'}{e_2}|\le c(\mu'/\mu)^{1/2}$, 
    we have $\inp{\nu'}x=O(1)$ for $x\in \cB$. By the same argument, we obtain
    \begin{align*}
        |\inp{\nu'}{\partial_y}^N\cos S (x,y)|\lesssim \begin{cases}
            1 & N=1 ,\\
            0 & N\ge 2.
        \end{cases}
    \end{align*}
   One immediately obtains $|\inp{\nu'}{\partial_y}S (x,y)|\lesssim \mu^{-1/2}$, using  the identity
   \[
        \inp{\nu'}{\partial_y}\cos S (x,y) = -\sin S (x,y) \inp{\nu'}{\partial_y}S (x,y).\]
This gives \eqref{nup} for $N=1$.  
    For general $N$, we argue by induction as before. We omit the details.
\end{proof}

We proceed to prove Lemma \ref{prop:alortho-easy}.

\begin{proof}[Proof of Lemma \ref{prop:alortho-easy}]
    Assume that $y\in  \supp (\wt\chi_{\cB'})_j$, $w\in \supp (\wt\chi_{\cB'})_k$. Note
    \begin{align*}
        \mathcal E_j^*\mathcal E_k(y,w)  = \frac{ (\wt\chi_{\cB'})_j(y)(\wt\chi_{\cB'})_k(w)}{\mu}\iiint e^{-i\lambda\Psi(x,y,w,t,s) }     A(x,y,w,t,s) dxdsdt,
    \end{align*}
    where
    \begin{align*}
    \Psi(x,y,w,t,s)&=\cP(x,y,t)-\cP(x,w,s),
    \\
        A (x,y,w,t,s) &=  \big(\wt\chi_\cB(x)\big)^2  \widetilde\psi_\vartriangle^\ell(x,y,t)  \widetilde\psi_\vartriangle^\ell(x,w,s).
    \end{align*}
    
    By Lemma \ref{lem:bdsS*}, $\partial_{x_2}^N S (x,y)$ and $\partial_{x_2}^N S (x,w)$ are $O(\mu^{1/2}(\mu^{-3/2}\smp)^N)$.  Recalling \eqref{psi-tri} and \eqref{psi-tri-l},  note that
    \begin{align*}
        \big|\partial_{x_2}^N\big(  \widetilde\psi_\vartriangle^\ell(x,y,t)  \widetilde\psi_\vartriangle^\ell(x,w,s)\big)\big|\lesssim_N \mu^{-N}
    \end{align*}
    for all $N\in \N$, because $2^\ell\lesssim (\mu')^{-1/2}$ and $\mu'\ll \mu$.  
    By the adaptedness of $\wt\chi_\cB$,
    $|\partial_{x_2}^N\wt\chi_\cB(x)|\lesssim_N (\mu\mu')^{-N/2}$. Combining  these two estimates, we obtain
    \begin{align}\label{e:px2A*}
        |\partial_{x_2}^N A(x,y,w,t,s)|\lesssim_N (\mu\mu')^{-\frac{N}{2}}.
    \end{align}

   On the other hand, it follows from \eqref{px2} that $\partial_{x_2}^2 \Psi(x,y,w,t,s) = \sin(s-t)/ (\sin s \sin t)$. 
   Since $|s-t| \lesssim 2^{-\ell}$ (see the proof of Lemma \ref{lem:lbdiffcP-easy}), we have 
    \begin{align}\label{e:px2cPdiff}
       \partial_{x_2}^N  \Psi(x,y,w,t,s) = \begin{cases}
      O( 2^{-\ell}\mu^{-1}) , \quad &  N=2,
       \\ 
     \qquad  0, \quad & N\ge 3. 
     \end{cases}
           \end{align} 
               
    Now, we consider a differential operator $\mathcal L$ given by
    \begin{align*}
        \mathcal L f(x)
        =\partial_{x_2}\Big(
        \frac{f(x)}{-i\lambda\partial_{x_2}\Psi(x,y,w,t,s)}
        \Big).
    \end{align*}
    Let $\mathcal L^\ast$ denote the adjoint of $\mathcal L$. Since
    $\mathcal L^\ast e^{-i\lambda\Psi(x,y,w,t,s)}
    =e^{-i\lambda\Psi(x,y,w,t,s)}$, by integration by parts, for any
    $N\in\N$ we have
    \begin{align*}
        |\mathcal E_j^*\mathcal E_k(y,w)|
        \le \frac{(\wt\chi_{\cB'})_j(y)(\wt\chi_{\cB'})_k(w)}{\mu}
        \bigg|\iiint e^{-i\lambda\Psi(x,y,w,t,s)}
        \mathcal L^N\!A(x,y,w,t,s)\,dx\,ds\,dt\bigg|.
    \end{align*}

   Recalling \eqref{Bn}, we note   that  $|y_2-w_2|\sim 2^{-\ell}(\mu')^{1/2} |j-k|$  
    whenever  $y\in \supp (\wt\chi_{\cB'})_j$ and  $w\in \supp (\wt\chi_{\cB'})_k$ and $|j-k|\ge 3$.  Using Lemma \ref{lem:lbdiffcP-easy} we have the lower bound 
    \[|\partial_{x_2}\Psi(x,y,w,t,s)|\gtrsim  \mu^{-\frac{1}{2}}  2^{-\ell}(\mu')^{\frac{1}{2}} |j-k|\] 
  if    $|j-k|\ge C$ for a sufficiently large $C>0$. 
 Putting this  together with \eqref{e:px2A*} and \eqref{e:px2cPdiff}, via a routine computation we get
    \begin{align*}
        \big|\mathcal L^N \! A(x,y,w,t,s)\big|
        \lesssim_N |\lambda 2^{-\ell}\mu'(j-k)|^{-N}
    \end{align*}
    for any $N\in \N$.    
    
     Since $\cB$ has dimensions about $\mu\times \sqrt{\mu\mu'}$ and since 
    $\supp \widetilde\psi_\vartriangle^\ell(x,y,\cdot), \supp  \widetilde\psi_\vartriangle^\ell(x,w,\cdot)$ are contained in intervals of length $2^{-\ell}$, 
     we get
    \begin{align*}
        |\mathcal E_j^*\mathcal E_k(y,w)| \lesssim_N  2^{-2\ell}  \sqrt{\mu\mu'} |\lambda 2^{-\ell}\mu'(j-k)|^{-N}(\wt\chi_{\cB'})_j(y)(\wt\chi_{\cB'})_k(w).
    \end{align*}
         Note that $\supp (\wt\chi_{\cB'})_j$ and $\supp (\wt\chi_{\cB'})_k$ are included in rectangles of dimensions about $\mu'\times 2^{-\ell}\smup$. 
           Thus, using  Lemma \ref{schur}, we obtain  the desired bound \eqref{ejek}.  
\end{proof}
Thus the contribution away from the symmetric time satisfies \eqref{l2}; the remaining analysis begins with the geometry of \(\{\cD=0\}\) in Section~\ref{sec:geometry}.
 \section{Geometry of the degeneracy set \texorpdfstring{$\{\cD=0\}$}{D=0}}\label{sec:geometry}  
Through this and the subsequent sections we prove  the following, which  together with  \eqref{e:sumfElk}  proves  Theorem \ref{thm:core} (see \eqref{e:decompm-1st}).

\begin{prop}\label{main-p}
Under the hypotheses of Theorem~\ref{thm:core}, let
\(\mathcal O_\lambda[\psi_\vartriangle]\) be defined by
\eqref{ola} and \eqref{psi-tri}. Then,
\begin{align}\label{e:main-fP}
    \| \wt\chi_{\cB} \mathcal O_\lambda[\psi_\vartriangle]  \wt\chi_{\cB'}\|_{L^{q_e'}\to L^{q_e}}\lesssim   \lambda^{-\frac{4}{5}}\biggl(\frac{\mu'}{\mu}\biggr)^{\frac{1}{20}}.
\end{align}
\end{prop}

In this section, we determine the sizes and directions of the gradients of \(\cD\), estimate its Hessians in the corresponding normal and tangential directions, and establish the stability properties needed for the rectangular decomposition in Section~\ref{sec:decomp}.

The decomposition \eqref{e:decompm-1st}  gives rise to a tighter  localization in $t$ for the kernel representation of $\mathcal O_\lambda[\psi_\vartriangle]$, while 
the other part is controlled by an acceptable bound (Proposition \ref{prop:estfE}). In fact, recalling \eqref{i:ptcP-Q}, one sees that 
the scale of the localization for $\mathcal O_\lambda[\psi_\vartriangle]$  is determined by the size of $\mathcal D$.  Hence, one may expect that further localization in time becomes possible, after discarding some acceptable contribution, if one  considers  the operator $\mathcal O_\lambda[\psi_\vartriangle]$ over a subset  $\cB\times \cB'$ where $\mathcal D\sim c$  for some nonzero constant $c$.   This idea naturally leads to decomposition of  $\cB\times \cB'$ into subsets on which    $\mathcal D$ has values comparable to dyadic numbers.  This  decomposition will play a crucial role in establishing the estimate \eqref{e:main-fP}.   We refer the reader forward to 
Section \ref{subsec:decomp-prod} for details.  

In order to obtain an appropriate decomposition for our purpose, we need to investigate  the behavior of $\cD$ near $\{ (x, y): \cD(x,y)=0\}$.   

Let $\dels$ denote the bottom scale given by
\begin{align}\label{d:top&bottom}
    \dels = 2^{\lfloor\log_2(\lambda^{-\frac{2}{3}}\mu)\rfloor}.
\end{align}
The scale \(\dels\) is chosen so that the corresponding time width
\(\dels^{1/2}\mu^{-1/2}\) is comparable to the cubic oscillatory scale
\(\lambda^{-1/3}\). Moreover, by \eqref{mm} and \eqref{d:top},
choosing \(C_\ast\) sufficiently large guarantees
\(2\dels\leq\delt\).
Here and henceforth,  $\delta$ denotes a  dyadic number  such that 
\Be \label{ddd} \dels \le \delta\le \delt.\Ee 
In this hierarchy, \(\delt\sim\epz\mu\mu'\) is the initial degeneracy scale, \(\delta\) records the dyadic size of \(|\cD|\), and \(\dels\sim\lambda^{-2/3}\mu\) is the terminal oscillatory scale. At level \(\delta\), the corresponding time width is \(\delta^{1/2}\mu^{-1/2}\), while the associated spatial scales are given in \eqref{scale}; the further input-side refinement is introduced in \eqref{s'}.

For each $\delta$, we define
\begin{align}
   \label{fsd} \fS_\delta^\pm &:= \{(x,y)\in \cB\times\cB': \delta < \pm\cD(x,y) \le 4 \delta\}, \\
    \label{fsdc}\fS_\delta^\circ &:= \left\{(x,y)\in \cB\times\cB': |\cD(x,y)| \le \frac{3\delta}{2} \right\}.
\end{align}
Thus, 
\Be
\label{bb-ss}
 \cB\times\cB'=\Big(\bigcup_{\kappa=\pm}\bigcup_{\dels \le \delta\le \delt} \fS_\delta^\kappa \Big) \cup   \fS_\dels^\circ .
 \Ee
 Throughout the paper, $\kappa$ denotes an element in $\{+,-,\circ\}$. 
Note that $\fS_\delta^+$ and $\fS_\delta^-$ overlap with $\fS_\delta^\circ$.  The regions $\fS_\delta^\pm$, $\fS_\delta^\circ$ serve as a preliminary decomposition 
for the more refined one given in Section \ref{sec:decomp}.  We also have
\[|\cD(x,y)|\le 2\delt, \quad  (x,y)\in \cB\times\cB'.\]

\subsection{The function \texorpdfstring{$\cD$}{D} and the normal vectors
\texorpdfstring{$\fa,\fb$}{n and n-prime}} 
In this subsection, we collect  the various properties of $\cD$,  which we use later.  We begin by observing the following. 

\begin{lem} Let $\mu, \mu'$ satisfy \eqref{mm}.  
\label{lem:pxpy} For $(x,y)\in \cB\times\cB'$, we have
\begin{align}
\label{e:sizeofxDyD}
    |\partial_x\cD(x,y)|&\sim (\mu\mu')^{\frac{1}{2}},
    \\
    \label{e:sizeofxDyD'}
     |\partial_y \cD(x,y)|&\sim \mu. 
\end{align}
\end{lem}

\begin{proof} 
Note that 
\begin{align}
\label{pxpy0} \partial_x \cD= 2y \inp x y- 2x, 
\\
 \label{pxpy1}   \partial_y \cD= 2x \inp x y- 2y .
\end{align}
Hence, we have 
\begin{align*}
    2^{-2} {|\partial_x \cD(x,y)|^2} = (2-|x|^2 - |y|^2)(1-|y|^2) + \cD(x,y) (|y|^2-2). 
 \end{align*}
By \eqref{bb},  \eqref{mm}, and \eqref{r:cond-cDh}  with a sufficiently  small $\epz>0$,  \eqref{e:sizeofxDyD} follows. 
Interchanging the roles of  $x$ and $y$ in the above identity, we obtain \eqref{e:sizeofxDyD'}. 
\end{proof}

\begin{figure}[t]
  \centering
  \vspace*{0.25cm}
  \begin{tikzpicture}[
    >={Stealth[length=2.5pt,width=2.8pt]},
    every node/.style={font=\scriptsize},
    normal/.style={->,thin,black},
    connector/.style={dash pattern=on 1.2pt off 0.7pt,gray!60},
    guide/.style={densely dotted,gray!60}
  ]
    \begin{scope}[scale=1.55]
      \draw[->,thin] (-0.15,0) -- (4.15,0) node[right] {$e_1$};
      \draw[guide,gray!80]
        plot[domain=-14:34,samples=80,variable=\t]
        ({3.6*cos(\t)},{3.6*sin(\t)});
      \node[gray!70] at (3.60,1.55) {$|z|=1$};
      \filldraw[fill=blue!12,draw=blue!70!black,thick]
        (1.55,-0.133) rectangle (2.45,0.133);
      \filldraw[fill=red!12,draw=red!70!black,thick]
        (3.00,0.62) rectangle (3.20,1.02);
      \coordinate (Bglobalcenter) at (2.00,0);
      \coordinate (Bpglobalcenter) at (3.10,0.82);
      \node[blue!70!black,font=\scriptsize] at (2.00,0) {$\mathcal B$};
      \node[red!70!black,font=\scriptsize] at (3.10,0.82) {$\mathcal B'$};
    \end{scope}

    \begin{scope}[shift={(6.95,-1.15)}]
      \begin{scope}[scale=0.67,transform shape]
        \filldraw[fill=blue!8,draw=blue!70!black,thick]
          (0,0) rectangle (5.0,1.45);
        \coordinate (Binsetcenter) at (2.50,0.725);
        \node[blue!70!black,font=\small] at (0.38,1.15) {$\mathcal B$};
        \draw[blue!85!black,thin]
          plot[smooth] coordinates
            {(0.25,0.42) (1.45,0.53) (2.55,0.70)
             (3.75,0.91) (4.75,1.15)};
        \coordinate (xpt) at (2.55,0.70);
        \fill (xpt) circle (2pt);
        \node[below right] at (xpt) {$x$};
        \draw[connector] (1.55,0.54) -- (3.85,0.94);
        \draw[normal] (xpt) -- ++(-0.20,1.15)
          node[above,font=\large] {$\mathbf n$};
        \node[blue!85!black] at (4.00,1.18) {$\Gamma_y$};
      \end{scope}
    \end{scope}

    \begin{scope}[shift={(9.30,0.70)}]
      \begin{scope}[scale=0.67,transform shape]
        \filldraw[fill=red!8,draw=red!70!black,thick]
          (0,0) rectangle (1.55,3.70);
        \coordinate (Bpinsetcenter) at (0.775,1.85);
        \node[red!70!black,font=\small] at (0.28,3.38) {$\mathcal B'$};
        \draw[red!85!black,thin]
          plot[smooth] coordinates
            {(0.98,0.22) (0.91,1.05) (0.82,1.85)
             (0.72,2.72) (0.60,3.45)};
        \coordinate (ypt) at (0.82,1.85);
        \fill (ypt) circle (2pt);
        \node[right] at (ypt) {$y$};
        \draw[connector] (0.92,0.95) -- (0.66,2.95);
        \draw[normal] (ypt) -- ++(-1.25,-0.163)
          node[left,yshift=0.15cm,font=\large] {$\mathbf n'$};
        \node[red!85!black,right] at (0.92,0.58) {$\Gamma'_x$};
      \end{scope}
    \end{scope}

    \draw[connector,shorten <=0.70cm,shorten >=1.69cm]
      (Bglobalcenter) -- (Binsetcenter);
    \draw[connector,shorten <=0.16cm,shorten >=0.52cm]
      (Bpglobalcenter) -- (Bpinsetcenter);
  \end{tikzpicture}
  \vspace{0.25cm}
  \caption{The global placement of $\cB$ and $\cB'$ is as in
  Figure~\ref{fig:typical-B-Bprime}. For $(x,y)\in\cB\times\cB'$, the
  local level curves
  $\Gamma_y=\{z\in\cB:\cD(z,y)=\cD(x,y)\}$ and
  $\Gamma'_x=\{w\in\cB':\cD(x,w)=\cD(x,y)\}$ are shown schematically.
  The arrows $\mathbf n$ and $\mathbf n'$ indicate their normal
  directions at $x$ and $y$, respectively.}
  \label{fig:local-level-curves}
\end{figure}
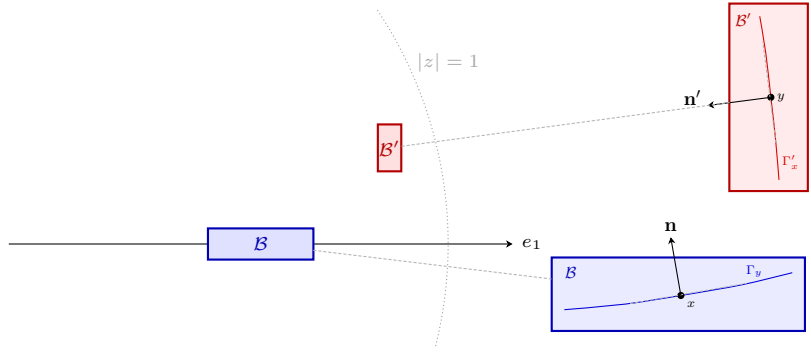

 For $(x,y)\in \cB\times\cB'$, 
we now define unit vectors $\fa, \fb$ by
\begin{align}
\label{normal}
 \fa=\fa(x,y) = \frac{\partial_x\cD(x,y)}{|\partial_x\cD(x,y)|},
\\
\label{normal'}
 \fb = \fb(x,y) = \frac{\partial_y \cD(x,y)}{|\partial_y \cD(x,y)|}.
\end{align}
Note that  $\fa(x,y)$ is normal to the curve  
$
    \{z\in \cB: \cD(z,y) = \cD(x,y)\}
$
at $z=x$. Similarly, $\fb(x,y)$ is normal to $\{w\in \cB': \cD(x,w) = \cD(x,y)\}$ at $w=y$.

The following lemma shows that $\fa$ and $\fb$ are almost parallel to $e_2$ and $e_1$, respectively.

\begin{lem}\label{lem:directionab}    Let $\mu, \mu'$ satisfy \eqref{mm}. 
For $(x,y)\in \cB\times\cB'$, we have  
    \begin{align}
    \label{e:direction}
        |\inp{\fa}{e_1}| \sim \biggl(\frac{\mu'}{\mu}\biggr)^{\frac{1}{2}},
       \\
       \label{e:direction'}
         |\inp{\fb}{e_2}|\sim \biggl(\frac{\mu'}{\mu}\biggr)^{\frac{1}{2}}.
    \end{align}
    Moreover, $\inp{\fa}{e_1}, \inp{\fb}{e_2}, \inp{\fb}{e_1}<0<  \inp{\fa}{e_2} $.
\end{lem}
\begin{proof}
    First,  from  \eqref{pxpy0}, \eqref{bb} and \eqref{inp} we note that 
    \[
    -\inp{\partial_x \cD(x,y)}{e_1}=2 (x_1-\inp{x}{y}y_1)=2\frac{x_1^2-\inp{x}{y}^2y_1^2}{x_1+\inp{x}{y}y_1}\sim x_1^2-\inp{x}{y}^2y_1^2.
    \]
    Since $|x_2|, \, |y_2| \lesssim  \sqrt{\mu\mu'}$,  $x_1^2-\inp{x}{y}^2y_1^2= |x|^2- \inp{x}{y}^2 |y|^2 +O(\mu\mu')$. Thus, 
    \begin{align*}
         -\inp{\partial_x\cD(x,y)}{e_1} =(1+\inp{x}{y}^2)(1-|y|^2) -\cD(x,y) +O(\mu\mu'). 
    \end{align*}
    Hence,   by \eqref{bb} and  \eqref{r:cond-cDh},  we  have  $0<- \inp{\partial_x\cD(x,y)}{e_1} \sim \mu'$.  
        Combining this with \eqref{e:sizeofxDyD}, we get \eqref{e:direction}. 
        
        On the other hand, since $\inp{\partial_y\cD(x,y)}{e_2} =2(\inp{x}{y}x_2 -y_2)$, by \eqref{bb} we get $0<-\inp{\partial_y\cD(x,y)}{e_2}\sim  \sqrt{\mu\mu'}$.  This and \eqref{e:sizeofxDyD'}   yield \eqref{e:direction'}. 
        
       Since $\inp{\fa}{e_1}, \inp{\fb}{e_2} <0$ is already shown above, we need only to show $\inp{\fb}{e_1}<0<  \inp{\fa}{e_2}$. Note that $\inp{\partial_x \cD(x,y)}{e_2}=2 (\inp{x}{y} y_2-x_2)$. By \eqref{bb}, it follows that $0<\inp{\partial_x \cD(x,y)}{e_2}\sim \sqrt{\mu\mu'}$. Hence, $0<  \inp{\fa}{e_2}$.  Finally, 
       note that $\inp{\partial_y\cD(x,y)}{e_1} =2(\inp{x}{y}x_1 -y_1)$. Since $|\cD(x,y)|\ll \mu'\mu$, by a similar argument as above we see  $-\inp{\partial_y\cD(x,y)}{e_1} \sim 
       y_1^2- x_1^2\inp{x}{y}^2= (1-|x|^2)(1+\inp{x}{y}^2)+O(\mu\mu')\sim \mu$.  Thus,  $\inp{\fb}{e_1}<0.$      
\end{proof}

Let $(x,y), (x_\zc,y_\zc)\in \cB\times\cB'$.  Then,  it follows  that 
\[  x-x_\zc= O(\epz \mu ) e_1+ O(\epz(\mu\mu')^{\frac{1}{2}}) e_2, \quad y-y_\zc= O(\epz \mu' ) e_1+ O(\epz(\mu\mu')^{\frac{1}{2}}) e_2.\]
Hence, the next lemma is an immediate consequence of Lemma \ref{lem:directionab}.

\begin{lem}\label{lem:size}
Let $(x,y), (x_\zc,y_\zc)\in \cB\times\cB'$. Then there exists $C>0$ such that
\begin{align*}
    &|\inp{x-x_\zc}{\fa(x_\zc,y_\zc)}|\le C\epz(\mu\mu')^{\frac{1}{2}}, \ &&|\inp{x-x_\zc}{\fa_{\!\perp}(x_\zc,y_\zc)}|\le C\epz \mu,\\
    & |\inp{y-y_\zc}{\fb(x_\zc,y_\zc)}|\le C\epz\mu', \ &&|\inp{y-y_\zc}{\fa'_{\!\perp}(x_\zc,y_\zc)}|\le C\epz (\mu\mu')^{\frac{1}{2}}.
\end{align*}
\end{lem}

We now observe that $\fa$ and $\fb$ are almost perpendicular.

\begin{lem}\label{lem:alp} 
Let $\dels \le \delta\le \delt$.  If  $(x,y)\in \fS^\pm_\delta$ or $\fS_{\delta}^\circ$, we have
\Be 
\label{e:anglefafb0}
 |\inp{\fa(x,y)}{\fb(x,y)}|    \lesssim \delta \mu^{-\frac{3}{2}}(\mu')^{-\frac{1}{2}}.
 \Ee
 In particular, we have  $ |\inp{\fa(x,y)}{\fb(x,y)}|  \ll  \delta^{1/2}\mu^{-1}.$  
\end{lem}
\begin{proof}  Using \eqref{pxpy0} and \eqref{pxpy1}, we have $
\inp{\partial_x\cD(x,y)}{\partial_y\cD(x,y)} = 4\inp xy \cD(x,y). 
$
Thus,  
\[  \inp{\fa(x,y)}{\fb(x,y)}= \frac{4\inp xy \cD(x,y)}{|\partial_x\cD(x,y)||\partial_y\cD(x,y)|}.\]
From \eqref{inp},  \eqref{e:sizeofxDyD},  and  \eqref{e:sizeofxDyD'}, the desired inequality  immediately follows.
\end{proof}

\newcommand{\fap}{\fa_{\!\perp}}
\newcommand{\fbp}{\fa'_{\!\perp}}

 For $(x,y)\in \cB\times \cB'$, let  
\[ \fa_{\!\perp} \equiv \fa_{\!\perp}(x,y), \quad \fa'_{\!\perp} \equiv \fa'_{\!\perp}(x,y)\] 
 be the unique unit vectors  perpendicular to $\fa(x,y)$ and $\fb(x,y)$ respectively, such that 
 $\inp{\fa_{\!\perp}}{e_1}, \inp{\fbp}{e_1}>0$.   Such choices are clearly possible thanks to Lemma \ref{lem:directionab}.
Using this, one can express
\begin{align}\label{i:relsfafbperp}
    \fa_{\!\perp} = \varrho_1\fa + c_1\fb,\quad \fa'_{\!\perp} = \varrho_2 \fb + c_2\fa
\end{align}
with $\varrho_1, \varrho_2$, $c_1$, and $c_2$ satisfying 
\Be
\label{rrcc} 
 |\varrho_1|, |\varrho_2| \lesssim \epz\biggl(\frac{\mu'}{\mu}\biggr)^{\frac{1}{2}},
\quad |c_1|, |c_2|\sim 1. 
\Ee 
Indeed, setting $\alpha=\inp{\fa}{\fb}$, the identity
$\inp{\partial_x\cD}{\partial_y\cD}=4\inp{x}{y}\cD$, Lemma \ref{lem:pxpy}, and
\eqref{r:cond-cDh} give
$
|\alpha|\lesssim |\cD(x,y)|\mu^{-3/2}(\mu')^{-1/2}
\lesssim \epz(\mu'/\mu)^{1/2}\ll1.
$
Taking inner products in \eqref{i:relsfafbperp} gives
$\varrho_i=-c_i\alpha$, while the unit-vector conditions give
$1=c_i^2(1-\alpha^2)$, proving \eqref{rrcc}.

Using this, we prove the following lemmas concerning the second-order derivatives of $\cD.$ While the inequality $\lesssim$ suffices for our purpose, we prove some of the estimates using the more informative relation $\sim$.

\begin{lem}\label{lem:pxpxD}
    Let $\mu$ and $\mu'$ satisfy \eqref{mm} and $(x,y)\in \cB\times\cB'$. Denote $\fa=\fa(x,y), \fb=\fb(x,y), \fap=\fap(x,y),$ and $\fbp=\fbp(x,y)$, respectively. Then, we have
    \begin{align}\label{e:quadDall}
|\inp{\fa}{\partial_x \partial_x^\intercal  \cD(x,y)\fa}|\sim 1,
          &\qquad |\inp{\fb}{\partial_y \partial_y^\intercal  \cD(x,y)\fb}|\lesssim 1 ,
    \\[6pt]
   \label{e:quadDall1}  |\inp{\fa_{\!\perp}}{\partial_x \partial_x^\intercal  \cD(x,y)\fa_{\!\perp}}| \sim \frac{\mu'}{\mu},
        &\qquad |\inp{\fa'_{\!\perp}}{\partial_y \partial_y^\intercal  \cD(x,y)\fa'_{\!\perp}}| \sim 1,
    \\[6pt]
   \label{e:quadDall2}
     |\inp{\fa}{\partial_x \partial_x^\intercal  \cD(x,y)\fa_{\!\perp}}|\sim \biggl(\frac{\mu'}{\mu}\biggr)^{\frac{1}{2}},
        &\qquad |\inp{\fb}{\partial_y \partial_y^\intercal  \cD(x,y)\fa'_{\!\perp}}|\sim \biggl(\frac{\mu'}{\mu}\biggr)^{\frac{1}{2}}.
\end{align}
\end{lem}

\begin{proof} 
Using \eqref{pxpy0} and \eqref{pxpy1}, we see that  
\begin{align*}
    \inp{\partial_x\cD}{(yy^\intercal - I)\partial_x \cD} &= 4\inp xy^2(|y|^2-1)^2 - |\partial_x\cD|^2, \\
    \inp{\partial_y\cD}{(xx^\intercal - I)\partial_y \cD} &= 4\inp xy^2(|x|^2-1)^2 - |\partial_y\cD|^2. 
\end{align*}
Since $\mu'\ll \mu$, combining this with \eqref{e:sizeofxDyD} and 
\eqref{e:sizeofxDyD'}, we have 
\begin{align}\label{e:quadD}
    |\inp{\partial_x\cD}{(yy^\intercal - I)\partial_x \cD}| \sim \mu\mu',\quad |\inp{\partial_y\cD}{(xx^\intercal - I)\partial_y \cD}|\lesssim \mu^2.
\end{align}
Note that
\begin{align}\label{i:D2x2y}
   yy^\intercal - I=\frac{1}{2} \partial_x \partial_x^\intercal  \cD(x,y) ,\quad    xx^\intercal - I=\frac{1}{2}\partial_y\partial_y^\intercal \cD(x,y).
\end{align}
Substituting this into \eqref{e:quadD}, by  \eqref{e:sizeofxDyD} and \eqref{e:sizeofxDyD'} we obtain the estimates in  \eqref{e:quadDall}.

Next, we prove the estimates in   \eqref{e:quadDall1}. Using \eqref{i:relsfafbperp} and \eqref{i:D2x2y}, we write 
\begin{align}
\label{i:abperp2D}
    &2^{-1}\inp{\fa_{\!\perp}}{\partial_x \partial_x^\intercal  \cD(x,y)\fa_{\!\perp}} = c_1^2\frac{\inp{\partial_y\cD}{(yy^\intercal - I)\partial_y \cD}}{|\partial_y\cD|^{2}} \\
    &\qquad + 2c_1\varrho_1\frac{\inp{\partial_y\cD}{(yy^\intercal - I)\partial_x \cD}}{|\partial_y\cD||\partial_x\cD|}+ \varrho_1^2 \frac{ \inp{\partial_x\cD}{(yy^\intercal - I)\partial_x \cD}}{|\partial_x\cD|^{2}},  \nonumber
    \end{align}
    \begin{align} 
    \label{i:abperp2D'}
    &2^{-1}\inp{\fa'_{\!\perp}}{\partial_y \partial_y^\intercal  \cD(x,y)\fa'_{\!\perp}} = c_2^2\frac{\inp{\partial_x\cD}{(xx^\intercal - I)\partial_x \cD}}{|\partial_x\cD|^{2}} \\
    &\qquad + 2c_2\varrho_2\frac{\inp{\partial_x\cD}{(xx^\intercal - I)\partial_y \cD}}{|\partial_y\cD||\partial_x\cD|} + \varrho_2^2 \frac{\inp{\partial_y\cD}{(xx^\intercal - I)\partial_y \cD}}{|\partial_y\cD|^{2}}. \nonumber
\end{align}

 We separately handle  each term on the right-hand sides to obtain the desired bounds.  To this end, we first note that 
\begin{align}\label{e:quadD2x2y}
    |\inp{\partial_y\cD}{(yy^\intercal - I)\partial_y \cD}|\sim \mu\mu', \quad |\inp{\partial_x\cD}{(xx^\intercal - I)\partial_x \cD}|\sim \mu\mu'. 
\end{align}
Indeed, using \eqref{pxpy1}, after a computation, one has
\begin{align*}
    4^{-1}\inp{\partial_y\cD}{(yy^\intercal - I)\partial_y \cD}  &= -(1-|y|^2)(1-|x|^2) + \cD(|x|^2 - 1) + \cD + \cD^2. 
\end{align*}
Thus, by \eqref{r:cond-cDh} and \eqref{bb}, the former estimate in \eqref{e:quadD2x2y} follows.  The latter one can be shown in the same manner, since a similar identity holds by symmetry if we  interchange  the roles of $x,y$. 

Now, we claim that 
\begin{align}\label{e:quadDxys}
    |\inp{\partial_y\cD}{(yy^\intercal - I)\partial_x \cD}| \sim \mu\mu', \quad |\inp{\partial_x\cD}{(xx^\intercal - I)\partial_y \cD}|\sim \mu\mu'.
\end{align}
To show the first estimate, we observe 
\[  \inp{\partial_y\cD}{(yy^\intercal - I)\partial_x \cD} = 4\inp xy(|y|^2 - 1)(\inp xy^2 - |y|^2) -4\inp xy \cD(x,y).\]
Since $|\inp xy^2 - |y|^2|\sim \mu$,   the desired estimate follows by  \eqref{r:cond-cDh}.  The latter one in \eqref{e:quadDxys} can be shown similarly using the identity obtained by interchanging the roles of $x,y$ in the above, since $|\inp xy^2 - |x|^2|\sim \mu'$.

Now, combining the former inequalities in  \eqref{e:quadD}, \eqref{e:quadD2x2y}, and \eqref{e:quadDxys} into \eqref{i:abperp2D}, we obtain the first estimate in \eqref{e:quadDall1}.  Now, replacing \eqref{i:abperp2D} by \eqref{i:abperp2D'} and using the latter inequalities in \eqref{e:quadD}, \eqref{e:quadD2x2y}, and \eqref{e:quadDxys}, 
 yields the second one in  \eqref{e:quadDall1}.

The estimates in \eqref{e:quadDall2} follow in the same manner from \eqref{i:relsfafbperp} and \eqref{i:D2x2y}. In each expansion, the \(c_i\)-term has size \((\mu'/\mu)^{1/2}\), whereas the \(\varrho_i\)-term is \(O(\epz)\) times this size by \eqref{rrcc}. Since the \(c_i\)-term dominates the \(\varrho_i\)-term, the asserted lower bounds follow.
\end{proof}

In the following, we prove similar estimates for $\partial_x \partial_y^\intercal \cD$.

\begin{lem}\label{lem:pxpyD}
Let $\mu,\mu'$ satisfy \eqref{mm} and $(x,y)\in \cB\times\cB'$. Then,  we have
    \begin{align}
    \label{e:quadD-xyabperp}
    |\inp{\fa_{\!\perp}}{\partial_x \partial_y^\intercal \cD(x,y) \fa'_{\!\perp}}|&\sim \biggl(\frac{\mu'}{\mu}\biggr)^{\frac{1}{2}},
    \\
    \label{e:quadD-xyabperp1}
    |\inp{\fa}{\partial_x \partial_y^\intercal \cD(x,y) \fb}|&\sim \biggl(\frac{\mu'}{\mu}\biggr)^{\frac{1}{2}},
    \\
    \label{e:quadD-xyabperp2}
    |\inp{\fa}{\partial_x \partial_y^\intercal \cD(x,y) \fa'_{\!\perp}}|&\sim 1,
    \\
    \label{e:quadD-xyabperp3}
    |\inp{\fa_{\!\perp}}{\partial_x \partial_y^\intercal \cD(x,y) \fb}|&\lesssim 1. 
\end{align}
\end{lem}

\begin{proof}  The estimate \eqref{e:quadD-xyabperp3} is clear.  We need only to show the other estimates. 

We first prove \eqref{e:quadD-xyabperp}.  Note from \eqref{pxpy0} that \[\partial_x \partial_y^\intercal \cD(x,y) = \mathrm M:= 2(yx^\intercal + \inp xy I).\]
Combining this and \eqref{i:relsfafbperp}, we write $ \inp{\fa_{\!\perp}}{\partial_x \partial_y^\intercal \cD(x,y) \fa'_{\!\perp}} $ as follows: 
\begin{align}\label{i:quadD-xyabperp}
    \begin{aligned}
       & c_1c_2\inp{\fb}{\mathrm M\fa} + \varrho_1\varrho_2 \inp{\fa}{\mathrm M\fb} + \varrho_1c_2\inp{\fa}{\mathrm M\fa} + \varrho_2c_1\inp{\fb}{\mathrm M\fb}. 
    \end{aligned}
\end{align}
Using \eqref{pxpy0} and \eqref{pxpy1}, via a computation we get
\begin{align*}
  E:=  \inp{\partial_y \cD}{\mathrm M\partial_x \cD} &= 8(1-|x|^2 - \cD)(1- |y|^2 - \cD) + 8(\inp xy)^2 \cD,
     \\[2pt]
 F:=   \inp{\partial_x \cD}{\mathrm M\partial_y \cD} &= 8\inp xy^2(1-|x|^2)(1- |y|^2) + 8(\inp xy)^2 \cD,
\end{align*}
and
\begin{align*}
G:=   \inp{\partial_x \cD}{\mathrm M\partial_x \cD} &= 8\inp xy (1-|y|^2)(1-|y|^2-\cD)+2\inp xy |\partial_x\cD|^2, 
    \\[2pt]
  H:=  \inp{\partial_y \cD}{\mathrm M\partial_y \cD} &= 8\inp xy (1-|x|^2)(1-|x|^2-\cD)+2\inp xy |\partial_y\cD|^2.
\end{align*}

Combining those identities with \eqref{bb} and \eqref{r:cond-cDh}, we see $E\sim \mu'\mu$ and $F\sim \mu'\mu$. 
Recall \eqref{normal} and \eqref{normal'}. 
 Thus, dividing the estimates by 
$|\partial_x\cD||\partial_y\cD|$, by  Lemma \ref{lem:pxpy} we obtain 
\Be
\label{npm}
    \inp{\fb}{\mathrm M\fa}\sim \biggl(\frac{\mu'}{\mu}\biggr)^{\frac{1}{2}}, \quad \inp{\fa}{\mathrm M\fb}\sim \biggl(\frac{\mu'}{\mu}\biggr)^{\frac{1}{2}}.
    \Ee
	   Similarly, we also see $G\sim \mu'\mu$ and $H\sim  \mu^2$. Thus,  dividing these estimates by
$|\partial_x\cD|^2$ and $|\partial_y\cD|^2$, respectively,  we obtain 
\begin{align}
\label{nmn}
    \inp{\fa}{\mathrm M\fa}\sim 1,\quad &\inp{\fb}{\mathrm M\fb}\sim 1.
\end{align}

Now, we recall \eqref{rrcc}. 
Since $|\varrho_1|, |\varrho_2|\lesssim \epz  (\mu'/\mu)^{1/2},$
combining those bounds with \eqref{i:quadD-xyabperp}, we obtain the desired estimate  \eqref{e:quadD-xyabperp}.

The remaining  estimates \eqref{e:quadD-xyabperp1} and \eqref{e:quadD-xyabperp2}  can be handled easily. 
In particular, \eqref{e:quadD-xyabperp1}  follows from the latter inequality in \eqref{npm}. For \eqref{e:quadD-xyabperp2}, one uses the latter identity 
in \eqref{i:relsfafbperp} to get $\inp{\fa}{\partial_x \partial_y^\intercal \cD(x,y) \fa'_{\!\perp}}=\varrho_2 \inp{\fa}{\mathrm M\fb}+ c_2 \inp{\fa}{\mathrm M\fa}$. 
Here the \(c_2\)-term has size comparable to \(1\) by \eqref{nmn}, whereas the \(\varrho_2\)-term is \(O(\epz\mu'/\mu)\) by \eqref{rrcc} and \eqref{npm}. Since the \(c_2\)-term dominates the \(\varrho_2\)-term, \eqref{e:quadD-xyabperp2} follows.
\end{proof}

\newcommand{\scl}{\sigma}
\newcommand{\sclr}{\sigma_{\!\mathsmaller \perp}}
\newcommand{\sclp}{ \sigma^{\prime}}
\newcommand{\sclrp}{\sigma_{\!\mathsmaller \perp}^{\,\prime}}

\subsection{Tangential rectangles}    
In Section \ref{sec:decomp}, the regions $\fS_\delta^\kappa$, $\kappa=\pm, \circ$ will be covered by collections of rectangles, which are essentially tangential to 
the level sets.

For $\delta$ satisfying \eqref{ddd},  define basic scales  by setting  
\Be
\label{scale}
\begin{aligned}   
  &\scl=\epz\delta (\mu\mu')^{-\frac{1}{2}},  \qquad &&    \sclr=\epz\delta^{\frac{1}{2}}\mu^{\frac{1}{2}}(\mu')^{-\frac{1}{2}},
\\  
 & \sclp=\epz\delta \mu^{-1},    &&\sclrp=\epz\delta^{\frac{1}{2}}. 
\end{aligned}
\Ee

For $(x,y)\in \cB\times\cB'$,  define 
    \begin{align}
  \label{rx0}      R_y(x) &=\big\{ u\in \cB: |\inp{u-x}{\fa(x, y)}|\le c \scl, \ |\inp{u-x}{\fa_{\!\perp}(x, y)}|\le  c\sclr \big \}, 
        \\[3pt]
  \label{ry0}      R'_x(y) &= \big\{w\in \cB': |\inp{w-y}{\fb(x, y)}|\le  c\sclp,\ |\inp{w-y}{\fa'_{\!\perp}(x, y)}|\le c\sclrp\big\}
    \end{align}
    for an absolute constant \(0<c\ll1\), chosen sufficiently small as
needed below.

    The next lemma, which shows stability of the value of $\cD$ on  $R_y(x) \times R'_x(y)$,   suggests that $\fS_\delta^\kappa$ can be covered by those rectangles.

\newcommand{\yo}{{y_\zc}}
\newcommand{\xo}{{x_\zc}}

\begin{lem}\label{lem:locD-strong} 
    Let $\dels \le \delta \le \delt$ and $(x_\zc,y_\zc)\in \cB\times\cB'$.  Assume that  $|\cD(x_\zc, y_\zc)|\le 4\delta$. Then, for  $(x,y)\in R_{\yo} (\xo) \times R'_\xo(\yo) $,  we have
    \[
    |\cD(x,y) - \cD(x_\zc,y_\zc)|\lesssim c\epz\delta. 
    \]
\end{lem}
\begin{proof}
    Using the fact that $\cD$ is a polynomial of degree $4$, we write 
    \begin{align*}
    \cD(x,y) -\cD(x_\zc,y_\zc) &= \Big(\sum_{|\alpha|=1,2} + \sum_{|\alpha|=3,4}  \Big)\frac{\partial_{x,y}^\alpha\cD(x_\zc,y_\zc)}{\alpha!}(x-x_\zc,y-y_\zc)^\alpha \\
        &=: \mathcal E_1 + \mathcal E_2.
    \end{align*}
It is sufficient to show $|\mathcal E_1|,\ |\mathcal E_2|\le Cc \epz \delta$. We break 
\[ \mathcal E_1= \mathcal E_{1,1}+ \mathcal E_{1,2}+ {\mathcal E_{1,3}}, \]
where 
\begin{align}
\label{i:exp-E1}
\mathcal E_{1,1}=&\inp{\partial_x \cD(x_\zc,y_\zc)}{x-x_\zc}+
         \inp{\partial_y \cD(x_\zc,y_\zc)}{y-y_\zc},
         \\[2pt]
         \label{i:exp-E2}
\mathcal E_{1,2}= &\frac{\inp{x-x_\zc}{\partial_x\partial_x^\intercal\cD(x_\zc,y_\zc)(x-x_\zc)}}{2}+\frac{ \inp{y-y_\zc}{\partial_y\partial_y^\intercal\cD(x_\zc,y_\zc)(y-y_\zc)}}{2} ,
\\[2pt]
\label{i:exp-E3}
     \mathcal E_{1,3}=& \inp{y-y_\zc}{\partial_x^\intercal\partial_y\cD(x_\zc,y_\zc)(x-x_\zc)}.
\end{align}

For simplicity, we denote  
\Be 
\label{nnn} \fa =\fa(x_\zc,y_\zc),\quad \fa_{\!\perp}=\fa_{\!\perp}(x_\zc,y_\zc),\quad  \fb =\fb(x_\zc,y_\zc),\quad \fa_{\!\perp}'=\fa_{\!\perp}'(x_\zc,y_\zc),\Ee  
and  write
\begin{align}\label{d:xfayfb}
    x-x_\zc = c_{\fa} \fa+ c_{\fa_{\!\perp}}\fa_{\!\perp},\quad y-y_\zc= c_{\fb} \fb +c_{\fa'_{\!\perp}}\fa'_{\!\perp}. 
\end{align}
 Since $(x,y)\in R_{\yo} (\xo) \times R'_\xo(\yo)$, it follows that 
\begin{align}
    \begin{aligned}\label{e:locxy-RR'}
        |c_\fa| \le c\scl,\quad |c_{\fa_{\!\perp}}|\le c\sclr, \quad 
        |c_{\fb}|\le c\sclp, \quad |c_{\fa'_{\!\perp}}|\le c\sclrp.
    \end{aligned}
\end{align}

Combining  \eqref{i:exp-E1} and \eqref{d:xfayfb},  by Lemma \ref{lem:pxpy} and \eqref{scale} we immediately obtain $|\mathcal E_{1,1}| \lesssim c\epz \delta$. Similarly, applying Lemma \ref{lem:pxpxD}  after inserting \eqref{d:xfayfb} into \eqref{i:exp-E2}, by \eqref{e:locxy-RR'} we have
 \[  |\mathcal E_{1,2}|\lesssim  c^2 (\scl^2 + (\sclp)^2)+ \biggl(\frac{\mu'}{\mu}\biggr)^{\frac{1}{2}} c^2 (\scl\sclr+\sclp\sclrp)+ \frac{\mu'}{\mu} c^2 \sclr^2+c^2 (\sclrp)^2 \lesssim c\epz \delta.\]
The last inequality follows by \eqref{scale},  since $\delta\le \epz \mu'\mu$.   In the same manner, using Lemma  \ref{lem:pxpyD}, one can  show   $|\mathcal E_{1,3}|\lesssim c\epz \delta$. We omit the detail.  Consequently, 
we obtain 
\[|\mathcal E_1|\lesssim c\epz\delta.\]
    
    It remains to estimate $\mathcal E_2$.  A  computation  gives
 \[
    \mathcal E_2 = \inp{x-x_\zc}{y-y_\zc}^2 +2 \inp{x-x_\zc}{y-y_\zc}(\inp{x_\zc}{y-y_\zc}+\inp{y_\zc}{x-x_\zc}).
    \]
    Using \eqref{d:xfayfb} and \eqref{e:locxy-RR'}, we note that 
    \[ |\inp{x-x_\zc}{y-y_\zc}|\lesssim \scl\sclp+\scl\sclrp+ \sclr\sclp+ \sclr\sclrp|\inp{\fa_{\!\perp}}{\fa'_{\!\perp}}|\lesssim c\epz\delta.\] 
    The last inequality follows from \eqref{scale} and \eqref{e:anglefafb0}, since $|\inp{\fa_{\!\perp}}{\fa'_{\!\perp}}|=|\inp{\fa}{\fb}|$.  Combining this and the above identity yields $|\mathcal E_2|\lesssim c\epz \delta$.
\end{proof}

In order to decompose the regions $\fS_\delta^\kappa$ in a suitable manner for our purpose, it will be necessary to understand how the vectors $\fa(x,y), \fb(x,y)$ behave as $x, y$ vary within the set $\{|\cD(x,y)|\sim \delta\}$. To this end, we show the following.

\begin{lem}\label{lem:directionsofab}
    Let $(x_\zc, y_\zc),\ (x,y)\in \cB\times\cB'$ and $\dels\le \delta\le \delt$. Suppose that $|\cD(x_\zc, y_\zc)|\le  4\delta$. Then, the following hold.
    \begin{itemize}[leftmargin=6.5mm]
    \setlength\itemsep{0.4em}
        \item [i)] If $|\cD(x,y_\zc)|\le 4 \delta$ and $|\inp{\fa_{\!\perp}(x_\zc, y_\zc)}{x_\zc - x}|\ge C\delta/ \mu'$ with a sufficiently  large constant $C>0$, then we have 
        \Be 
        \label{nx0} 
        |\fa(x_\zc, y_\zc) - \fa(x,y_\zc)|\sim  \mu^{-\frac{3}{2}}(\mu')^{\frac{1}{2}} |\inp{\fa_{\!\perp}(x_\zc, y_\zc)}{x_\zc - x}|. 
        \Ee       
        \item [ii)] If $|\cD(x_\zc,y)| \le 4 \delta$ and $|\inp{\fa'_{\!\perp}(x_\zc, y_\zc)}{y_\zc - y}|\ge C\delta \mu^{-3/2}  (\mu')^{1/2} $ with a sufficiently  large constant  $C>0$,  then we have
        \Be\label{nx1}
        |\fb(x_\zc, y_\zc) - \fb(x_\zc,y)|\sim \mu^{-1}|\inp{\fa'_{\!\perp}(x_\zc, y_\zc)}{y_\zc-y}|.
        \Ee
    \end{itemize}
\end{lem}
\begin{proof}  We keep using the same notation as in the proof of Lemma \ref{lem:locD-strong} (see \eqref{nnn} and \eqref{d:xfayfb}). 
    Before proving the lemma, we first observe 
    \begin{align}
        \label{i:expDbyaaperp}
            \cD(x,y_\zc) &= \cD(x_\zc,y_\zc)+a(x) c_\fa + b(x)c_{\fa_{\!\perp}}^2, 
            \\
            \label{i:expDbyaaperp1}
            \cD(x_\zc, y) &= \cD(x_\zc,y_\zc)+a'(y) c_{\fb} + b'(y) c_{\fa'_{\!\perp}}^2,
    \end{align}
    with  $a$, $b$, $a'$, and $b'$ satisfying 
    \begin{align}
    \label{Dbyaaperp}
        |a(x)|&\sim (\mu\mu')^{\frac{1}{2}},  \hspace{-140pt}  &&\ |b(x)|\sim \frac{\mu'}{\mu},\\
    \label{Dbyaaperp1}
         |a'(y)|&\sim \mu, && |b'(y)|\sim 1.
    \end{align}
	   Indeed, to see this, note that $\cD(x,y)$ is a polynomial of degree $2$ in $x$. Thus, we may write 
		    \[ \cD(x,\yo)= \cD(x_\zc, y_\zc)+ \inp{\partial_x\cD(x_\zc,y_\zc)}{x-x_\zc} + \frac{1}{2} \inp{ x-x_\zc}{\partial_x\partial_x^\intercal\cD(x_\zc,y_\zc) (x-x_\zc)}.\] 
    Combining this with \eqref{d:xfayfb},  we have $ \cD(x,y_\zc)$ in the form of \eqref{i:expDbyaaperp} with 
    \begin{align*}
     a(x) = \inp{\partial_x\cD(x_\zc,y_\zc)}{\fa} &+   c_{\fa_{\!\perp}}\inp{\partial_x\partial_x^\intercal \cD(x_\zc,y_\zc)\fa}{\fa_{\!\perp}}+ \frac{c_{\fa}\inp{\partial_x\partial_x^\intercal \cD(x_\zc,y_\zc)\fa}{\fa}}{2},\\
      \quad b(x)&=\frac{\inp{\partial_x\partial_x^\intercal \cD(x_\zc,y_\zc)\fa_{\!\perp}}{\fa_{\!\perp}}}{2}.
    \end{align*}
	   Using \eqref{e:sizeofxDyD}, Lemma \ref{lem:size}, and Lemma \ref{lem:pxpxD}, we see that $a$ and $b$ satisfy \eqref{Dbyaaperp}. The identity \eqref{i:expDbyaaperp1} can be shown in the same manner. We omit the detail. 

    We first prove \eqref{nx0}.    Lemma \ref{lem:directionab} tells that the distance between $\fa$ and $\fa(x,y_\zc)$ is $O((\mu'/\mu)^{1/2})\ll 1$. Thus,  we have
    \[ |\fa -\fa(x,y_\zc)| \sim |\inp{\fa_{\!\perp}}{\fa(x,y_\zc)}|.\] 
   Consequently,    \eqref{nx0} follows by  \eqref{e:sizeofxDyD}  if we show 
\Be 
\label{nx0'}
| \inp{\fa_{\!\perp}}{\partial_x \cD(x,y_\zc)} | \sim  |c_{\fa_{\!\perp}}| \mu'\mu^{-1}
\Ee
provided that  $|c_{\fa_{\!\perp}}|\ge C\delta(\mu')^{-1}$ for a large constant $C>0$.  
      Since  $\inp{\fa_{\!\perp}}{\partial_x \cD(\cdot,y_\zc)}$  is a polynomial of degree $1$, 
     $ \inp{\fa_{\!\perp}}{\partial_x \cD(x,y_\zc)}= \inp{\fa_{\!\perp}}{\partial_x \cD(x_\zc,y_\zc)}+ \inp{\fa_{\!\perp}}{\partial_x\partial_x^\intercal \cD(x_\zc,y_\zc)(x-x_\zc)}$. Recalling \eqref{nnn},  we have     
  \Be 
  \label{npp}
            \inp{\fa_{\!\perp}}{\partial_x \cD(x,y_\zc)} ={c_\fa \inp{\fa_{\!\perp}}{\partial_x\partial_x^\intercal\cD(x_\zc,y_\zc)\fa} +c_{\fa_{\!\perp}}\inp{\fa_{\!\perp}}{\partial_x\partial_x^\intercal\cD(x_\zc,y_\zc)\fa_{\!\perp}}}.
\Ee

 From \eqref{i:expDbyaaperp}  and \eqref{Dbyaaperp}, we have \[ c_\fa = -\frac{b(x)}{a(x)}c_{\fa_{\!\perp}}^2 + O(\delta(\mu\mu')^{-\frac{1}{2}}),\] since $|\cD(x_\zc, y_\zc)|\le  4\delta$ and $|\cD(x, y_\zc)|\le  4\delta$.  Combining this, 
  Lemma \ref{lem:pxpxD}, and \eqref{Dbyaaperp}, we have
 \[  |{c_\fa \inp{\fa_{\!\perp}}{\partial_x\partial_x^\intercal\cD(x_\zc,y_\zc)\fa}}|\lesssim    c_{\fa_{\!\perp}}^2 \frac{\mu'}{\mu^2}+  O(\delta\mu^{-1}).  \]
On the other hand, by  Lemma \ref{lem:pxpxD} we have $|c_{\fa_{\!\perp}}\inp{\fa_{\!\perp}}{\partial_x\partial_x^\intercal\cD(x_\zc,y_\zc)\fa_{\!\perp}}|\sim   |c_{\fa_{\!\perp}}| \mu'/\mu$.  Note  $|c_{\fa_{\!\perp}}| \le |x-x_\zc|\ll \mu$. Since $|c_{\fa_{\!\perp}}|\ge C\delta(\mu')^{-1}$ for a large constant $C>0$, we obtain \eqref{nx0'} as desired. 
    
    The proof of  \eqref{nx1}  is similar to that of \eqref{nx0}.  As before, note  $|\inp{\fa'_{\!\perp}}{\fb(x_\zc,y)}|\sim |\fb-\fb(x_\zc,y)|$. Thus, thanks to  \eqref{e:sizeofxDyD'},  it suffices to show 
    \Be 
\label{nx1'}
| \inp{\fa'_{\!\perp}}{\partial_y \cD(x_\zc,y)} | \sim    |c_{\fa'_{\!\perp}}|
\Ee
provided that $|c_{\fa'_{\!\perp}}|\ge C\delta \mu^{-3/2}  (\mu')^{1/2} $ with a large $C>0$.
   Using \eqref{i:expDbyaaperp1} as before, we get  $c_\fb = -\bigl(b'(y)/a'(y)\bigr)c_{\fa'_{\!\perp}}^2 + O(\delta\mu^{-1}).$
       Set $r=|c_{\fa'_{\!\perp}}|$. Then,  we have $|c_\fb|\lesssim r^2\mu^{-1}+\delta\mu^{-1}$ and  $r\lesssim\epz(\mu\mu')^{1/2}$ by Lemma \ref{lem:size}.
      Hence,  the assumed lower bound on $r$ and  \eqref{e:quadDall2}  yield 
      \[
      \left|c_\fb\inp{\fa'_{\!\perp}}{\partial_y\partial_y^\intercal\cD(x_\zc,y_\zc)\fb}\right|
      \lesssim \left(\epz \frac{\mu'}{\mu}+C^{-1}\right)r\ll r.
      \]
     Combining this with the identity 
            \[   \inp{\fa'_{\!\perp}}{\partial_y \cD(x_\zc,y)} ={c_\fb \inp{\fa'_{\!\perp}}{\partial_y\partial_y^\intercal\cD(x_\zc,y_\zc)\fb} +c_{\fa'_{\!\perp}}\inp{\fa'_{\!\perp}}{\partial_y\partial_y^\intercal\cD(x_\zc,y_\zc)\fa'_{\!\perp}}},\]
            we obtain \eqref{nx1'}, since     $|c_{\fa'_{\!\perp}}\inp{\fa'_{\!\perp}}{\partial_y\partial_y^\intercal\cD(x_\zc,y_\zc)\fa'_{\!\perp}}|\sim r$ by \eqref{e:quadDall1}. 
      \end{proof}

\subsection{Stability of the normal vectors on tangential rectangles} We close this section by proving the following, which shows  the unit vectors $\fa(x,y)$ and $\fb(x,y)$ are contained in circular arcs of length $\delta^{1/2}\mu^{-1}$ when $(x,y)\in R_{\yo} (\xo) \times R'_\xo(\yo)$.

\begin{lem}\label{lem:locfab}
Let \(\dels\leq\delta\leq\delt\) be dyadic.
Let $(x_\zc,y_\zc)\in \cB\times\cB'$  and let $R:= R_\yo(x_\zc)$ and  $R':=R'_\xo(y_\zc)$.  Suppose that $|\cD(x_\zc, y_\zc)|\le  4\delta$. Then, for $(x,y)\in R\times R'$  there  is a constant 
    $C>0$ such that 
    \begin{align}
    \label{fafa}
        |\fa(x,y) - \fa(x_\zc,y_\zc)| \le C\epz \delta^{\frac{1}{2}}\mu^{-1},
        \\
           \label{fafa'}
         |\fb(x,y) - \fb(x_\zc,y_\zc)|\le C\epz \delta^{\frac{1}{2}}\mu^{-1}.
    \end{align}
    \end{lem}

\begin{proof} As before, we continue  using the same notation as in the proof of Lemma \ref{lem:locD-strong} (see \eqref{nnn}, \eqref{d:xfayfb}).  
   We first prove the estimate only for the unit vector $\fa$.  
   
   As in the proof of Lemma \ref{lem:directionsofab}, since $ |\fa(x,y) - \fa|\sim | \inp{\fa_{\!\perp}}{\fa(x,y)}|$,  it is sufficient to show that 
    \begin{equation}\label{eq:dist}
        |\inp{\fa(x,y)}{\fa_{\!\perp}}  |\lesssim \epz \delta^{\frac{1}{2}}\mu^{-1},\quad (x,y)\in R\times R'.
    \end{equation}
    Note that the left hand side of \eqref{eq:dist} equals  $|\partial_x \cD(x,y)|^{-1}| F (x,y)|$ with a polynomial  $ F(x,y)= \inp{\partial_x\cD(x,y)}{\fa_{\!\perp}}$. Thus, to verify \eqref{eq:dist}, it suffices to show 
    \Be 
    \label{Fz}
    | F(z)|\lesssim \epz \biggl(\frac{\delta\mu'}{\mu}\biggr)^{\frac{1}{2}}
    \Ee
    for every $z=(x,y)\in R\times R'.$
    Since $ F(z)$ is a polynomial of degree $3$ and $F(z_\circ)=0$,
    \begin{align}\label{e:tem}
       F(z) = \inp{\partial_z F(z_\circ)}{z-z_\circ} +\sum_{|\alpha|=2,3} \frac{\partial_z^\alpha F(z_\circ)(z-z_\circ)^\alpha}{\alpha!},
    \end{align}
    where $z_\circ=(x_\zc,y_\zc)$. Using  \eqref{d:xfayfb}, we write the first-order terms  as follows: 
    \Be 
    \label{pzF}
    \begin{aligned}
    \inp{\partial_z F(z_\circ)}{&z-z_\circ}=\inp{\partial_x\partial_x^\intercal \cD(z_\circ) \fa_{\!\perp}}{\fa} c_\fa + \inp{\partial_x\partial_x^\intercal \cD(z_\circ) \fa_{\!\perp}}{\fa_{\!\perp}} c_{\fa_{\!\perp}}\\[3pt]
    & \quad+ \inp{\partial_y\partial_x^\intercal \cD(z_\circ) \fa_{\!\perp}}{\fb} c_{\fb} + \inp{\partial_y\partial_x^\intercal \cD(z_\circ) \fa_{\!\perp}}{\fa'_{\!\perp}} c_{\fa'_{\!\perp}}.
    \end{aligned}
    \Ee
       Combining this with \eqref{e:locxy-RR'} (see, also, \eqref{scale}), Lemma \ref{lem:pxpxD}, and Lemma \ref{lem:pxpyD}, we obtain $|\inp{\partial_z F(z_\circ)}{z-z_\circ}|\lesssim \epz (\delta\mu'/\mu)^{1/2}$ for $z\in R\times R'$.

    To estimate the higher-order terms in \eqref{e:tem}, we use the identity  $F(x,y)=2(\inp{x}{y}\inp{y}{\fa_{\!\perp}}-\inp{x}{\fa_{\!\perp}})$, which follows from \eqref{pxpy0}. Then, after discarding the first order terms,\footnote{$F(x,y)=2(\inp{x-x_\zc+x_\zc}{y-y_\zc+y_\zc}\inp{y-y_\zc+ y_\zc}{\fa_{\!\perp}}-\inp{x-x_\zc+ x_\zc}{\fa_{\!\perp}})$} the sum of the higher-order terms in \eqref{e:tem} are written as two times 
\Be
\label{in-}    \begin{aligned}
        \inp{\fa_{\!\perp}}{y-y_\zc}&\big(\inp{y_\zc}{x-x_\zc} +  \inp{x_\zc}{y-y_\zc}+ \inp{x-x_\zc}{y-y_\zc}\big) 
        \\ &\qquad+ \inp{y_\zc}{\fa_{\!\perp}}\inp{x-x_\zc}{y-y_\zc}.
    \end{aligned}
    \Ee
   Since $ |\inp{y_\zc}{\fa_{\!\perp}}|\lesssim 1$, the desired inequality \eqref{Fz} follows if we show
   \Be
   \label{inin}
    \begin{aligned}
        &|\inp{\fa_{\!\perp}}{y-y_\zc}| \lesssim \epz \frac{\delta}{\mu},  & & |\inp{y_\zc}{x-x_\zc}| \lesssim \epz \biggl(\frac{\delta\mu}{\mu'}\biggr)^{\frac{1}{2}}, 
        \\
         &|\inp{x_\zc}{y-y_\zc}| \lesssim \epz \biggl(\frac{\delta\mu'}{\mu}\biggr)^{\frac{1}{2}},  & &|\inp{x-x_\zc}{y-y_\zc}| \lesssim \epz \delta.
    \end{aligned}
    \Ee
	  To show the first, using \eqref{d:xfayfb}, we have  $\inp{\fa_{\!\perp}}{y-y_\zc}= c_{\fb}  \inp{\fa_{\!\perp}}\fb +c_{\fa'_{\!\perp}} 
  \inp{\fa_{\!\perp}}{\fa'_{\!\perp}}$. Thus,  \eqref{e:locxy-RR'} and Lemma \ref{lem:alp} give
  $ |\inp{\fa_{\!\perp}}{y-y_\zc}|\le  |c_{\fb}| +  |c_{\fa'_{\!\perp}} |  |\inp{\fa}{\fb}| \lesssim  \epz \delta/\mu.$ 
  To show the second,  we write  $y_\zc= y_{\zc,1} e_1+ y_{\zc,2} e_2$. Then, combining  this with the first identity in \eqref{d:xfayfb} gives 
 \[  \inp{y_\zc}{x-x_\zc}= y_{\zc,1} c_{\fa}  \inp {e_1}\fa+ y_{\zc,2} c_{\fa}  \inp {e_2}\fa+ y_{\zc,1} c_{\fa_{\!\perp}} \inp {e_1} {\fa_{\!\perp}}+y_{\zc,2} c_{\fa_{\!\perp}} \inp {e_2} {\fa_{\!\perp}}.  \] 
 Note from \eqref{bb} that $|y_{\zc,1}|\sim 1$ and $|y_{\zc,2} |\lesssim (\mu\mu')^{1/2}$. Thus,  the desired inequality follows by Lemma \ref{lem:directionab} and \eqref{e:locxy-RR'}. 
  The third can be shown similarly. To show the last inequality, using \eqref{d:xfayfb}, we have 
  \[ \inp{x-x_\zc}{y-y_\zc}=  c_{\fa}c_{\fb}  \inp \fa\fb+  c_{\fa}c_{{\fa'_{\!\perp}}}  \inp\fa{\fa'_{\!\perp}} + c_{\fa_{\!\perp}} c_{\fb}  \inp{\fa_{\!\perp}}\fb+c_{\fa_{\!\perp}} c_{{\fa'_{\!\perp}}}  \inp{\fa_{\!\perp}}{\fa'_{\!\perp}}.\]
  Hence, the last inequality follows from \eqref{e:locxy-RR'} (\eqref{scale}) and Lemma \ref{lem:alp}. 
  
 We now address the estimate \eqref{fafa'}, whose proof proceeds in
	the same manner as before. Thus, we shall be brief.  As before, by \eqref{e:sizeofxDyD'}  it is sufficient to show that
\Be \label{Gz} |G(z)|\lesssim \epz \delta^{\frac{1}{2}}\Ee
 for $z=(x,y)\in R\times R'$, where
$
  G(x,y)=\bigl\langle \partial_y \cD(x,y),\, \fa'_{\!\perp} \bigr\rangle.
$
Since $G$ is a polynomial of degree 3 and $G(z_\circ)=0$,  we may write 
\[
  G(z)
	  = \bigl\langle \partial_z G(z_\circ),\, z-z_\circ \bigr\rangle
  + \sum_{|\alpha|=2,3}
	    \frac{\partial_z^\alpha G(z_\circ)}{\alpha!}(z-z_\circ)^\alpha.
\]
Using \eqref{d:xfayfb} (cf. \eqref{pzF}), we have
  \[ 
    \begin{aligned}
    \inp{\partial_z G(z_\circ)}{&z-z_\circ}=\inp{\partial_x\partial_y^\intercal \cD(z_\circ) \fbp}{\fa} c_\fa + \inp{\partial_x\partial_y^\intercal \cD(z_\circ) \fbp}{\fa_{\!\perp}} c_{\fa_{\!\perp}}\\[3pt]
    & \quad+ \inp{\partial_y\partial_y^\intercal \cD(z_\circ) \fbp}{\fb} c_{\fb} + \inp{\partial_y\partial_y^\intercal \cD(z_\circ) \fbp}{\fa'_{\!\perp}} c_{\fa'_{\!\perp}}.
    \end{aligned}
    \]
Combining this with \eqref{e:locxy-RR'} (see, also, \eqref{scale}), Lemma \ref{lem:pxpxD}, and Lemma \ref{lem:pxpyD},  we obtain 
$ 
	  \bigl|\bigl\langle \partial_z G(z_\circ),\, z-z_\circ \bigr\rangle\bigr|
  \lesssim \epz \delta^{1/2}.
$

Now, from  \eqref{pxpy1} we note  $G(x,y)=2(\inp{x}{y}\inp{x}{\fbp}-\inp{y}{\fbp})$. Thus, discarding the first order terms, we see the sum of the higher-order terms is equal to  2 times 
\begin{align*}
  \bigl\langle  \fbp, x-x_\zc \bigr\rangle
    \big(& \bigl\langle y_\zc, x-x_\zc \bigr\rangle
  +\bigl\langle x_\zc, y-y_\zc \bigr\rangle
     +    \bigl\langle x-x_\zc, y-y_\zc \bigr\rangle \big)
    \\
    &\qquad +\bigl\langle x_\zc, \fbp \bigr\rangle
    \bigl\langle x-x_\zc, y-y_\zc \bigr\rangle.
\end{align*}
Using \eqref{d:xfayfb}, \eqref{e:locxy-RR'}, and Lemma \ref{lem:alp}, we have
\[
|\langle \fbp, x-x_\zc\rangle|\lesssim\delta^{\frac{1}{2}}.
\]
Since $\delta\le\delt\lesssim\epz\mu\mu'$, it follows from
\eqref{inin} that all the higher-order terms above are
$O(\epz\delta^{1/2})$. Together with the estimate for the
first-order term, this proves \eqref{Gz}.
    \end{proof}

\begin{cor}\label{rem:sizefab}
Let \((x_\ast,y_\ast)\in\cB\times\cB'\). Then, for
\((x,y)\in\cB\times\cB'\),
\[
|\fa(x,y)-\fa(x_\ast,y_\ast)|
+|\fb(x,y)-\fb(x_\ast,y_\ast)|
\lesssim
\epz\biggl(\frac{\mu'}{\mu}\biggr)^{\frac12}.
\]
\end{cor}
\begin{proof}
Argue as in Lemma~\ref{lem:locfab}, using the full-rectangle bounds
from Lemma~\ref{lem:size} and
\(|\cD|\lesssim\epz\mu\mu'\).
\end{proof}
The stability estimates above provide the geometric input for the multiscale rectangle decomposition in Section~\ref{sec:decomp}.
 \section{Multiscale space--time decomposition}\label{sec:decomp}
In this section, we decompose the operator
\[
    \wt\chi_{\cB}\mathcal O_\lambda[\psi_\vartriangle]\wt\chi_{\cB'},
\]
which will play a crucial role in the proof of \eqref{e:main-fP}.

As observed in Section \ref{sec:prelim}, the behavior of the oscillatory kernel of \(\mathcal O_\lambda[\psi_\vartriangle]\) is closely related to the value of \(\cD(x,y)\). From the perspective of the van der Corput lemma, \(\cD(x,y)\) affects the decay properties of the kernel. Indeed, the identity \eqref{i:ptcP-Q} shows that \(\partial_t\cP(x,y,t)\) becomes small when \(t\) is close to \(S(x,y)\) as \(\cD(x,y)\) approaches zero. Moreover, if \(\cD(x,y)=0\), then both the first and second derivatives \(\partial_t\cP\) and \(\partial_t^2\cP\) vanish simultaneously at \(t=S(x,y)\). This connection between the discriminant \(\cD\) and the oscillatory behavior of the kernel naturally leads us to decompose the operator in the \(x,y\) variables according to the size of \(|\cD(x,y)|\).

Recalling \eqref{fsd} and \eqref{fsdc}, we write 
\begin{align*}
    \cB\times\cB' = \fS_{\dels}^\circ \cup \Big(\bigcup_{\dels \le \delta \le \delt}\fS_\delta^+\cup \fS_\delta^-\Big).
\end{align*}

\subsection{Decomposition of the dyadic level regions}\label{subsec:setsS}
The aim of this subsection is to decompose the regions $\fS_\delta^\kappa$ into almost disjoint rectangles $R_y(x)\times R_x'(y)$ given by \eqref{rx0} and \eqref{ry0}. 
Lemma \ref{lem:locfab}  tells that the vectors  $\fa$, $\fb$ are stable at the scale of $\delta^{1/2}\mu^{-1}$ within the rectangle $R_y(x)\times R_x'(y)$.  
Note both the rectangles $R_y(x)$, $R_x'(y)$ have ({\it length of shorter side/length of longer side})$\sim \delta^{1/2}\mu^{-1}$.  Such rectangles remain essentially unchanged under rotation  of  an angle $\lesssim  \delta^{1/2}\mu^{-1}$.  Thus, we discretize the possible  major directions of  the rectangles.

From Corollary \ref{rem:sizefab} we note that 
the sets $\{\fa(x,y):  x\in \cB, y\in \cB'\}$  and $\{\fb(x,y):  x\in \cB, y\in \cB'\}$ are contained in circular arcs $\mathcal A, \mathcal A'$  of length $\ll (\mu'/\mu)^{1/2}$.
Let $\Theta$, $\Theta'$ be  maximally $\epz\delta^{1/2}\mu^{-1}$-separated subsets of  $\mathcal A, \mathcal A'$, respectively.

For $\nu\in \Theta\cup \Theta'$, we denote by $\nu_\sprp$ the unit vector in $\mathbb S^1$ orthogonal to $\nu$
such that $\inp{\nu_\sprp}{e_1}>0$. Recalling \eqref{scale}, for $\nu\in \Theta$ and $\mathfrak n = (\fn_1, \fn_2)\in \Z^2$, we define 
\Be
\label{ftau}
\begin{aligned}
    \fr_{\nu,\fn} = \{x\in \R^2:\, & 0 \le \inp{x}{\nu} -\fn_1\scl < \scl, \quad  0\le \inp{x}{\nu_\sprp} - \fn_2 \sclr < \sclr\}.
\end{aligned}
\Ee

\newcommand{\sig}{\sigma_{\mathsmaller \delta}}
\newcommand{\sigr}{\sigma_{\mathsmaller \delta, \mathsmaller \perp}}
\newcommand{\sigp}{\sigma_{\mathsmaller \delta}^{\,\prime}}
\newcommand{\sigrp}{\sigma_{\mathsmaller \delta, \mathsmaller \perp}^{\,\prime}}

Similarly, for $\nu'\in \Theta'$ and $\fn\in \Z^2$, let
\Be
\label{ftaup}
\begin{aligned}
    \fr_{\nu',\fn}' = \{y\in \R^2:\, &0 \le \inp{y}{\nu'} -\fn_1\sclp < \sclp, \quad   0 \le \inp{y}{\nu'_\sprp} -\fn_2 \sclrp< \sclrp\}.
\end{aligned}
\Ee
Thus, the collections   \[ \cR_{\delta, \nu} := \{\fr_{\nu,\fn}: {\fn\in \Z^2} \}, \quad \cR_{\delta, \nu'}' := \{\fr'_{\nu',\fn}:{\fn\in \Z^2}\}, \] 
 tile $\R^2$. Clearly,  the set $\{ \fr\times \fr': \fr\in \cR_{\delta, \nu}, \, \fr'\in \cR_{\delta, \nu'}'\}$  tiles $\mathbb R^4$ for every pair $(\nu,\nu')\in \Theta\times\Theta'$.  
 Let us set
\begin{align*}
    \mathcal U_\delta = \bigcup_{\nu\in \Theta} \cR_{\delta, \nu},\qquad \mathcal U_\delta' = \bigcup_{\nu'\in \Theta'} \cR_{\delta, \nu'}'.
\end{align*}

Henceforth, unless otherwise specified, \(\nu\) and \(\nu'\) denote elements of \(\Theta\) and \(\Theta'\), respectively. For simplicity, we shall occasionally write \(\fr\) and \(\fr'\) in place of \(\fr_{\nu,\fn}\) and \(\fr'_{\nu',\fn}\), respectively, suppressing the subscripts \(\nu,\nu'\), and \(\fn\), whenever no ambiguity arises. 
If \(\fr=\fr_{\nu,\fn}\) and \(\fr'=\fr'_{\nu',\fn}\), then we denote, respectively,  
\[
    \theta(\fr)=\nu, 
    \qquad 
    \theta(\fr')=\nu'. 
\]

\begin{defn}
\label{rkd}
For $\kappa=\pm, \circ$, we denote by $\mathfrak R^\kappa_\delta$ the collection of  rectangles $\fr\times \fr' \in   \mathcal U_\delta\times  \mathcal U_\delta'$  satisfying 
\begin{align}
 \label{int-con}
 &\qquad\quad \  \fr\times \fr'\cap \fS_\delta^\kappa \cap ( \supp\wt\chi_{\cB} \times \supp\wt\chi_{\cB'}) \neq\emptyset, 
 \\[3pt]
 \label{dir-con}
 &\begin{aligned}
  \exists (x_\zc, y_\zc), (\tilde x_\zc, \tilde y_\zc)&\in\,\fr\times \fr' \cap  \fS_\delta^\kappa \ \text{ s.t. }  
  \\[-1pt] &
   |\fa(x_\zc, y_\zc)-\theta(\fr)| + |\fb(\tilde x_\zc,\tilde y_\zc) - \theta(\fr')|\le 4\epz\delta^{\frac{1}{2}}\mu^{-1}.
 \end{aligned}
\end{align} 
By Lemma \ref{lem:locfab}  we may take $(x_\zc, y_\zc)= (\tilde x_\zc, \tilde y_\zc)$ by enlarging the constant. 
\end{defn}

It is clear that the elements in  $\mathfrak R^\kappa_\delta$ cover  $\fS_\delta^\kappa \cap ( \supp\wt\chi_{\cB} \times \supp\wt\chi_{\cB'})$. Indeed, let $(x,y)\in \fS_\delta^\kappa\cap(\supp\wt\chi_{\cB}\times\supp\wt\chi_{\cB'}).$ Since $\fS_\delta^\kappa \subset \cB\times\cB'$, by Corollary \ref{rem:sizefab}, there exists $(\nu,\nu')\in \Theta \times \Theta'$ such that $|\fa(x, y)-\nu| \le \epz \delta^{1/2}\mu^{-1}$, $|\fb(x, y)-\nu'| \le \epz \delta^{1/2}\mu^{-1}$.  Since $\cR_{\delta, \nu}\times\cR_{\delta, \nu'}'$ is a tiling of $\mathbb R^4$, there exists a pair of rectangles $\fr_\nu\times \fr'_{\nu'}$ in $\cR_{\delta, \nu} \times \cR_{\delta, \nu'}'$ such that $(x,y)\in \fr_\nu\times \fr'_{\nu'}$. By  Definition \ref{rkd}, $\fr_\nu\times \fr'_{\nu'}\in \mathfrak R^\kappa_\delta$.   

As throughout the proof, constants in the covering and incidence bounds below may depend on the fixed \(\epz\), but are uniform in \(\lambda,\mu,\mu',\delta\) and the rectangles.

Let \(2\fr\) and \(2\fr'\) denote the concentric dilates of \(\fr\) and \(\fr'\), respectively, by a factor of \(2\). Moreover, by Lemma \ref{lem:locfab}, the rectangles in $\{2\fr\times 2\fr': {\fr\times \fr'\in \fR^\kappa_\delta}\}$ have $O_{\epz}(1)$ overlaps, since there are only $O_{\epz}(1)$ rectangles satisfying \eqref{dir-con}. Indeed, 
let $(x,y)\in 2\fr\times2\fr'$, $\fr\times\fr'\in \fR^\kappa_\delta$. Then,  by  Lemma \ref{lem:locfab} and \eqref{dir-con}, we have $|\fa(x,y)-\theta(\fr)|\lesssim \delta^{1/2}\mu^{-1}$.
Similarly, $|\fb(x,y)-\theta(\fr')|\lesssim \delta^{1/2}\mu^{-1}$. Since $\Theta$ and $\Theta'$ are $\epz\delta^{1/2}\mu^{-1}$-separated, there are $O_{\epz}(1)$ possible choices of $\nu=\theta(\fr)\in\Theta$ and $\nu'=\theta(\fr')\in \Theta'$.  For each $(\nu, \nu')$, there are at most four $\fr\times \fr'\in \cR_{\delta, \nu} \times \cR_{\delta, \nu'}'$. Consequently, $(x,y)$ is contained in $O_{\epz}(1)$ rectangles  $\fr\times\fr'\in \fR^\kappa_\delta$. 
  Additional properties will be discussed in what follows.
  
  Choosing $\epz>0$ and $c$ (see \eqref{rx0} and \eqref{ry0}) sufficiently small, for $\delta$ satisfying \eqref{ddd},  we have
\begin{align}\label{r:inc-rr'fR}
    \bigcup_{\fr\times \fr'\in \fR^\kappa_\delta} &2\fr\times  2\fr'\subset \Big\{(x,y)\in \cB\times\cB':  \frac{3}{4} \delta\le \kappa \cD(x,y) \le \frac{17}{4}\delta \Big\},
    \\
    \label{rr'1}
       \bigcup_{\fr\times \fr'\in \fR^\kappa_\delta} \fr\times \fr'&\supset \Big\{(x,y)\in ( \supp\wt\chi_{\cB} \times \supp\wt\chi_{\cB'}):  \delta\le \kappa \cD(x,y)  \le  4 \delta \Big\}
\end{align}
for $\kappa\in \{+,-\}$, and
\begin{align}\label{r:inc-rr'fRcirc}
    \bigcup_{\fr\times \fr'\in \fR^\circ_\delta} &2\fr\times 2\fr'\subset \Big\{(x,y)\in \cB\times\cB': |\cD(x,y)|\le  \frac{7\delta}{4}\Big\},
    \\
     \label{rr'2}
    \bigcup_{\fr\times \fr'\in \fR^\circ_\delta} \fr\times \fr'&\supset \Big\{(x,y)\in ( \supp\wt\chi_{\cB} \times \supp\wt\chi_{\cB'}): |\cD(x,y)|\le  \frac{3\delta}{2} \Big\}.
\end{align}
Indeed, since  $\mathfrak R^\kappa_\delta$ covers $\fS_\delta^\kappa \cap ( \supp\wt\chi_{\cB} \times \supp\wt\chi_{\cB'})$,  
\eqref{rr'1} and \eqref{rr'2} are clear. Recalling \eqref{suppBB},  \eqref{fsd}, and  \eqref{fsdc},    we use Lemma \ref{lem:locD-strong}, which implies $\delta- Cc\epz \delta \le |\cD(x,y)|\le 4\delta + Cc\epz \delta$ for $(x,y)\in \fr\times \fr'$ if  $\fr\times \fr'\in \fR^\kappa_\delta$.
Taking $\epz$ so small that $Cc \epz\le 1/4$ gives \eqref{r:inc-rr'fR}. Of course, \eqref{r:inc-rr'fRcirc} is clear by the same argument.

\medskip

We prove the following, which will be used in Section \ref{sec:conclude}.

\begin{lem}\label{lem:cardinal-TT'}
    Let $\kappa\in \{+,-,\circ\}$. Then the following holds for $\fr_\zc\in \mathcal U_\delta$ and $\fr'_\zc\in \mathcal U_\delta^{\prime}$. 
    There exists a constant $C>0$,   independent of $\delta$ and  $\fr_\zc,\fr_\zc'$, such that
     \begin{align}
     \label{sharp0}
     \#\big\{\fr\in \mathcal U_\delta:  \fr\times \fr_\zc'\in \fR^\kappa_\delta \big\}&\le C,
     \\
     \label{sharp0'}
     \#\big\{\fr'\in \mathcal U_\delta': \fr_\zc\times \fr'\in \fR^\kappa_\delta \big\}&\le C. 
     \end{align} 
\end{lem}

\begin{proof}
We first  claim that  
\begin{align}
\label{2nd-claim} \# \{ \nu \in \Theta:  \exists  \fr\times \fr_\zc'\in \fR^\kappa_\delta \ \text{  s.t. } \ \theta(\fr)=\nu \} \le C,
\\
\label{2nd-claim'} \# \{ \nu' \in \Theta':  \exists  \fr_\zc \times \fr'\in \fR^\kappa_\delta \ \text{  s.t. } \ \theta(\fr')=\nu' \} \le C
\end{align}
for some constant $C>0$. 
We only prove \eqref{2nd-claim}. A  symmetric argument shows   \eqref{2nd-claim'}.

 To show \eqref{2nd-claim}, let  $\nu, \bar \nu\in \Theta$ and assume that there exist 
\[\fr\times \fr_\zc', \quad \bar\fr\times \fr_\zc' \in \fR^\kappa_\delta   \ \text{  s.t. } \  \theta(\fr)=\nu,  \quad \theta(\bar\fr)=\bar\nu. \] 
Since   $\Theta$ is $\epz\delta^{1/2}\mu^{-1}$-separated,  it is sufficient for \eqref{2nd-claim} to show  that 
\Be 
\label{nub}
|\nu - \bar\nu|\le C \epz^{\frac{1}{2}} \delta^{\frac{1}{2}}\mu^{-1}
\Ee
for some constant $C>0$, independent of the choice of $\nu, \bar\nu$.
Thus, by \eqref{dir-con} and Lemma \ref{lem:locfab}, there exist two points
$(x, y)\in \fr\times \fr_\zc'$ and
$(\bar x , \bar y )\in \bar\fr\times \fr_\zc'$ such that
\begin{align}
\nonumber |\nu  - \fa(x, y)| + |\bar\nu  - \fa(\bar x , \bar y )|\le C\epz \delta^{\frac{1}{2}}\mu^{-1},  
 \\
 \label{2nd-}
  |\nu'_\zc - \fb(x, y)| + |\nu'_\zc - \fb(\bar x , \bar y )|\le  C\epz \delta^{\frac{1}{2}}\mu^{-1}
 \end{align}
 for some positive constant $C$ where $\nu'_\zc= \theta(\fr'_\zc)$.  Thus, by the first inequality and the triangle inequality  it follows that
\begin{align}
  \label{e:nu1-nu2}  |\nu  - \bar\nu | \le C\epz \delta^{\frac{1}{2}}\mu^{-1} + |\fa(x, y) - \fa(\bar x , \bar y )|.
\end{align}
We now recall 
 \eqref{e:anglefafb0}, which, with the orientation convention for the perpendicular vectors, tells that
$|\fa(z)+\fa'_\sprp(z)|\lesssim \delta \mu^{-3/2}(\mu')^{-1/2}$ for $z\in \cB\times \cB'$. Since the operation $\fa'(z)\mapsto\fa'_\sprp(z)$ is an isometry on the range under consideration, using this estimate twice gives
\begin{align*}
    |\fa(x, y) - \fa(\bar x , \bar y )|\le |\fb(x, y) - \fb(\bar x , \bar y )| + C\delta\mu^{-\frac{3}{2}}(\mu')^{-\frac{1}{2}}.
\end{align*}
From \eqref{2nd-}, note  that  $ |\fb(x, y) - \fb(\bar x , \bar y )|\le C\epz \delta^{1/2}\mu^{-1}$.  We get 
\[
    |\fa(x, y) - \fa(\bar x , \bar y )|\le  C\epz^{\frac{1}{2}} \delta^{\frac{1}{2}}\mu^{-1}, 
\]
since  $\delta \lesssim \epz \mu\mu'$. Combining  this and \eqref{e:nu1-nu2} gives \eqref{nub}.

We now show  that,  for $\nu\in \Theta$, 
\Be\label{1st-claim} \#\{\fr\times \fr_\zc'\in \fR^\kappa_\delta: \theta (\fr)=\nu \}\le C\Ee 
 with a uniform constant $C$.   Combining this and  \eqref{2nd-claim} yields \eqref{sharp0}.
 
  Let  $\fr\times \fr_\zc',\ \bar\fr\times \fr_\zc' \in \fR^\kappa_\delta$ such that $\theta(\fr)=\theta(\bar\fr)=\nu$.  Let 
 \Be 
 \label{(xy)} (x, y)\in \fr\times \fr_\zc', \quad  (\bar x ,\bar y )\in \bar\fr\times \fr_\zc'.
  \Ee
 By  the definition of $\fR^\kappa_\delta$  and Lemma \ref{lem:locfab}  we have 
  \Be
  \label{nxy} |\fa(x, y) - \nu| + |\fa(\bar x ,\bar y )-\nu|\le C\epz \delta^{\frac{1}{2}}\mu^{-1}\Ee
  for some $C>0$.   Since $\fr, \bar\fr\in \cR_{\delta, \nu}$,   \eqref{1st-claim} follows by \eqref{scale}  if we show 
 \Be
 \label{inpn}   |\inp{\nu}{x-\bar x }|\le C \delta(\mu\mu')^{-\frac{1}{2}} , \quad     |\inp{\nu_\sprp}{x-\bar x }|\le C \delta^{\frac{1}{2}}\biggl(\frac{\mu}{\mu'}\biggr)^{\frac{1}{2}}. 
 \Ee
  
    By Lemma \ref{lem:locfab} and  \eqref{nxy}, we have $|\fa(x, c(\fr_\zc')) - \fa(\bar x , c(\fr_\zc'))|\lesssim \epz\delta^{1/2}\mu^{-1}$. Here, $c(\fr_\zc')$ denotes the centers of $\fr_\zc'$.      Combining this with  \eqref{nx0}  gives \begin{align}\label{e:b-r1r2ap}
    |\inp{\fa_{\!\perp}(x,c(\fr_\zc'))}{x-\bar x }|\le C
  \epz \delta^{\frac{1}{2}}\biggl(\frac{\mu}{\mu'}\biggr)^{\frac{1}{2}}. 
\end{align} 
Indeed, if  $|\inp{\fa_{\!\perp}(x,c(\fr_\zc'))}{x-\bar x }|\ge C\delta(\mu')^{-1}$, \eqref{e:b-r1r2ap} follows from  \eqref{nx0}. Otherwise, 
note that $\delta(\mu')^{-1} \lesssim \delta^{1/2}(\mu/\mu')^{1/2}$.

On the other hand,  $|\cD(x,y) - \cD(x , c(\fr_\zc') )|,  $  $|\cD(\bar x,\bar y) - \cD(\bar x , c(\fr_\zc') )|\lesssim \epz \delta$ by Lemma \ref{lem:locD-strong}. Thus, using the  identity  \eqref{i:expDbyaaperp} with $x_\zc=\bar x$ and $y_\zc=c(\fr_\zc')$ (also, see  \eqref{nnn} and \eqref{d:xfayfb}), we obtain 
\begin{align*}
    |\cD(x,y) - \cD(\bar x , \bar y )| &\ge C_1(\mu\mu')^{\frac{1}{2}}|\inp{\fa(x, c(\fr_\zc'))}{x - \bar x }| \\
    & \qquad - C_2\mu'\mu^{-1}|\inp{\fa_{\!\perp}(x, c(\fr_\zc'))}{x - \bar x }|^2 - C_3\epz \delta 
\end{align*}
 for some constants $C_1, C_2, C_3$. Since $\fr\times \fr_\zc',\ \bar\fr\times \fr_\zc' \in \fR^\kappa_\delta$, from \eqref{r:inc-rr'fR}, \eqref{r:inc-rr'fRcirc}, and \eqref{(xy)} we have $|\cD(x,y) - \cD(\bar x , \bar y )|\lesssim \delta$. Thus, 
combining the above inequality with  \eqref{e:b-r1r2ap}, we get
\begin{align}\label{e:b-r1r2a}
    |\inp{\fa(x,c(\fr_\zc'))}{x-\bar x }|\le C\delta(\mu\mu')^{-\frac{1}{2}}
\end{align}
for a constant $C =C(d,\epz)>0$. 

Note that the rectangle determined by \eqref{e:b-r1r2ap} and \eqref{e:b-r1r2a} has ({\it length of shorter side/length of longer side})$\sim \delta^{1/2}\mu^{-1}$. Since $|\fa(x,c(\fr_\zc')) - \nu|,  |\fa_\sprp (x,c(\fr_\zc')) - \nu_\sprp|  \lesssim \epz \delta^{1/2}\mu^{-1}$ by Lemma \ref{lem:locfab},  we may  replace  $\fa_\sprp (x,c(\fr_\zc'))$ in \eqref{e:b-r1r2ap} and  $\fa(x,c(\fr_\zc')) $ in \eqref{e:b-r1r2a} with $\nu_\sprp$ and $\nu$, respectively, by enlarging the constant $C$. Therefore, \eqref{inpn}  follows. 

We now show that,  for $\nu'\in \Theta'$, 
\Be\label{1st-claim'} \#\{\fr_\zc\times \fr'\in \fR^\kappa_\delta: \theta (\fr')=\nu' \}\le C\Ee 
 with a uniform constant $C$.  This, combined with \eqref{2nd-claim'}, yields \eqref{sharp0'}.  The inequality \eqref{1st-claim'}  can be shown similarly as before, so we shall be brief. 
 
 As before, 
  let 
 \Be 
 \label{(xy)0} (x, y)\in \fr_\zc\times \fr', \quad  (\bar x ,\bar y )\in \fr_\zc\times \bar \fr', 
  \Ee
  where $\fr_\zc\times \fr',  \fr_\zc\times \bar \fr'  \in \fR^\kappa_\delta$ such that $\theta(\fr')=\theta(\bar\fr')=\nu'$. 
  For $\eqref{1st-claim'}$,  recalling \eqref{scale}, we need only to show that 
 \Be
 \label{inpn0}   |\inp{\nu'}{y-\bar y }|\le C \delta\mu^{-1} , \quad     |\inp{\nu_\sprp'}{y-\bar y }|\le C \delta^{\frac{1}{2}}. 
 \Ee
 
 By  Lemma \ref{lem:locfab}  we have  
$ |\fb(x, y) - \nu'| + |\fb(\bar x ,\bar y )-\nu'|\lesssim \epz \delta^{1/2}\mu^{-1}$. Using Lemma \ref{lem:locfab} again, we also have  $|\fa'(c(\fr_\zc),y) - \fa'(c(\fr_\zc),\bar y)|\lesssim \epz\delta^{1/2}\mu^{-1}$.  Combining this with  \eqref{nx1}   gives
 \begin{align}\label{e:b-r1r2ap0}
    |\inp{\fa_{\!\perp}'(c(\fr_\zc),y)}{y-\bar y }|\le C
  \epz \delta^{\frac{1}{2}}. 
\end{align} 
By Lemma \ref{lem:locD-strong}, \eqref{r:inc-rr'fR},
\eqref{r:inc-rr'fRcirc}, and \eqref{(xy)0}, we have
$
|\cD(c(\fr_\zc),y)-\cD(c(\fr_\zc),\bar y)|\lesssim\delta.
$
Hence, applying \eqref{i:expDbyaaperp1} with
$x_\zc=c(\fr_\zc)$ and $y_\zc=\bar y$, we obtain
\[
\delta\gtrsim
\mu|\inp{\fa'(c(\fr_\zc),y)}{y-\bar y}|
-C|\inp{\fa_{\!\perp}'(c(\fr_\zc),y)}{y-\bar y}|^2.
\]
Combining this with \eqref{e:b-r1r2ap0}, we get
\begin{align}\label{e:b-r1r2a0}
|\inp{\fa'(c(\fr_\zc),y)}{y-\bar y}|
\lesssim\delta\mu^{-1}.
\end{align}
 Now, note that the rectangle given by \eqref{e:b-r1r2ap0} and \eqref{e:b-r1r2a0} has ({\it length of shorter side/length of longer side})$\sim \delta^{1/2}\mu^{-1}$. Since $|\fa'(c(\fr_\zc),y) - \nu'|,  |\fa'_{\!\perp}(c(\fr_\zc),y) - \nu_\sprp'|  \lesssim \epz \delta^{1/2}\mu^{-1}$ by Lemma \ref{lem:locfab},  we may  replace $\fa'_{\!\perp}(c(\fr_\zc),y)$ in \eqref{e:b-r1r2ap0} and $\fa'(c(\fr_\zc),y)$ in  \eqref{e:b-r1r2a0} with $\nu_\sprp'$ and $\nu'$, respectively, by enlarging the constant $C$. Therefore, \eqref{inpn0}  follows. 
\end{proof}

\subsection{Recursive decomposition in space and time}
\label{subsec:decomp-prod}
The aim of this subsection is to obtain a dyadic decomposition of $\wt\chi_{\cB} \mathcal O_\lambda[\psi_\vartriangle]  \wt\chi_{\cB'}$ of the form
\begin{align}\label{i:decomp-main}
    \wt\chi_{\cB} \mathcal O_\lambda[\psi_\vartriangle]  \wt\chi_{\cB'}  = \sum_{\delta\in 2^{-\N}: \dels\le \delta < \delt} \big(\fP_{\lambda,\delta}^+ + \fP_{\lambda,\delta}^-+ \fP_{\lambda,\delta}^\circ \big) +  \fP_{\lambda, \dels}  \end{align}
where $\fP_{\lambda,\delta}^+$, $\fP_{\lambda,\delta}^-$, $\fP_{\lambda,\delta}^\circ$, and $\fP_{\lambda,\dels} $ are operators to  be defined  below. 

We define the operators $\fP_{\lambda,\delta}^\kappa$ in \eqref{i:decomp-main} inductively along dyadic scales in decreasing order from the operator   given in the previous step. For the purpose,  set
\Be
\label{olambda}
 \mathcal O_\lambda[ \varphi, \psi](x,y)=( \varphi \mathcal O_\lambda[ \psi]) (x,y):=\varphi(x,y)\bigg(\int \psi(x,y,t) \frac{e^{i\lambda \cP(x,y,t)}}{\sqrt\mu}dt\bigg).
\Ee

We call an operator 
\Be
\label{def:fpl}    \fP_{\lambda,\delta} =   \mathcal O_\lambda[ \varphi_\delta, \psi_\delta] ,
\Ee
admissible at scale $\delta$ if the smooth functions $\varphi_\delta$ and $\psi_{\delta}$ satisfy 
\begin{align}
\label{r:suppcon-vphi}
    \supp(\varphi_\delta )\subset \{(x,y)\in \cB\times\cB': |\cD(x,y)|\le 2\delta\}, 
    \\[2pt]
    \label{r:suppcon-psi}
    \psi_{\delta}(x,y,t) = 0 \  \  \text{if} \ \ |t-S (x,y)|>  2^{3}C_\ast\delta^{\frac{1}{2}}\mu^{-\frac{1}{2}}.
\end{align}

Let $\dels\le\delta<\delt$, and assume that an admissible operator
$\fP_{\lambda,2\delta}$ at scale $2\delta$ is given. We define the next generation operators
$\fP_{\lambda,\delta}$ and $\fP_{\lambda,\delta}^\kappa$,
$\kappa=+,-,\circ$, out of  $\fP_{\lambda,2\delta}$, as follows. 
At each step, the two signed spatial pieces and the central piece away from \(t=S(x,y)\) are retained at scale \(\delta\), while only the central piece near \(t=S(x,y)\) forms the admissible remainder passed to the next dyadic scale.

 For $\fr\in \cR_{\delta, \nu}$, let $\varphi_\fr$ be a smooth nonnegative function adapted to $2\fr$ such that
 $\varphi_\fr\ge c_0$ on $\fr$ for a  constant $c_0>0$, and
$\sum_{\fr\in \cR_{\delta, \nu}}\varphi_\fr \equiv 1$ on $\R^2$, and 
\begin{align}\label{e:bvarphifr}
    |\inp{\nu}{\partial_x}^a\inp{\nu_\sprp}{\partial_x}^b \varphi_\fr(x)|\lesssim \scl^{-a}\sclr^{-b}
\end{align}
for  $x\in\R^2$ and $a,b\in \N_0$. Similarly, for $\fr'\in \cR_{\delta, \nu'}'$, let $\varphi_{\fr'}$ be a nonnegative smooth function adapted to $2\fr'$
 such that
$\varphi_{\fr'}\ge c_0$ on $\fr'$, $\sum_{\fr'\in \cR_{\delta, \nu'}'}\varphi_{\fr'} \equiv 1$ and
\begin{align}\label{e:bvarphifr'}
    |\inp{\nu'}{\partial_y}^a\inp{\nu'_\sprp}{\partial_y}^b \varphi_{\fr'}(y)|\lesssim (\sclp)^{-a}(\sclrp)^{-b}
\end{align}
for $y\in\R^2$ and $a,b\in \N_0$. 
We now define, for $\fr\times \fr'\in \fR_{\delta}^\kappa$ 
\begin{align}\label{pkxy}
\varphi_{\delta,\fr,\fr'}^\kappa (x,y)
   =\frac{\varphi_{2\delta}(x,y)\varphi_\fr(x)\varphi_{\fr'}(y)}
   {\displaystyle\sum_{\kappa_0=\pm,\circ}
    \sum_{\bar\fr\times\bar\fr'\in\fR_\delta^{\kappa_0}}
    \varphi_{\bar\fr}(x)\varphi_{\bar\fr'}(y)}.
\end{align}
Recall that the sets $2\fr\times 2\fr'$, with $\fr\times\fr'\in \fR_{\delta}^\kappa$, are boundedly overlapping. Moreover, thanks to \eqref{r:suppcon-vphi}, \eqref{rr'1}, and \eqref{rr'2}, the denominator does not vanish on $\supp\varphi_{2\delta}$. Thus, the cutoff functions $\varphi_{\delta,\fr,\fr'}^\kappa$ are well defined.

Let us set
\begin{align}
\label{i:decomp-varp*del1}
     \varphi_{\delta}= \sum_{\fr\times \fr'\in \fR_{\delta}^\circ} \varphi_{\delta,\fr,\fr'}^\circ, \qquad
     \varphi_{\delta}^\pm= \sum_{\fr\times \fr'\in \fR_{\delta}^\pm} \varphi_{\delta,\fr,\fr'}^\pm.
\end{align}
Hence, we have
\Be
\label{i:decomp-varp*del10}
\varphi_{2\delta}(x,y)=\varphi_{\delta}(x,y)+\varphi_{\delta}^+(x,y)+\varphi_{\delta}^-(x,y).
\Ee
By \eqref{r:inc-rr'fRcirc}, it is clear that $\varphi_{\delta}$ satisfies \eqref{r:suppcon-vphi}.
By \eqref{r:inc-rr'fR}, we also have
\begin{align}
\label{r:supp-vphipm}
    \supp(\varphi_{\delta}^\pm)\subset \Big\{(x,y)\in \cB\times\cB': \frac{3}{4}\delta\le \pm \cD(x,y)\le \frac{17}{4}\delta\Big\}.
\end{align}
We define the operators $\fP_{\lambda,\delta}^+$ and $\fP_{\lambda,\delta}^-$ by
\begin{align}
\label{Oldpm}
    \fP_{\lambda,\delta}^\pm(x,y)=\mathcal O_\lambda[\varphi_{\delta}^\pm,\psi_{2\delta}](x,y).
\end{align}
Next, we make decomposition in the $t$-variable. Define $\psi_{\delta},\psi_{\delta}^\circ:\cB\times\cB'\times\R\to\R$ by
\begin{align}
\label{psi0}
    \psi_{\delta}(x,y,t)&=\psi_{2\delta}(x,y,t)\eta\Big(\frac{t-S(x,y)}{2^2C_\ast\delta^{\frac{1}{2}}\mu^{-\frac{1}{2}}}\Big),
    \\
\label{psi1}
    \psi_{\delta}^\circ(x,y,t)&=\psi_{2\delta}(x,y,t)\Big(1-\eta\Big(\frac{t-S(x,y)}{2^2C_\ast\delta^{\frac{1}{2}}\mu^{-\frac{1}{2}}}\Big)\Big).
\end{align}
Thus, $\psi_{\delta}$ satisfies \eqref{r:suppcon-psi}. Since
$\varphi_{\delta}$ satisfies \eqref{r:suppcon-vphi}, the bound
\eqref{e:simpleobv} holds whenever
\[
\varphi_{\delta}(x,y)\psi_{\delta}^\circ(x,y,t)\neq0.
\]
Now, define the operators $\fP_{\lambda,\delta}^\circ$ and $\fP_{\lambda,\delta}$ by
\begin{align}
\label{Oldc}
   \fP_{\lambda,\delta}^\circ&=\mathcal O_\lambda[\varphi_{\delta},\psi_{\delta}^\circ],
   \\
\label{Old}
   \fP_{\lambda,\delta}&=\mathcal O_\lambda[\varphi_{\delta},\psi_{\delta}].
\end{align}
Putting all together with \eqref{i:decomp-varp*del10} gives
\begin{align}
\label{iterat}
    \fP_{\lambda,2\delta}=\sum_{\kappa\in\{+,-,\circ\}}\fP_{\lambda,\delta}^\kappa+\fP_{\lambda,\delta}.
\end{align}

Recalling \eqref{ola}, let us set
\Be
\label{iter}
\varphi_{\delt}(x,y) = \wt\chi_\cB(x)\wt\chi_{\cB'}(y),\qquad \psi_{\delt}(x,y,t) = \psi_\vartriangle(x,y,t). 
\Ee
Then, we have 
\[ \wt\chi_{\cB} \mathcal O_\lambda[\psi_\vartriangle]  \wt\chi_{\cB'} = \mathcal O_\lambda[\varphi_{\delt}, \psi_\delt].\] 
 Note also from \eqref{r:cond-cDh}, \eqref{d:top}, and \eqref{psi-tri} that these functions satisfy \eqref{r:suppcon-vphi} and \eqref{r:suppcon-psi} with $\delta=\delt$.
 By \eqref{mm}, \eqref{d:top}, and the choice of \(\dels\) in \eqref{d:top&bottom}, we have \(2\dels\le\delt\); hence the first step \(\delta=\delt/2\) is available and the recursion produces the terminal cutoffs \(\varphi_{\dels}\) and \(\psi_{\dels}\).
 Starting with $2\delta=\delt$ and applying the identity \eqref{iterat} iteratively down to $\delta=\dels$, we obtain  the desired decomposition \eqref{i:decomp-main}. 
 
 In the following sections, we prove  Proposition \ref{main-p} by obtaining estimates for  the operators appearing in \eqref{i:decomp-main}.  

\subsection{Decomposition into tangential rectangles } We make additional decompositions of the operators $\fP_{\lambda,\delta}^\kappa$, $\fP_{\lambda,\dels}$ so that the decomposed pieces are localized in some rectangles $\fr\times \fr'$.

Recall  \eqref{i:decomp-varp*del1} and \eqref{psi1}.  
For $\dels\le \delta < \delt$,  let
\begin{align}
\label{Oldpmt}   \fP_{\delta, \fr,\fr'}^\pm &= \mathcal O_\lambda[ \varphi_{\delta, \fr,\fr'}^\pm,  \psi_{2\delta}],  \quad \fr\times \fr'\in \fR^\pm_{\delta}
\\
 \label{Oldct}  \fP_{\delta, \fr,\fr'}^\circ &= \mathcal O_\lambda[ \varphi_{\delta, \fr,\fr'}^\circ,  \psi_{\delta}^\circ], \quad \fr\times \fr'\in \fR^\circ_{\delta}. 
 \end{align} 
Moreover, for $\fr\times \fr'\in \fR_{\dels}^\circ$ define
\begin{align}
\label{Olds}     \fP_{\dels, \fr,\fr'} = \mathcal O_\lambda[ \varphi_{\dels, \fr,\fr'}^\circ,  \psi_{\dels}].
\end{align}
Consequently,  from \eqref{i:decomp-varp*del1} it is clear that 
\begin{align}
\label{pdk}
\fP_{\lambda,\delta}^\kappa &=\sum_{\fr\times \fr'\in \fR_{\delta}^\kappa}\fP_{\delta,\fr,\fr'}^\kappa, \quad \kappa= \pm, \circ,
\\
\label{pds}
  \fP_{\lambda, \dels} &=\sum_{\fr\times \fr'\in \fR_{\dels}^\circ}\fP_{\dels,\fr,\fr'} . 
  \end{align}

Before  closing  this section, we prove bounds on derivatives of cutoff functions  $\psi_\delta$, $\psi^\circ_\delta$, and $\varphi_{\delta, \fr, \fr'} ^\kappa$. 
Those bounds will be used to estimate kernels appearing in $TT^*$ and $T^*T$ forms.  

For $\dels\le \delta \le \delt$ and $\nu, \nu'\in \mathbb S^1$, let $\mathcal L_{\delta,\nu}$ and $\mathcal L_{\delta,\nu'}$ denote differential operators given by
\begin{align}\label{d:diff-L}
    \mathcal L_{\delta,\nu} f =  \scl  \inp{\nu}{\partial_x f},\quad \mathcal L_{\delta,\nu'} g =\sclp\inp{\nu'}{\partial_y g}, 
\end{align}
where $\sigma$ and $\sigma'$  are given by   \eqref{scale}.

\begin{lem}\label{lem:bd-derivpsi}   Let $\dels\le\delta<\delt$, and let $\nu, \nu'\in \mathbb S^1$ satisfy $|\inp{\nu}{e_1}|,\,|\inp{\nu'}{e_2}|\le c(\mu'/\mu)^{1/2}$ for a constant $c>0$ $(cf.\, \eqref{e:direction}\,  and\, \eqref{e:direction'})$. Then, for every $ M, N\in\N_0$, there exists a constant $C_{M,N}>0$,  independent of $\delta$, such that
     \begin{equation} 
     \label{lmln}
        |\mathcal L^M\partial_t^N \psi_\delta(x,y,t)| \le C_{M,N} (\delta^{\frac{1}{2}}\mu^{-1})^M(\delta^{-\frac{1}{2}}\mu^{\frac{1}{2}})^N
    \end{equation}
       for  $(x,y, t)\in \cB\times\cB'\times \R$ where $\cL$ denotes either $\cL_{\delta,\nu}$ or $\cL_{\delta,\nu'}$. The same bounds 
       also hold for $\psi^\circ_\delta$. 
The corresponding top-scale estimate for \(\psi_{\delt}\) is recorded in \eqref{e:bd-psi-top} below.
\end{lem}

\begin{proof}
For $\dels\le\rho\le\delt$, set 
\[
    a_\rho=\rho^{\frac{1}{2}}\mu^{-1},\qquad
    b_\rho=\rho^{-\frac{1}{2}}\mu^{\frac{1}{2}},
\]
and let $\cL_\rho$ denote either $\cL_{\rho,\nu}$ or
$\cL_{\rho,\nu'}$.  We also set
\[
    q_\rho(x,y,t)
    =\frac{t-S(x,y)}{2^2C_\ast\rho^{\frac{1}{2}}\mu^{-\frac{1}{2}}},
    \qquad
    \chi_\rho(x,y,t)=\eta(q_\rho(x,y,t)).
\]
Thus, \eqref{psi0} and \eqref{psi1} become
\begin{equation}\label{i:recursion-psid}
    \psi_\delta=\psi_{2\delta}\chi_\delta,
    \qquad
    \psi_\delta^\circ=\psi_{2\delta}(1-\chi_\delta).
\end{equation}

We first record derivative bounds for the cutoff functions.  By
Lemma \ref{lem:bdsS*} and \eqref{scale}, for every $j\ge1$,
$
    |\cL_\rho^jS(x,y)|
    \lesssim_j \rho^j\mu^{1/2-2j}.
$
Since $\rho\le\delt\lesssim\mu\mu'\lesssim\mu^2$, the chain rule
therefore gives
\begin{equation}\label{e:bd-chid}
    |\cL_\rho^M\partial_t^N\chi_\rho|
    +|\cL_\rho^M\partial_t^N(1-\chi_\rho)|
    \lesssim_{M,N}a_\rho^M b_\rho^N
\end{equation}
for all $M,N\in\N_0$.  Indeed,
$|\partial_tq_\rho|\lesssim b_\rho$, while $|\cL_\rho^jq_\rho|  \lesssim_j
    \rho^{j-1/2}\mu^{1-2j}     \lesssim_j a_\rho^j$
    for $j\ge1$.

We next verify the estimate at the initial scale.  By \eqref{psi-tri}
and \eqref{iter},
\[
    \psi_{\delt}=\psi_\mu\chi_{\delt}.
\]
Moreover, \eqref{d:top} and \eqref{mm} imply
$\delt^{1/2}\lesssim\mu$, and hence
$
    |\partial_t^N\psi_\mu(t)|
    \lesssim_N\mu^{-N/2}
    \lesssim_N b_{\delt}^N.
$
Combining this with \eqref{e:bd-chid} and the product rule gives
\begin{equation}\label{e:bd-psi-top}
    |\cL_{\delt}^M\partial_t^N\psi_{\delt}|
    \lesssim_{M,N}a_{\delt}^M b_{\delt}^N.
\end{equation}

We now prove the desired estimate \eqref{lmln}  for $\psi_\delta$ by induction on
$M+N$ and, for each fixed pair $(M,N)$, by descending induction over
the dyadic scales.  Notice from \eqref{scale} that
\[
    \cL_\delta=2^{-1}\cL_{2\delta},\qquad
    a_{2\delta}=2^{\frac{1}{2}}a_\delta,\qquad
    b_{2\delta}=2^{-\frac{1}{2}}b_\delta.
\]
Consequently, if the desired estimate holds at scale $2\delta$, then
\begin{equation}\label{e:scale-psi2d}
    |\cL_\delta^m\partial_t^n\psi_{2\delta}|
    \le
    2^{-\frac{m+n}{2}}C_{m,n}a_\delta^m b_\delta^n.
\end{equation}
Besides, taking only $t$-derivatives in
\eqref{i:recursion-psid} gives the  identity
\begin{align}\label{i:ptN-psidc}
    \partial_t^K\psi_\delta
    =\sum_{L=0}^K A_L
    \partial_t^L\psi_{2\delta}\,
    \partial_t^{K-L}\chi_\delta,
    \qquad A_L=\binom KL.
\end{align}

The case $M=N=0$ follows from \eqref{e:bd-psi-top},
\eqref{i:recursion-psid}, and $0\le\chi_\delta\le1$.  Suppose now
that $M+N\ge1$ and that the estimates have already been proved for
all lower total derivative orders.  Applying the product rule to
\eqref{i:recursion-psid}, and then using
\eqref{e:bd-chid} and \eqref{e:scale-psi2d} for $m+n< M+N$, yields
\begin{align*}
    |\cL_\delta^M\partial_t^N\psi_\delta|
    &\le
    \sum_{m=0}^M\sum_{n=0}^N
    \binom Mm\binom Nn
    |\cL_\delta^m\partial_t^n\psi_{2\delta}|\,
    |\cL_\delta^{M-m}\partial_t^{N-n}\chi_\delta|  \\
    &\le
    \big(2^{-\frac{M+N}{2}}C_{M,N}+C'_{M,N}\big)
    a_\delta^M b_\delta^N,
\end{align*}
where $C'_{M,N}$ depends only on constants belonging to lower total
derivative orders.  Since $2^{-(M+N)/2}<1$, we may choose
$C_{M,N}$, independently of $\delta$, so that
\[
    2^{-\frac{M+N}{2}}C_{M,N}+C'_{M,N}\le C_{M,N}.
\]
Starting from \eqref{e:bd-psi-top} and descending over the dyadic
scales proves the asserted estimate for $\psi_\delta$.

Finally, the second identity in \eqref{i:recursion-psid}, together
with the already established estimate for $\psi_{2\delta}$ and the
second bound in \eqref{e:bd-chid}, gives the same estimate for
$\psi_\delta^\circ$.
\end{proof}

\begin{rem}
\label{diff-proof}   From the construction of $\psi_\delta$, $\psi_\delta$ is given by a finite product of cutoff  functions. 
Considering the product, one can also prove Lemma \ref{lem:bd-derivpsi} by direct computation. 
\end{rem}

\begin{lem}\label{lem:bd-derivwtphi}
    Let $\kappa\in \{+,-,\circ\}$ and $\dels \le \delta <\delt$.    Let $\nu, \nu'\in \mathbb S^1$ satisfying $|\inp{\nu}{e_1}|,$ $|\inp{\nu'}{e_2}|\le c(\mu'/\mu)^{1/2}$ for a constant $c>0$.
    Suppose that  $\fr\times \fr'\in \fR^\kappa_\delta$. 
    Then, for every $N\in \N_0$, we have
    \begin{align}
    \begin{aligned}\label{b:wtphircrc'}
        |\mathcal L^N\varphi_{\delta,\fr, \fr'}^\kappa(x,y)|& \lesssim_N  1 ,
    \end{aligned}
    \end{align}
    for all $(x,y)\in \cB\times\cB'$, where $\cL$ denotes either $\cL_{\delta,\nu}$ or $\cL_{\delta,\nu'}$ defined by \eqref{d:diff-L}.
\end{lem}

\begin{proof}
We begin by setting 
\[
    G_\delta(x,y)
    =\sum_{\kappa'\in\{+,-,\circ\}}
      \sum_{\fr\times\fr'\in\fR_\delta^{\kappa'}}
      \varphi_\fr(x)\varphi_{\fr'}(y),
    \qquad
    G_\delta^\circ(x,y)
    =\sum_{\fr\times\fr'\in\fR_\delta^\circ}
      \varphi_\fr(x)\varphi_{\fr'}(y).
\]
By the choice of the partitions $\{\varphi_\fr\}$ and $\{\varphi_{\fr'}\}$ (see near
\eqref{e:bvarphifr} and \eqref{e:bvarphifr'}),  together with \eqref{r:suppcon-vphi}, \eqref{rr'1}, and \eqref{rr'2}, it follows that
$    G_\delta\ge c_0^2
    \quad\text{on }\supp\varphi_{2\delta}. $
Consequently, there is an open neighborhood $U_\delta$ of
$\supp\varphi_{2\delta}$ on which $G_\delta\ge c_0^2/2$. On this
neighborhood, set
\[
    \cA_\delta=G_\delta^{-1},
    \qquad
    W_\delta^\circ=\cA_\delta G_\delta^\circ.
\]
In particular,  $0\le W_\delta^\circ\le1$ on $U_\delta$.

We first record uniform derivative bounds for these functions. If
$\fr\in\cR_{\delta,\omega}$ for some $\omega\in \Theta$, then 
\eqref{e:bvarphifr}, together with $\scl\le\sclr$, gives  $
    |\cL_{\delta,\nu}^N\varphi_\fr|\lesssim_N1.
$
Indeed, this follows by writing $\nu$ in the orthonormal basis
$\{\omega,\omega_\sprp\}$. The same argument, using
$\sclp\le\sclrp$ and \eqref{e:bvarphifr'}, gives
$
    |\cL_{\delta,\nu'}^N\varphi_{\fr'}|\lesssim_N1.
$
The families $\{2\fr\times2\fr'\}$ have bounded overlap, uniformly in
$\delta$. Hence, 
\[
    |\cL_\delta^N G_\delta|
    +|\cL_\delta^N G_\delta^\circ|\lesssim_N1,
\]
where $\cL_\delta$ denotes either
$\cL_{\delta,\nu}$ or $\cL_{\delta,\nu'}$. The chain and product rules therefore yield, for $N\in\N_0$ and on
$U_\delta$,
\begin{equation}\label{b:normalizer}
    |\cL_\delta^N\cA_\delta|
    +|\cL_\delta^N W_\delta^\circ|\lesssim_N1.
\end{equation}

We next prove
\begin{equation}\label{b:spatial-recursive}
    |\cL_\delta^N\varphi_\delta|\lesssim_N1
\end{equation}
uniformly over all dyadic scales
$\dels\le\delta\le\delt$.  The bound at the top scale follows from  
$
    \varphi_{\delt}
    =\wt\chi_{\cB}\wt\chi_{\cB'}
$
together with \eqref{suppBB}, \eqref{d:top}, and the assumptions on
$\nu,\nu'$. For $\delta<\delt$, we have
\[
    \varphi_\delta=\varphi_{2\delta}W_\delta^\circ,
    \qquad
    \cL_\delta=2^{-1}\cL_{2\delta}.
\]

We now use the same induction on the derivative order, followed by
descending induction on the dyadic scale, as in the proof of
Lemma \ref{lem:bd-derivpsi}. The case $N=0$ follows from
$0\le W_\delta^\circ\le1$. For $N\ge1$, assume that the bounds for
all lower derivative orders have already been established and that
the bound of order $N$ holds at scale $2\delta$ with constant $C_N$.
Then the product rule and \eqref{b:normalizer} give
\[
    |\cL_\delta^N\varphi_\delta|
    \le 2^{-N}C_N+C_N',
\]
where $C_N'$ depends only on the bounds for lower derivative orders
and is independent of $\delta$. Since $2^{-N}<1$, we may choose
$C_N$, independently of $\delta$, so that  $
    2^{-N}C_N+C_N'\le C_N. $
Starting from $\delta=\delt$ and descending to $\delta=\dels$
proves \eqref{b:spatial-recursive}.

Finally, on $U_\delta$, \eqref{pkxy} can be written as
\[
    \varphi_{\delta,\fr,\fr'}^\kappa
    =\varphi_{2\delta}\,\cA_\delta\,
      \varphi_\fr\varphi_{\fr'}.
\]
The support of $\varphi_{2\delta}$ is compactly contained in
$U_\delta$, so the right-hand side extends smoothly by zero outside
$U_\delta$. Applying the product rule, and using
\eqref{b:normalizer}, \eqref{b:spatial-recursive},
\eqref{e:bvarphifr}, and \eqref{e:bvarphifr'}, proves
\eqref{b:wtphircrc'}.
\end{proof}
The resulting decomposition and cutoff bounds will be combined with the off-diagonal estimates of Section~\ref{sec:alortho} and the diagonal estimates of Section~\ref{sec:individual}.
 \section{Almost orthogonality for the localized pieces}\label{sec:alortho}

In order to prove Proposition \ref{main-p}, we need to obtain the estimates for the operators $\fP_{\lambda,\delta}^+$, $\fP_{\lambda,\delta}^-$, $\fP_{\lambda,\delta}^\circ$, and $\fP_{\lambda,\dels} $ appearing in \eqref{i:decomp-main} (see Propositions \ref{prop:main-components1-0} and \ref{prop:main-components2-2} below). For this purpose, 
we make use of almost orthogonality between the localized operators appearing in   \eqref{pdk} and \eqref{pds}. 
The main objective of this section is   to establish the off-diagonal estimates, which will be used for the almost orthogonality argument. 
In fact, we prove these estimates with an additional decomposition.

\subsection{Further decomposition on the input side}\label{subsec:input-refinement}
We further decompose the operator  $\fP_{\delta,\fr,\fr'}^\kappa$ to make it possible to obtain the precise diagonal estimate (see Section \ref{sec:individual}). 

Let $\fr\times\fr'\in \fR^\kappa_{\delta}$ with 
\[\fr'=  \fr_{\nu',\fn}',\] 
which is given by \eqref{ftaup} and \eqref{scale}. 
 We break $2\fr'$ into $O((\mu/\mu')^{1/2})$ smaller rectangles
of equal side length $\sim  \sclrp\big(\mu'/\mu\big)^{1/2}$ in the direction of $\nu'_\sprp$.
Choosing a constant \(c_1\sim1\), we may denote those rectangles by
\begin{align}\label{s'}
\fs= \fs(\ell,\fr')
=\Big\{y\in2\fr':0\le
\inp{y-c(\fr')}{\nu'_\sprp}
-c_1\ell\sclrp\Big(\frac{\mu'}{\mu}\Big)^{1/2}
<c_1\sclrp\Big(\frac{\mu'}{\mu}\Big)^{1/2}\Big\},
\end{align}
for $\ell\in \Z$. 
This refinement scale is where the diagonal estimate in Section~\ref{sec:individual} balances the off-diagonal summation, and it is the source of the asymmetric gain in Proposition~\ref{main-p}. 
Let $\mathfrak S(\fr')$ denote the collection of nonempty rectangles $\fs(\ell, \fr')$. 

For $\fs\in\mathfrak S(\fr')$, let $\varphi_{\fr',\fs}$ be a
nonnegative smooth cutoff function adapted to $2\fs$, chosen so that
\[
\sum_{\fs\in\mathfrak S(\fr')}\varphi_{\fr',\fs}(y)=1
\quad\text{on }2\fr'.
\]
In particular,
\begin{align}\label{e:bvarphirs}
\big|
\inp{\nu'}{\partial_y}^{a}
\inp{\nu'_\sprp}{\partial_y}^{b}
\varphi_{\fr',\fs}(y)
\big|
\lesssim_{a,b}
(\sclp)^{-a}
\biggl(
    \sclrp\Bigl(\frac{\mu'}{\mu}\Bigr)^{\frac12}
\biggr)^{-b}
\end{align}
for all $a,b\in\mathbb N_0$. For $\fs\in\mathfrak S(\fr')$, we define
an operator $\fP_{\delta,\fr,\fs}^\kappa$ by setting its kernel
\Be\label{bss}
 \fP_{\delta,\fr,\fs}^\kappa (x,y) =   \fP_{\delta,\fr,\fr'}^\kappa(x,y)\varphi_{\fr',\fs}(y).
\Ee 

Also,  for $\fr\times\fr'\in \fR^\circ_{\dels}$ and  $\fs\in  \mathfrak S(\fr')$,  
\Be\label{bsso}
    \fP_{\dels, \fr,\fs}(x,y) :=  \fP_{\dels, \fr,\fr'}(x,y) \varphi_{\fr',\fs}(y).
\Ee
Consequently, since the kernel of $\fP_{\delta,\fr,\fr'}^\kappa$ is supported in $2\fr\times 2\fr'$, we have 
\begin{align}
\label{pdts}
&\fP_{\delta, \fr,\fr'}^\kappa= \sum_{\fs\in \mathfrak S(\fr')} \fP_{\delta,\fr,\fs}^\kappa,  \quad \kappa\in \{+,-, \circ\},
\\
\label{pdts'}
&\fP_{\dels, \fr,\fr'} = \sum_{\fs\in \mathfrak S(\fr')}\fP_{\dels, \fr,\fs}.
\end{align}
We also set
\begin{align*}
    \varphi^\kappa_{\delta,\fr,\fs}(x,y)
    &= \varphi^\kappa_{\delta,\fr,\fr'}(x,y)\varphi_{\fr',\fs}(y),
    \quad \fr\times\fr'\in \fR^\kappa_{\delta},\quad
    \kappa\in\{+,-,\circ\}.
\end{align*}
For the bottom-scale operator, we write
\[
    \varphi_{\dels,\fr,\fs}
    :=\varphi^\circ_{\dels,\fr,\fs}.
\]
Then, note that 
\[ \fP_{\delta, \fr,\fs}^\pm= \mathcal O_\lambda[ \varphi^\pm_{\delta,\fr,\fs}, \psi_{2\delta}], \quad 
\fP_{\delta, \fr,\fs}^\circ= \mathcal O_\lambda[ \varphi^\circ_{\delta,\fr,\fs}, \psi_{\delta}^\circ], \quad 
\fP_{\dels, \fr,\fs}= \mathcal O_\lambda[ \varphi_{\dels,\fr,\fs}, \psi_{\dels}].\]

Let $\theta(\fr)=\nu$ and $\theta(\fr')=\nu'$.  Recall  \eqref{bss}, \eqref{bsso}. By \eqref{e:bvarphirs},
$\mathcal L_{\delta,\nu'}^N\varphi_{\fr',\fs}(y)\lesssim_N1$. Thus,
Lemma \ref{lem:bd-derivwtphi} implies that, for every $\nu, \nu'\in \mathbb S^1$ such that $|\inp{\nu}{e_1}|$, $|\inp{\nu'}{e_2}|\lesssim (\mu'/\mu)^{1/2}$, and for every $N\in\N$,
    \begin{align}
    \begin{aligned}\label{b:wtphirr's'}
        |\mathcal L_{\delta, \nu}^N \varphi_{\delta,\fr, \fs}^\kappa(x,y)| + |\mathcal L_{\delta, \nu'}^N \varphi_{\delta,\fr,\fs}^\kappa(x,y)|& \lesssim_N 1.
    \end{aligned}
    \end{align}
To obtain the bound for $\mathcal L_{\delta, \nu'}^N\varphi_{\delta,\fr, \fs}^\kappa$, we also use the fact that $\delta\mu^{-1}\lesssim \delta^{1/2}\smp\mu^{-1/2}$.

\subsection{Almost orthogonality across rectangles} 
The following proposition provides quantitative controls for
\begin{align*}
(\fP_{\delta,\fr_\zc,\fs_\zc}^\kappa)^*
    \fP_{\delta,\fr,\fs}^\kappa
\quad\text{and}\quad
\fP_{\delta,\fr_\zc,\fs_\zc}^\kappa
    (\fP_{\delta,\fr,\fs}^\kappa)^*
\end{align*}
in terms of $\nu',\nu_\zc'$ and $\nu,\nu_\zc$.

\begin{prop}\label{prop:alortho-1}
Let $\dels\le\delta<\delt$ and $\kappa\in\{+,-,\circ\}$.  Let $\fr_\zc\times \fr_\zc'\in \fR^\kappa_\delta$, $\fr\times \fr'\in \fR^\kappa_\delta$, and let  $\fs_\zc\in \mathfrak S(\fr_\zc')$, $\fs\in \mathfrak S(\fr')$.  Suppose that $\theta(\fr_\zc)=\nu_\zc$, $\theta(\fr_\zc')=\nu_\zc'$, $\theta(\fr)=\nu$,  and $\theta(\fr')=\nu'$. Then, the following hold. 
    \begin{itemize}[leftmargin=2.0em]
     \setlength{\itemsep}{2pt}
    \setlength{\topsep}{2pt}
        \item [i)] If         $|\nu_\zc'-\nu'|\gg \delta^{1/2}\mu^{-1}$, then
        \begin{align*}
            \|(\fP_{\delta,\fr_\zc,\fs_\zc}^\kappa)^*\fP_{\delta,\fr,\fs}^\kappa\|_{2\to 2}& \lesssim_N   B(\delta):=\frac{\delta^4\mu^{-\frac{7}{2}}(\mu')^{-\frac{1}{2}}}{(|\nu'-\nu'_\zc|\lambda\delta(\mu')^{-\frac{1}{2}})^{N}}, \quad \kappa=\pm, \circ. 
        \end{align*}
        \item [ii)] If          $|\nu_\zc-\nu |\gg \delta^{1/2}(\mu\mu')^{-1/2}$, then 
        \begin{align*}
            \|\fP_{\delta,\fr_\zc,\fs_\zc}^\kappa(\fP_{\delta,\fr,\fs}^\kappa)^*\|_{2\to 2}& \lesssim_N  B'(\delta):= \frac{\delta^4\mu^{-\frac{7}{2}}(\mu')^{-\frac{1}{2}}}{(|\nu_\zc - \nu |\lambda^{}\delta\mu^{-1}\smp)^{N}}, \quad \kappa=\pm, \circ.
               \end{align*}
 \item [iii)] At the bottom scale, suppose that
 $\fr_\zc\times\fr_\zc',\,\fr\times\fr'\in\fR^\circ_{\dels}$.
 If $|\nu_\zc'-\nu'|\gg\dels^{1/2}\mu^{-1}$, then
 \[
 \|(\fP_{\dels,\fr_\zc,\fs_\zc})^*
 \fP_{\dels,\fr,\fs}\|_{2\to2}\lesssim B(\dels).
 \]
 If $|\nu_\zc-\nu|\gg\dels^{1/2}(\mu\mu')^{-1/2}$, then
 \[
 \|\fP_{\dels,\fr_\zc,\fs_\zc}
 (\fP_{\dels,\fr,\fs})^*\|_{2\to2}\lesssim B'(\dels).
 \]
\par
    \end{itemize}

\end{prop}

\begin{proof} 
The assertions $i)$ and $ii)$ for the three cases $\kappa = \pm, \circ$ can be proved  in the  same manner. Furthermore, the same argument also works for the assertion $iii)$.  
    This  is possible because of  the assumptions $|\nu_\zc'-\nu'|\gg \delta^{1/2}\mu^{-1}$,  $|\nu_\zc-\nu |\gg \delta^{1/2}(\mu\mu')^{-1/2}$, which allow us  to obtain the same bound for the associated oscillatory integrals, regardless of $\kappa = \pm, \circ$.

    We first prove the assertion \textit{i)}. We use Lemma \ref{schur} to obtain the desired estimate. Since $|\fs_\zc|,\ |\fs|\sim \delta^{3/2}\mu^{-3/2}\smp$, the estimate in the assertion  \textit{i)} follows from Lemma \ref{schur}  once we prove
    \begin{align}\label{e:kerest-P*P1}
        \sup_{(y,w)\in \fs_\zc\times \fs} \big|(\fP_{\delta,\fr_\zc,\fs_\zc}^\kappa)^*\fP_{\delta,\fr,\fs}^\kappa (y,w)\big|  \lesssim_N  \frac{\delta^{\frac{5}{2}}\mu^{-2}(\mu')^{-1}}{(|\nu'-\nu'_\zc|\lambda\delta(\mu')^{-\frac{1}{2}})^{N}}.
    \end{align}
   
   We show this estimate by integration by parts. To this end, 
   set
     \begin{align}
    \label{axyw} 
       \mathfrak a(x,y,w,t,s) = \varphi_{\delta,\fr_\zc,\fs_\zc}^\kappa(x,y) \varphi_{\delta,\fr,\fs}^\kappa(x,w)
       \Psi_\delta^\kappa(x,y,t)\Psi_\delta^\kappa(x,w,s),
    \end{align}
    where
    \[
    \Psi_\delta^\kappa=
    \begin{cases}
    \psi_{2\delta},&\kappa\in\{+,-\},\\
    \psi_\delta^\circ,&\kappa=\circ.
    \end{cases}
    \]
      Since $\supp\psi_\delta^\circ\subset\supp\psi_{2\delta}$, the same support property  remains valid for every $\kappa\in\{+,-,\circ\}$.
   Then,  fixing $(y,w)\in \fs_\zc\times \fs$, we note  that
    \begin{align}
    \label{e:kerest-P*P12}
        (\fP_{\delta,\fr_\zc,\fs_\zc}^\kappa)^*\fP_{\delta,\fr,\fs}^\kappa (y,w) = \frac{1}{\mu} \iiint e^{i\lambda(\cP(x,w,s) - \cP(x,y,t))}   \mathfrak a(x,y,w,t,s) dxdsdt. 
    \end{align}

   We may assume $x\in 2\fr_\zc\cap 2\fr$, since the integral vanishes  otherwise.     Note that $x\in 2\fr_\zc\cap 2\fr$, $y\in  2\fr_\zc'$, $w\in 2\fr'$.
   Moreover, by Lemma \ref{lem:bd-derivpsi}, together with
   \eqref{e:bd-psi-top} when \(2\delta=\delt\), and by
   \eqref{b:wtphirr's'}, it follows that
    \begin{align}\label{e:bounds-B[1]}
        |\inp{\bar\nu}{\partial_x}^N  \mathfrak a(x,y,w,t,s) | \lesssim_N   (\delta^{-1}(\mu\mu')^{\frac{1}{2}})^N, 
    \end{align}
    for every $\bar\nu\in \Theta$ and $N\in\N$.  Recalling   \eqref{ftaup}, \eqref{scale},  and  \eqref{psi0},  we  also note, for $(y,w)\in  2\fr_\zc'\times 2\fr'$, 
    \begin{align}\label{e:suppcB[1]}
        \big|\supp   \mathfrak a(\cdot,y,w,\cdot,\cdot)\big| \lesssim \delta^{\frac{5}{2}}(\mu\mu')^{-1}. 
    \end{align}

  We now consider  the phase function $\cP(x,w,s) - \cP(x,y,t)$ and  claim    
   \begin{align}\label{e:lb-px1diffcP1}
        \big|\inp{\nu_\zc}{\partial_x(\cP(x,w,s) - \cP(x,y,t))}\big|\gtrsim \mu^{\frac{1}{2}}|\nu_\zc'-\nu'|,
    \end{align}
   if  $|\nu_\zc'-\nu'|\gg \delta^{1/2}\mu^{-1}$. This plays a crucial role for integration by parts. 
In order to show this, 
    we write
    \begin{align}
    \label{i:p1xdiffcP}
        \partial_x\big(\cP(x,w,s) - \cP(x,y,t)\big) 
     = \mathcal M(x,y, w) + \cR(x,y,w,s,t),
    \end{align}
    where
    \begin{align}
       \label{Mxy}
     \mathcal M(x,y,w) &= \partial_x\cP(x,w,S (x,w)) - \partial_x\cP(x,y,S (x,y)),
     \\
    \label{i:defR[1]}
        \cR(x,y,w,s,t)  &= \int_{S (x,w)}^s \partial_x\partial_u\cP(x,w,u) du - \int_{S (x,y)}^t \partial_x\partial_u\cP(x,y,u) du.
    \end{align}
  Since  $|\nu_\zc'-\nu'|\gg \delta^{1/2}\mu^{-1}$, the lower bound \eqref{e:lb-px1diffcP1}  follows if we obtain 
       \begin{align}
\label{e:p1xdiffcP-L} 
      \big|\inp{\nu_\zc}{\mathcal M (x,y,w)}\big|\gtrsim \mu^{\frac{1}{2}}|\nu'-\nu'_\zc|,  
\\
\label{e:bd-R[1]}
        \big|\cR(x,y,w,s,t)\big|\lesssim \delta^{\frac{1}{2}}\mu^{-1}\smp.
    \end{align}

    Since $\partial_x \cP(x,y,t)= ( \cos t\, x  -y)/\sin t $,   we have 
    \Be
    \label{pPs}\partial_x\cP(x,w,S (x,w)) = \frac{\inp{x}{w}x-w}{\sin S (x,w)}. 
    \Ee 
	    Here we used $\cos S(x,w)= \inp{x}{w}$. Hence,  combining this with  \eqref{pxpy1}  and \eqref{normal'}, we 
    have 
    \[ \partial_x\cP(x,w,S (x,w)) =\frac{|\partial_y\cD(x,w)|\fb(x,w)}{ 2 \sin S (x,w)}.\] Consequently, we get 
        \begin{align*}
        \mathcal M(x,y,w)=
        \frac{1}{2}\Big(\frac{|\partial_y\cD(x,w)|\fb(x,w)}{\sin S (x,w)} - \frac{|\partial_y\cD(x,y)|\fb(x,y)}{\sin S (x,y)}\Big). 
    \end{align*}
    
  Since $\sin S (x,w)$, $\sin S (x,y)\sim \smu$, by \eqref{e:sizeofxDyD'} we also see that 
    \[\mathcal M(x,y,w)=\mu^{\frac{1}{2}}(C_1\fb(x,w) - C_2\fb(x,y))\] with positive numbers  $C_1=C_1(x,w), C_2=C_2(x,y)$ such that $C_1, \,C_2\sim 1$ and $|C_1-C_2|\lesssim  \epz$ since $w, y\in \cB'$.  Recall that  $x\in 2\fr_\zc\cap 2\fr$, $y\in  2\fr_\zc'$, $w\in 2\fr'$, $\theta(\fr')=\nu'$, and $\theta(\fr'_\zc)=\nu'_\zc$.  Also, note that 
	    $\fr_\zc\times \fr_\zc'\in \fR^\kappa_\delta$, $\fr\times \fr'\in \fR^\kappa_\delta$. Hence, writing $C_1\fb(x,w) - C_2\fb(x,y)= C_1 (\fb(x,w) -\fb(x,y))+ (C_1-C_2)\fb(x,y)$, by Lemma \ref{lem:locfab} we have 
	    \[   \mu^{-\frac{1}{2}} \mathcal M(x,y,w)= C_1 (\nu ' -\nu_\zc')+ (C_1-C_2)\nu_\zc'+ O(\delta^{\frac{1}{2}}\mu^{-1}),\] 
    Since $\fr_\zc\times \fr_\zc'\in \fR^\kappa_\delta$, 
    by Lemma \ref{lem:alp} and Lemma \ref{lem:locfab}, $|\inp {\nu_\zc}{\nu_\zc'}|\lesssim \delta^{1/2}\mu^{-1}$. Hence,  it follows that 
	   \[\mu^{-\frac{1}{2}}\inp{\nu_\zc}{ \mathcal M(x,y,w)}= C_1 \inp{\nu_\zc}{(\nu ' -\nu_\zc')} + O(\delta^{\frac{1}{2}}\mu^{-1}).\] 
	 Also,  from Lemma \ref{lem:directionab} and Lemma \ref{lem:alp} we note that the unit vectors $\nu', \nu_\zc'$ satisfy  $ |\nu'-\nu_\zc'|\lesssim  (\mu'/\mu)^{1/2}\ll 1$ and $|\inp{\nu_\zc}{\nu_\zc'}|\lesssim  (\mu'/\mu)^{1/2}\ll 1$, taking $\epz$ sufficiently small. Thus,  it follows that  $|\inp{\nu_\zc}{(\nu ' -\nu_\zc')}|\sim |\nu ' -\nu_\zc'|$.  Therefore,  we get \eqref{e:p1xdiffcP-L}.

 To show \eqref{e:bd-R[1]}, we first observe that 
    \begin{align}\label{e:pxpucP}
        |\partial_x\partial_u\cP(x,w,u)| \lesssim \biggl(\frac{\mu'}{\mu}\biggr)^{\frac{1}{2}}
    \end{align}
    for $u\in [S (x,w), s]$ whenever $\psi_{2\delta}(x,w,s)\neq 0$.  
    Indeed, since $[S (x,w), s]\subset I_\mu$, note that 
        \[ |\partial_x\partial_u\cP(x,w,u)| =\frac{|w\cos u - x|}{\sin^2 u}\sim  \mu^{-1}|w\cos u - x| .\]
   Write  $w\cos u-x = \inp{w}{x}w - x + w(\cos u-\cos S (x,w)).$  By   \eqref{pxpy0} and   \eqref{e:sizeofxDyD}, we have  $\inp{w}{x}w - x =O((\mu\mu')^{1/2})$. 
   Hence, it suffices to show 
   \[\cos u-\cos S (x,w)= O((\mu\mu')^{\frac{1}{2}}).\]  Recalling \eqref{axyw}, the support properties of the cutoff function  $\psi_{2\delta}(x,w,s)$ gives  $|s-S (x,w)|\lesssim \delta^{1/2}\mu^{-1/2}$ (see \eqref{psi0}). Since 
$
        |\cos u -\cos S (x,w)|\sim \mu^{1/2} |u-S (x,w)|
$
    for $u\in [S (x,w),s]$, we have $|\cos u-\cos S (x,w)|\lesssim \delta^{1/2}$, which is less than $(\mu\mu')^{1/2}$.

    By the same argument, we also obtain $|\partial_x\partial_u\cP(x,y,u)|\lesssim (\mu'/\mu)^{1/2}$ for $u\in [S (x,y), t]$ whenever 
    $\psi_{2\delta}(x,y, t)\neq 0$.
    Combining  this  and \eqref{e:pxpucP} with \eqref{i:defR[1]} yields \eqref{e:bd-R[1]}.  

 To apply integration by parts, besides \eqref{e:lb-px1diffcP1} we also need estimates for the higher-order derivatives of the phase function $\cP(x,w,s)-\cP(x,y,t)$.
 For any unit vector $\bar\nu\in\Theta$, we have
    \begin{align}\label{e:ub-px2diffcP1}
        \big|\inp{\bar\nu}{\partial_x}^2\big(\cP(x,w,s)-\cP(x,y,t)\big)\big|
        \lesssim \mu^{-\frac{1}{2}}
    \end{align}
 provided that $\mathfrak a(x,y,w,t,s)\neq0$, and
 $\inp{\bar\nu}{\partial_x}^N(\cP(x,w,s)-\cP(x,y,t))=0$ for $N\ge3$.
 Indeed, the latter assertion follows since $\cP(\cdot,\cdot,s)$ and
 $\cP(\cdot,\cdot,t)$ are polynomials of degree $2$, while
 \[
    \inp{\bar\nu}{\partial_x}^2
    \big(\cP(x,w,s)-\cP(x,y,t)\big)
    =\frac{\sin(t-s)}{\sin s\sin t}.
 \]
 Since $s,t\in I_\mu$, we have $|t-s|\lesssim\mu^{1/2}$ and
 $\sin s,\sin t\sim\mu^{1/2}$, which proves \eqref{e:ub-px2diffcP1}.

    Now we consider  a  differential operator  
    \begin{align*}
        \cL f = \inp{\nu_\zc}{\partial_x}\Big(\frac{f(x)}{-i \la\inp{\nu_\zc}{\partial_x}(\cP(x,w,s) - \cP(x,y,t))}\Big).
    \end{align*}
Repeated integration by parts using the adjoint of \(\mathcal L\) gives, up to a factor of modulus one,
    \begin{align*}
        (\fP_{\delta,\fr_\zc,\fs_\zc}^\kappa)^*\fP_{\delta,\fr,\fs}^\kappa(y,w) = \frac{1}{\mu} \iiint e^{i\lambda(\cP(x,w,s) - \cP(x,y,t))}  \cL^N \mathfrak a(x,y,w,t,s) dxdsdt.
    \end{align*}
 Thus, the matter of proving \eqref{e:kerest-P*P1} reduces to showing 
  \Be
    \label{LNa}
        |\cL^N\mathfrak a(x,y,w,t,s)|
        \lesssim_N
        (|\nu_\zc'-\nu'|\lambda\delta(\mu')^{-\frac{1}{2}})^{-N}, 
    \Ee
 which  combined with   \eqref{e:suppcB[1]}
 gives \eqref{e:kerest-P*P1}. 
 To show \eqref{LNa},  set
\[
    \Phi(x)=\cP(x,w,s)-\cP(x,y,t),  
    \qquad
    \vartheta=|\nu_\zc'-\nu'|.
\]
By \eqref{e:lb-px1diffcP1} and \eqref{e:ub-px2diffcP1}, we have 
$
    |\inp{\nu_\zc}{\partial_x}\Phi|\gtrsim\mu^{1/2}\vartheta 
    $ 
    and 
    $
    |\inp{\nu_\zc}{\partial_x}^2\Phi|\lesssim\mu^{-1/2},
$
and $\inp{\nu_\zc}{\partial_x}^j\Phi=0$ for $j\ge3$. Hence, using
$\vartheta\gg\delta^{1/2}\mu^{-1}$ and
$\delta\lesssim\mu\mu'$, we have
$
    |\inp{\nu_\zc}{\partial_x}^2\Phi|/|\inp{\nu_\zc}{\partial_x}\Phi|
    \lesssim \mu^{-1}\vartheta^{-1} 
    \lesssim \delta^{-1}(\mu\mu')^{1/2}.
$
Thus, it follows that
\[
    \big|\inp{\nu_\zc}{\partial_x}^k(\inp{\nu_\zc}{\partial_x}\Phi)^{-1}\big|
    \lesssim_k \mu^{-\frac{1}{2}}\vartheta^{-1} 
    \big(\delta^{-1}(\mu\mu')^{\frac{1}{2}}\big)^k.
\]
Combining this with \eqref{e:bounds-B[1]} and repeatedly applying
the product rule, we obtain \eqref{LNa}. 

    We proceed to show \textit{ii)}. The proof is similar to that of \textit{i)}. Thus, we shall be brief. 
    Since $|\fr_\zc|,\ |\fr|\sim \delta^{3/2}(\mu')^{-1}$,  thanks to Lemma \ref{schur}  it is sufficient to prove
    \begin{align}\label{e:kerest-PP*1}
        \sup_{(x,z)\in \fr_\zc\times \fr} \big|\fP_{\delta,\fr_\zc,\fs_\zc}^\kappa(\fP_{\delta,\fr,\fs}^\kappa)^*(x,z)\big|  \lesssim_N  \frac{\delta^{\frac{5}{2}}\mu^{-\frac{7}{2}}\smp}{(|\nu_\zc - \nu |\lambda\delta\mu^{-1}\smp)^{N}}. 
    \end{align}
Note that 
    \begin{align*}
      \fP_{\delta,\fr_\zc,\fs_\zc}^\kappa(\fP_{\delta,\fr,\fs}^\kappa)^*(x,z)&= \frac{1}{\mu} \iiint e^{i\lambda(\cP(x,y,t) - \cP(z,y,s))}  \tilde{\mathfrak a}(x,z,y,t,s)  dydsdt, 
    \end{align*}
    where 
       \[
        \tilde{\mathfrak a}(x,z,y,t,s) = \varphi_{\delta,\fr_\zc,\fs_\zc}^\kappa(x,y) \varphi_{\delta,\fr,\fs}^\kappa(z,y) 
       \Psi_\delta^\kappa(x,y,t)\Psi_\delta^\kappa(z,y,s).
    \]

      As before, we prove \eqref{e:kerest-PP*1} by integration by parts.   We claim that 
      \begin{align}\label{e:p1ydiffcP}
        \big|\inp{\nu_\zc'}{\partial_y}(\cP(x,y,t) - \cP(z,y,s))\big|\gtrsim \smp|\nu_\zc - \nu |. 
    \end{align}
    Since  $|\nu_\zc-\nu |\gg \delta^{1/2}(\mu\mu')^{-1/2}\ge \delta^{1/2}\mu^{-1}$, it suffices to show that 
        \begin{align}
        \label{nzm}
        |\inp{\nu_\zc'}{ \tilde{\mathcal M}(x,z,y)  }|&\gtrsim \smp|\nu_\zc - \nu |, 
    \end{align}
    where 
     \[\tilde{\mathcal M}(x, z, y)= \partial_y\cP(x,y,S (x,y)) - \partial_y\cP(z,y,S (z,y)).\]
   Indeed, $ \partial_y(\cP(x,y,t) - \cP(z,y,s))= \tilde{\mathcal M}(x, z, y)+ \tilde{\mathcal R}(x,z, y,t,s)$ where 
    \[ 
       \tilde{\cR}(x,z,y,t,s) = \int_{S (x,y)}^t \partial_y\partial_u\cP(x,y,u) du - \int_{S (z,y)}^s \partial_y\partial_u\cP(z,y,u) du
\]
    Observe $|\partial_y\partial_u\cP(x,y,u)| = |x\cos u - y|/\sin^2 u\lesssim 1.$\footnote{This is clear since $x\cos u - y= x(\cos u-1) +(x- y)= O(\mu)$ and $u\in I_\mu$.}  
    Also, by the support properties of the functions $\psi_{2\delta}(x,y,\cdot)$ and $\psi_{2\delta}(z,y,\cdot)$,   $|t-S (x,y)|, |s- S(z,y)| \lesssim \delta^{1/2}\mu^{-1/2}$. 
    Hence, $|\widetilde{\cR}(x,z,y,t,s)|\lesssim \delta^{1/2}\mu^{-1/2}=O(\delta^{1/2}\mu^{-1/2}).$

	    As before, since $\partial_y\cP(x,y,t)= ( \cos t\, y  -x)/\sin t $, using  \eqref{pxpy1} and \eqref{normal'},  we get 
    \[   \tilde{\mathcal M}(x,z,y)   = \frac{|\partial_x\cD(x,y)|\fa(x,y)}{2\sin S (x,y)} - \frac{|\partial_x\cD(z,y)|\fa(z,y)}{2\sin S (z,y)} .\]
	   Note that  $x\in 2\fr_\zc, y\in 2\fr_\zc'\cap 2\fr', z\in 2\fr $.  Combining the identity with  \eqref{e:sizeofxDyD},  Lemma \ref{lem:alp},  and Lemma \ref{lem:locfab},  the same argument gives   \eqref{nzm}. 
    
In addition to \eqref{e:p1ydiffcP}, we have the bound 
\[
        \big|\inp{\nu_\zc'}{\partial_y}^2\big(\cP(x,y,t) - \cP(z,y,s)\big)\big|\lesssim \mu^{-\frac{1}{2}}, 
            \]
  where the higher order derivatives are identically zero. Using
  Lemma \ref{lem:bd-derivpsi}, together with
  \eqref{e:bd-psi-top} when \(2\delta=\delt\), and
  \eqref{b:wtphirr's'}, we also have
  \[ 
        |\inp{\nu'_\zc}{\partial_y}^N\tilde{\mathfrak a}(x,z,y,t,s)| \lesssim_N  (\delta^{-1}\mu)^N
\] 
 for any $N\in \mathbb N_0$.
        Putting those estimates together with \eqref{e:p1ydiffcP}, we get
     \begin{align}\label{e:bounds-B[2]} 
             |\tilde \cL^N \tilde{\mathfrak a}(x,z,y,t,s)| \lesssim_N  (|\nu_\zc-\nu |\lambda\delta\mu^{-1}\smp)^{-N}, 
    \end{align}
 since  $\delta\lesssim \mu\mu'$  and $|\nu_\zc-\nu |\gg \delta^{1/2}(\mu\mu')^{-1/2}$,   where 
     \begin{align*}
        \tilde\cL f = \inp{\nu_\zc'}{\partial_y}\Big(\frac{f(y)}{-i\la\inp{\nu_\zc'}{\partial_y}(\cP(x,y,t) - \cP(z,y,s))}\Big).
    \end{align*}
Since \(\tilde \cL^\ast e^{i\lambda(\cP(x,y,t)-\cP(z,y,s))}=e^{i\lambda(\cP(x,y,t)-\cP(z,y,s))}\), we use integration by parts for the integral \(\fP_{\delta,\fr_\zc,\fs_\zc}^\kappa(\fP_{\delta,\fr,\fs}^\kappa)^*(x,z)\) with \(\tilde\cL^\ast\) and \(\tilde\cL\). As a result, using \eqref{e:bounds-B[2]}, we obtain \eqref{e:kerest-PP*1}.

Finally, concerning the assertion $iii)$, one can show the estimates for  the operators  $(\fP_{\dels, \fr_\zc,\fs_\zc})^*\fP_{\dels, \fr,\fs}$ and $\fP_{\dels, \fr_\zc,\fs_\zc}(\fP_{\dels, \fr,\fs})^*$ by the same argument. We omit the details. 
\end{proof}

\subsection{Almost orthogonality within a fixed input rectangle} 
Proposition \ref{prop:alortho-1}  shows almost orthogonality between  the operators $\fP_{\delta,\fr_\zc,\fs_\zc}^\pm$ and $\fP_{\delta,\fr,\fs}^\pm$ when $\theta(\fr_\zc)$ and $\theta(\fr)$  or $\theta(\fr_\zc')$ and  $\theta(\fr')$ are sufficiently separated. 
However,  if  $\theta(\fr_\zc)=\theta(\fr)$  and $\theta(\fr_\zc')=\theta(\fr')$, no such almost orthogonality is expected. 
In this case, thanks to the fact that $\fr\times\fr', \fr_\zc\times \fr'_\zc\in \mathfrak R^\pm_\delta$, we need to consider  the 
case $\fr=\fr_\zc$ and $\fr'=\fr'_\zc$ (see Section \ref{sec:conclude}). However,  even for such a case, it turns out that there is  almost orthogonality between the operators  
$\fP_{\delta,\fr,\fs}^\kappa$, $\fs\in \mathfrak S(\fr')$, provided that $\fs_\zc$ and $\fs$ are separated enough.

\begin{prop}\label{prop:alortho-2} Let $\dels\le\delta<\delt$ and $\kappa\in\{+,-,\circ\}$.  Let $\fr\times \fr'\in \fR^\kappa_\delta$, and   $\fs_\zc=\fs(\ell_1,\fr'),\ \fs=\fs(\ell_2,\fr')\in \mathfrak S(\fr')$
 (see \eqref{s'}). Then, if $|\ell_1 - \ell_2|\gg 1$, i.e. $\dist(\fs_\zc, \fs)\gg  \delta^{1/2}(\mu'/\mu)^{1/2} $, the estimate
    \begin{align}
    \label{fpdel}
            \|(\fP_{\delta,\fr,\fs_\zc}^\kappa)^*\fP_{\delta,\fr,\fs}^\kappa\|_{2\to 2} \lesssim_N   B(\delta):=\frac{\delta^4\mu^{-\frac{7}{2}}(\mu')^{-\frac{1}{2}}}{(|\ell_1-\ell_2|\lambda\delta^{\frac{3}{2}}\mu^{-\frac{3}{2}})^{N}}, \quad \kappa=\pm, \circ
    \end{align}
    holds  for any  $N\in \N$. Furthermore, if $\delta =\dels$, then for $|\ell_1-\ell_2|\gg 1$,
    \begin{align}
        \label{fpdel*}
        \|(\fP_{\dels, \fr,\fs_\zc})^*\fP_{\dels, \fr,\fs}\|_{2\to 2} \lesssim_N  B(\dels)=|\ell_1-\ell_2|^{-N}\lambda^{-\frac{8}{3}}\mu^{\frac{1}{2}}(\mu')^{-\frac{1}{2}}.
    \end{align}
\end{prop}
\begin{proof}
    The proof is similar to that of Proposition \ref{prop:alortho-1}.   As before,  we prove \eqref{fpdel} by obtaining an estimate for the kernel  $(\fP_{\delta,\fr,\fs_\zc}^\kappa)^*\fP_{\delta,\fr,\fs}^\kappa$ via integration by parts. 
    
 To do this, we recall that  \eqref{e:kerest-P*P12} with $\fr_\zc=\fr$. Once 
    we have \begin{align}
    \label{e:ub-difffss0}
        \big|\inp{\nu}{\partial_x(\cP(x,w,s) - \cP(x,y,t))}\big|\gtrsim \delta^{\frac{1}{2}}\mu^{-1}\smp|\ell_1-\ell_2|.
    \end{align}
   whenever $\mathfrak a(x,y,w,t,s)\neq 0$, routine integration by parts yields \eqref{fpdel}. Indeed, 
   Observe also that \eqref{e:bounds-B[1]} and \eqref{e:suppcB[1]} remain
   valid. Moreover, in the present case $y,w\in2\fr'$, so
   \eqref{axyw}, \eqref{psi0}, \eqref{psi1}, and \eqref{scale} give
   \[
       \big|\inp{\nu}{\partial_x}^2
       \big(\cP(x,w,s)-\cP(x,y,t)\big)\big|
       \lesssim\delta^{\frac{1}{2}}\mu^{-\frac{3}{2}}.
   \]
   Since $|\ell_1-\ell_2|\gg1$, the same argument as in the proof of
   Proposition \ref{prop:alortho-1} yields the desired estimate.
   
    To show \eqref{e:ub-difffss0}, breaking $\cP(x,w,s) - \cP(x,y,t)$ in the same way as before,  we have \eqref{i:p1xdiffcP}. Since $\fr\times \fr'\in \fR^\kappa_\delta$, the estimate  \eqref{e:bd-R[1]} continues to hold.  Consequently, we may regard $\cR$ as a minor error, since $|\ell_1-\ell_2|\gg 1$.  
 Therefore, \eqref{e:ub-difffss0} follows if we prove 
    \begin{align}\label{e:ub-difffss}
        \big|\inp{\nu}{\mathcal M(x,y,w)}\big|\gtrsim \delta^{\frac{1}{2}}\mu^{-1}\smp|\ell_1-\ell_2| 
    \end{align}
    for $(y,w)\in \fs_\zc\times \fs$ and $x\in \fr$, where $\mathcal M(x,y,w)$ is given by \eqref{Mxy}.  Recalling  \eqref{pPs}, \eqref{pxpy1},   and \eqref{normal'}, we can write 
    \begin{align}
        \begin{aligned}\label{i:diffpxcP-2t}
            \mathcal M(x,y,w)= H(x,w,y) (\fb(x,w) - \fb(x,y)) + I(x,w,y)\fb(x,y),
        \end{aligned}
    \end{align}
    where
    \begin{align*}
     H(x,w,y)= \frac{|\partial_y\cD(x,w)|}{2\sin S (x,w)}, \quad   I(x,w,y) = \frac{1}{2}\Big(\frac{|\partial_y\cD(x,w)|}{\sin S (x,w)} - \frac{|\partial_y\cD(x,y)|}{\sin S (x,y)}\Big).
    \end{align*}
   Note that $H(x,w,y)\sim \mu^{1/2}$ by \eqref{e:sizeofxDyD'} and
   $\sin S(x,w)\sim\mu^{1/2}$. Since $|\ell_1-\ell_2|\gg 1$ and
   $\mu'\ll \mu$, to show \eqref{e:ub-difffss} we need only to prove
    \begin{align}\label{e:ub-difffss-normal}
     |\inp{\nu}{\fb(x,w) - \fb(x,y)}|  \sim    |\ell_1-\ell_2| \delta^{\frac{1}{2}}\mu^{-\frac{3}{2}} (\mu')^{\frac{1}{2}}.
    \end{align} 
     
      Indeed, this follows  since  $
        |I(x,w,y)|\lesssim \delta\mu^{-3/2}\ll  \delta^{1/2}(\mu')^{1/2}\mu^{-1}. $ To see this, setting  $\displaystyle J(x,w,y)= 2|\partial_y\cD(x,w)|/\sin S (x,w) + 2|\partial_y\cD(x,y)|/\sin S (x,y)$, we  have
    \begin{align}\label{i:Ixwy24}
        I(x,w,y)J(x,w,y)= \mathfrak b(x,w){|\partial_y\cD(x,w)|^2} - \mathfrak b(x,y) {|\partial_y\cD(x,y)|^2}. 
    \end{align} 
    where $\mathfrak b(x,y)=(1-\inp xy^2)^{-1}$.  On the other hand,  \eqref{pxpy1} gives   $|\partial_y\cD(x,u)|^2/4= (1- \inp xu^2)(1-|x|^2)- \cD(x,u)$. Thus, 
      \begin{align*}
     I(x,w,y)J(x,w,y)= 4\big(\mathfrak b(x,y){\cD(x,y)}- \mathfrak b(x,w){\cD(x,w)}\big).
    \end{align*}
Note from \eqref{inp} that  $\mathfrak b(x,y)$, $\mathfrak b(x,w)\sim \mu^{-1}$, and $J(x,w,y)\sim  \mu^{1/2} $ by \eqref{e:sizeofxDyD'}. 
  Since 
$|\cD(x,y)|\lesssim \delta$ and $|\cD(x,w)|\lesssim \delta$,   we get $|I(x,w,y)|\mu^{1/2}\lesssim \delta/\mu$. Therefore,  
\Be \label{e:bd-Ixwy}         |I(x,w,y)|\lesssim \delta\mu^{-\frac{3}{2}}.\Ee
 
Now, it remains to show \eqref{e:ub-difffss-normal}. 
    On the other hand, since 
    \[ |\inp{\nu_\sprp'}{w-y}|\gtrsim  |\ell_1-\ell_2| \delta^{\frac{1}{2}}\mu^{-\frac{1}{2}} (\mu')^{\frac{1}{2}}, 
    \] 
    for $(y,w)\in \fs_\zc\times \fs$, using Lemma \ref{lem:locfab} we have $|\inp{\fb_\sprp(x,w)}{w-y}|\gg  \delta \mu^{-3/2}(\mu')^{1/2}$. Thus, 
	     by $ii)$ in  Lemma \ref{lem:directionsofab}, 
	     \[ |\fb(x,w) - \fb(x,y)|\sim \mu^{-1}| \inp{\fb_\sprp(x,w)}{y-w}|\gtrsim    |\ell_1-\ell_2| \delta^{1/2}\mu^{-3/2} (\mu')^{1/2} .\]
Set \(u=\fb(x,w)\) and \(v=\fb(x,y)\). Since \(u\) and \(v\) are
unit vectors, \(u-v\) is orthogonal to \(u+v\). By Lemma
\ref{lem:locfab}, both \(u\) and \(v\) lie in a sufficiently small arc
about \(\nu'\); hence
$
    |\inp{\nu_\sprp'}{u-v}|\sim |u-v|.
$
Moreover, Lemma \ref{lem:alp} and the orientation in Lemma
\ref{lem:directionab} imply that \(|\nu+\nu_\sprp'|\ll1\). Therefore, $|\inp{\nu}{u-v}|\sim |u-v|.$ Combining this with the preceding estimate gives
\[
    |\inp{\nu}{\fb(x,w)-\fb(x,y)}|
    \sim |\ell_1-\ell_2|\delta^{\frac12}\mu^{-\frac32}(\mu')^{\frac12}.
\]
   
     Consequently, using Lemma \ref{lem:pxpy}  we have 
	     \[  |H(x,w,y) \inp{\nu}{\fb(x,w) - \fb(x,y)}|  \sim    |\ell_1-\ell_2| \delta^{\frac{1}{2}}\mu^{-1} (\mu')^{\frac{1}{2}}.\] 
 On the other hand,    $|\inp{\nu}{\fb(x,y)}|\lesssim \delta^{1/2}  \mu^{-1}$ by Lemmas \ref{lem:alp} and \ref{lem:locfab}. Combining this with 
    \eqref{e:bd-Ixwy}, we see 
    \[  |I(x,w,y) \inp{\nu}{\fb(x,y)}|\ll  \delta^{\frac{1}{2}}\mu^{-1} (\mu')^{\frac{1}{2}}.\] 
 Putting these two inequalities together yields \eqref{e:ub-difffss}. 
\end{proof}
These off-diagonal estimates provide the almost-orthogonality input for Section~\ref{sec:conclude}; the corresponding diagonal \(L^2\) bounds are established next in Section~\ref{sec:individual}.
 \section{Estimates for the localized pieces}
\label{sec:individual}
In this section, we obtain $L^2$ bounds on the localized operators $\fP_{\delta,\fr,\fs}^\kappa$ and $\fP_{\dels, \fr,\fs}$.

\begin{prop}\label{prop:L2est}
    Let $\dels\le\delta<\delt$ and $\kappa\in\{+,-,\circ\}$. Suppose that $\fr\times \fr'\in \fR^\kappa_\delta$ and $\fs\in \mathfrak S(\fr')$. Then 
    \begin{align}\label{e:fP-dkap}
        \|\fP_{\delta,\fr,\fs}^\kappa \|_{L^2\to L^2}\lesssim \lambda^{-1}\delta^{\frac{1}{2}}(\mu\mu')^{-\frac{1}{4}}.
    \end{align}
    Furthermore,  if  $\fr\times \fr'\in \fR_{\dels}^\circ$ and $\fs\in \mathfrak S(\fr')$, then 
    \begin{align}\label{e:fP-ast}
        \|\fP_{\dels, \fr,\fs} \|_{L^2\to L^2}\lesssim \lambda^{-\frac{4}{3}}\mu^{\frac{1}{4}}(\mu')^{-\frac{1}{4}}.
    \end{align}
\end{prop}

As shown below, the estimates for $\fP_{\delta,\fr,\fs}^-$, $\fP_{\delta,\fr,\fs}^\circ$, and $\fP_{\dels, \fr,\fs}$ are less involved than the estimate for 
$\fP_{\delta,\fr,\fs}^+$, whose analysis constitutes the main part of this section.

\subsection{Nonstationary and bottom-scale pieces}   We first deal with $\fP_{\delta,\fr,\fs}^-$, $\fP_{\delta,\fr,\fs}^\circ$, and $\fP_{\dels,\fr,\fs}$. 
We begin by showing estimates for the operator $\fP_{\dels,\fr,\fs}$, which is the marginal case with $\dels\sim \lambda^{-2/3}\mu$. 

\subsubsection*{Estimate for $\fP_{\dels,\fr,\fs}$} The desired estimate \eqref{e:fP-ast} is easy to show; in fact, it follows from the van der Corput lemma and Schur’s test. Indeed, note that 
\[ (\partial_t)^3\mathcal P(x,y,t)
=
\frac{  H(x,y,\cos t) }{\sin^4 t}, \] 
where $ H(x,y,s)=\langle x,y\rangle s(s^2+5)
-(|x|^2+|y|^2)(1+ 2s^2)
$. Expanding at $t=0$, we have 
\[
H(x,y,\cos t) 
=
-3|x-y|^2
+2|x-y|^2t^2
+\left(
\frac{13}{12}\langle x,y\rangle
-\frac{2}{3}(|x|^2+|y|^2)
\right)t^4
+O(t^6).
\]
By the localization of \(\cB\times\cB'\), the coefficient of \(t^4\)
is bounded above by a fixed negative constant, uniformly in \(x\) and
\(y\). On
\(\supp\psi_\vartriangle\), we have \(t\sim\mu^{1/2}\). Since
\(\mu\le\epz\) with \(\epz\) sufficiently small, the positive
\(2|x-y|^2t^2\) term and the \(O(t^6)\) remainder are absorbed by the
two negative leading terms. Consequently,
\[
    -H(x,y,\cos t)\gtrsim |x-y|^2+t^4\gtrsim\mu^2.
\]
The reverse bound follows from the same expansion, and hence
\(|H(x,y,\cos t)|\sim\mu^2\) on \(\supp\psi_\vartriangle\).
Consequently,  $|(\partial_t)^3\mathcal P(x,y,t)|\sim 1$ for $(x,y,t)\in \supp \psi_\vartriangle$.  

Now, recall \eqref{Olds} and \eqref{bsso}.   By the van der Corput lemma (see \cite{St93}) we see that 
\Be\label{e:ptbd-fP*}  |\fP_{\dels,\fr,\fs}(x,y)|\lesssim \lambda^{-\frac{1}{3}} \mu^{-\frac{1}{2}} \chi_{2\fr} (x) \chi_{2\fs}(y).\Ee 
 Since $|\fr|\le  \sigma\sigma_\sprp  $ and $|\fs|\le \sclp \sclrp\big(\mu'/\mu\big)^{1/2} $, by Lemma \ref{schur} and \eqref{scale} with $\delta=\dels$ we get \eqref{e:fP-ast}. 
	 Here we used that $\dels\sim \lambda^{-2/3}\mu$. 
 
\subsubsection*{Estimates for $\fP_{\delta,\fr,\fs}^-$ and $\fP_{\delta,\fr,\fs}^\circ$}

To obtain \eqref{e:fP-dkap} for $\kappa = -,\circ$, we first  recall  \eqref{Oldpm},  \eqref{Oldc}, and \eqref{bss}.  Observe that $\partial_t\cP(x,y,t)$ has a lower bound 
\begin{align}\label{e:1stlb-cP}
    |\partial_t\cP(x,y,t)|\gtrsim \delta\mu^{-1}
\end{align}
provided that  $\varphi^-_\delta(x,y)\psi_{2\delta}(x,y,t)\neq 0$ or $\varphi_\delta(x,y)\psi_{\delta}^\circ(x,y,t)\neq 0$. Concerning the first case, note that $-\cD(x,y) \sim \delta$, \eqref{e:1stlb-cP} follows from \eqref{i:ptcP-Q}. In the latter case, \eqref{psi1} gives \(|t-S(x,y)|\ge 2^2C_\ast\delta^{1/2}\mu^{-1/2}\), while \eqref{r:suppcon-vphi} gives \(|\cD(x,y)|\le2\delta\). Hence \eqref{tsxy} holds, and \eqref{e:simpleobv} yields \eqref{e:1stlb-cP}.

As seen in the proof of \eqref{osc-est}, $\partial_t \cP(x,y,t)$ is monotonic on a finite number of intervals. Thus, the van der Corput lemma together with Lemma \ref{lem:bd-derivpsi} (and \eqref{e:bd-psi-top} when \(2\delta=\delt\)) gives 
\begin{align}\label{e:ptbd-fPcn}
   | \fP_{\delta,\fr,\fs}^\kappa(x,y)|\lesssim \lambda^{-1}\delta^{-1}\mu^{\frac{1}{2}}  \chi_{2\fr} (x) \chi_{2\fs}(y)
\end{align}
for $\kappa=\circ,-.$ Combining this with Lemma \ref{schur}, we obtain the desired estimates in Proposition \ref{prop:L2est}.

\medskip

The rest of this section is devoted to the proof of \eqref{e:fP-dkap} for $\kappa=+$.

\subsection{Stationary-phase reduction}\label{subsec:stationary-phase}
 We now address the task of establishing \eqref{e:fP-dkap} with $\kappa = +$. For this purpose, we use the stationary phase method to obtain an asymptotic expansion for the kernel of $\fP_{\delta,\fr,\fs}^+$.

Note that  $\cD(x,y)\sim\delta$ and $t\in I_\mu$ whenever  $\varphi^+_\delta(x,y)\psi_{2\delta}(x,y,t)\neq 0$.
Henceforth in this subsection, \(S_1\) and \(S_2\) are considered only
on \(\supp\varphi_\delta^+\), where \(\cD(x,y)\sim\delta>0\).
Recalling \eqref{i:ptcP-Q}, for such \((x,y)\) the function
\(\partial_t\cP(x,y,t)\) has two zeros \(S_1(x,y),S_2(x,y)\) on
\([0,\pi/2]\), given by
\begin{equation}
\label{cosS}
\cos S_\ell (x,y) = \inp xy - (-1)^\ell \sqrt{\cD(x,y)}, \quad \ell=1,2
\end{equation}
On this support, \(S_1(x,y)<S(x,y)<S_2(x,y)\), and both
\(S_1(x,y)\) and \(S_2(x,y)\) belong to \(I_\mu\). Moreover,
\[
\cos S_1(x,y)-\cos S_2(x,y)=2\sqrt{\cD(x,y)}.
\]
Thus, it follows that 
\[
   S_2(x,y) - S_1(x,y) \sim \mu^{-\frac{1}{2}}\delta^{\frac{1}{2}}.
\]
By \eqref{i:ptcP-Q},  we have 
\Be 
 \label{1st-d}  \partial_t \cP(x,y,t)=- \frac{\prod_{\ell=1}^2(\cos t - \cos S_\ell (x,y))}{2\sin^2 t}.
\Ee  Hence, it follows that
\[ 
    \partial_t^2 \cP(x,y,S_\ell(x,y)) =  (-1)^\ell \frac{ \cos S_2(x,y) - \cos S_1(x,y)}{2\sin S_\ell(x,y)}, \quad \ell=1, 2. 
\] 
So, we have 
\begin{align}\label{e:est-pt^2cP}
    |\partial_t^2 \cP(x,y,S_\ell(x,y))| \sim \delta^{\frac{1}{2}}\mu^{-\frac{1}{2}}, \quad \ell=1, 2. 
\end{align}

On the other hand, a calculation shows
\begin{align*}
    \cP(x,y,t) = \frac{t}{2} -\frac{(|x|^2+|y|^2)\sin t}{2(1 + \cos t)} + \frac{(1 - \inp xy)^2}{2\sin t} - \frac{\cD(x,y)}{2 \sin t}.
\end{align*}
Since $(1 - \inp xy)^2\lesssim \mu^2$ and $\cD(x,y)\lesssim \mu\mu'$, the above expression implies
\begin{align}\label{e:est-pt^ncP}
    |\partial_t^n \cP(x,y,t)| \lesssim \mu^{-\frac{1}{2}(n-3)}
\end{align}
for $t\sim \mu^{1/2}$ and $n\ge 3$. To apply the stationary phase method, we need to decompose the integral into the parts near and 
away from the critical points.  To do this,  break 
\begin{align*}
    \psi_{2\delta}(x,y,t) = \psi_{\delta,0}^*(x,y,t) + \sum_{\ell =1,2}\psi_{\delta,\ell}^*(x,y,t), 
\end{align*}
where 
\begin{align}\label{i:defpsipm}
    \psi_{\delta, \ell}^*(x,y,t) = \psi_{2\delta}(x,y,t)\eta\Big(\frac{t - S_\ell(x,y)}{\epz\delta^{\frac{1}{2}}\mu^{-\frac{1}{2}}}\Big), \quad \ell=1, 2
\end{align}
for a sufficiently small $\epz > 0$. Note that $\supp(\psi_{\delta, 1}^*)$ and $\supp(\psi_{\delta, 2}^*)$ are separated by $c\delta^{1/2}\mu^{-1/2}$ with some $c>0$ if $\epz$ is sufficiently small. Thus, by Lemma \ref{lem:bd-derivpsi} (and by \eqref{e:bd-psi-top} when $2\delta=\delt$), we have
\begin{align}\label{e:psi*circbds}
    |\partial_t^N\psi_{\delta,\ell}^*(x,y,t)| \lesssim_N (\delta^{-\frac{1}{2}}\mu^{\frac{1}{2}})^N, \quad \ell=0,1,2 
\end{align}
for $N\in\N$.

Recalling $\fP_{\delta, \fr,\fs}^+= \mathcal O_\lambda[ \varphi^+_{\delta,\fr,\fs}, \psi_{2\delta}]$, we set 
\[ \fP_{\ell,\delta,\fr,\fs}^+=  \mathcal O_\lambda[ \varphi^+_{\delta,\fr,\fs}, \psi_{\delta, \ell}^*],\quad \ell=0,1,2.  \] 
Note that \eqref{e:1stlb-cP}  holds on the support of $\psi_{\delta,0}^*(x,y,t)$.  
 Using the van der Corput lemma, one can deduce that 
 \Be
 \label{p+circ}
  |\fP_{0,\delta,\fr,\fs}^+(x,y)| \lesssim \lambda^{-1}\delta^{-1}\mu^{\frac{1}{2}}  \chi_{2\fr} (x) \chi_{2\fs}(y).
  \Ee
As before,     applying Lemma \ref{schur} yields the desired $L^2$ estimate for $\fP_{0,\delta,\fr,\fs}^+$ (cf. \eqref{e:ptbd-fPcn}). 

We proceed to estimate $\fP_{\ell,\delta,\fr,\fs}^+$, $\ell=1,2$. The two operators can be handled in the same manner. 
Recall  \eqref{olambda}. Note that the kernel of the operator $\fP_{\ell,\delta,\fr,\fs}^+$ is given by a product
$\varphi_{\delta,\fr,\fs}^+(x,y)$ and  $\mathcal O_\lambda[\psi_{\delta, \ell}^*]$. 
After changing  variables $t\to \tau_\ell(t)$,  we have 
\[ 
 \mathcal O_\lambda[\psi_{\delta, \ell}^*](x,y)= {\delta^{\frac{1}{2}} \mu^{-1}}   \int   e^{i   (\delta^{\frac{3}{2}}\mu^{-\frac{3}{2}} \lambda)  \Phi_\ell(x,y, t)}     \psi_{\delta,\ell}^*(x,y, \tau_\ell(t)) dt,   \] 
where
\begin{align*}
    \tau_\ell(x,y,t) &:= \delta^{\frac{1}{2}}\mu^{-\frac{1}{2}} t + S_\ell (x,y),
    \\ 
    \Phi_\ell(x,y, t)&:=   \delta^{-\frac{3}{2}}\mu^{\frac{3}{2}} \cP(x,y, \tau_\ell(x,y,t)). 
\end{align*}

Note $\supp  \psi_{\delta,\ell}^*(x,y, \tau_\ell(\cdot))\subset (-\epz, \epz)$.  To obtain an asymptotic expansion for $ \mathcal O_\lambda[\psi_{\delta, \ell}^*]$, we use \cite[Theorem 7.7.5]{H90}, which allows us to handle the
oscillatory integral uniformly in $x$ and $y$. For this purpose, we need to check several conditions. First, note from  \eqref{e:est-pt^ncP} that 
 $ \partial_t^k \Phi_\ell(x,y,\cdot)= O(1)$ for $k\ge 3$ on $\supp   \psi_{\delta,\ell}^*(x,y, \tau_\ell(\cdot)).$  
Also, using \eqref{1st-d}, we have $\partial_t\Phi_\ell(x,y,0)=0$, $|\partial_t^2\Phi_\ell(x,y,\cdot)|\sim 1$, and  
 \[ \left|\frac{t}{\partial_t\Phi_\ell(x,y,t)}\right|\lesssim 1\] 
on $\supp  \psi_{\delta,\ell}^*(x,y, \tau_\ell(\cdot))$ with sufficiently small $\epz>0$.  
For $\ell=1,2$, set
\[
c_\ell
:=\sqrt{2\pi}\exp\left(
  \frac{i\pi}{4}\operatorname{sgn}
  \big(\partial_t^2\Phi_\ell(x,y,0)\big)
\right)
=\sqrt{2\pi}\,e^{\frac{i(-1)^{\ell+1}\pi}{4}}.
\]
Finally, the large parameter \(\Lambda:=\lambda\delta^{3/2}\mu^{-3/2}\) satisfies \(\Lambda\gtrsim1\), since \(\delta\ge\dels\sim\lambda^{-2/3}\mu\). Together with the preceding support, derivative, and nondegeneracy bounds, this verifies the hypotheses for a uniform application of \cite[Theorem 7.7.5]{H90}.
Then, by \cite[Theorem 7.7.5]{H90}, we obtain  
\begin{align}\label{i:exp-fP+}
   \mathcal O_\lambda[\psi_{\delta, \ell}^*](x,y)= \lambda^{-\frac{1}{2}}(\delta\mu)^{-\frac{1}{4}} 
   c_\ell\psi_{\delta,\ell}^*(x,y, S_\ell(x,y)) \frac{e^{i(\lambda\delta^{\frac{3}{2}}\mu^{-\frac{3}{2}})\Phi_\ell(x,y,0)}}{\sqrt {|\partial_t^2 \Phi_\ell(x,y, 0)|}}   + E_{\ell}(x,y)
\end{align}
for $\ell=1,2$, where   
\begin{align*}
    E_{\ell}(x,y) &= O(\lambda^{-\frac{3}{2}}\delta^{-\frac{7}{4}} \mu^{\frac{5}{4}}). 
\end{align*}
The contribution from $E_{\ell}$ is acceptable. Indeed, since
$\lambda^{-2/3}\mu\lesssim \delta$,
\[
|(\varphi^+_{\delta,\fr,\fs} E_{\ell})(x,y)|
\lesssim \lambda^{-1}\delta^{-1}\mu^{\frac{1}{2}}
\chi_{2\fr}(x)\chi_{2\fs}(y).
\]
Lemma \ref{schur} gives
\[
\| \varphi^+_{\delta,\fr,\fs} E_{\ell}\|_{2\to 2}
\lesssim \lambda^{-1}\delta^{\frac{1}{2}}(\mu\mu')^{-\frac{1}{4}}.
\]
Therefore, we need only to handle the contribution from the major term. 

Let us set 
\begin{align}
\label{phiell}
 \Phi_\ell(x,y) &=\Phi_\ell(x,y,0)=  \delta^{-\frac{3}{2}}\mu^{\frac{3}{2}} \cP(x,y, S_\ell(x,y))
 \\ 
 \label{aell}
    A_{\ell}(x,y) &=  c_\ell\varphi_{\delta,\fr,\fs}^+(x,y)  \frac{\psi_{\delta,\ell}^*(x,y, S_\ell(x,y))}{\sqrt {|\partial_t^2 \Phi_\ell(x,y, 0)|}}.  
\end{align}
Consequently, the matter now reduces to showing 
\begin{align}\label{e:L2est-fG}
    \|\mathfrak O_{\ell} f\|_{L^2}\lesssim \lambda^{-\frac{1}{2}}\delta^{\frac{3}{4}}(\mu')^{-\frac{1}{4}}\|f\|_{L^2},
\end{align}
where 
\[
    \mathfrak O_{\ell} f(x) = \int e^{i(\lambda\delta^{\frac{3}{2}}\mu^{-\frac{3}{2}})\Phi_\ell(x,y)} A_{\ell}(x,y)  f(y) dy.
\]

\subsection{Estimates for the amplitude and the phase functions}
To establish \eqref{e:L2est-fG}, we follow an argument in \cite{LR22}. We will apply H\"ormander's $L^2$ boundedness theorem after freezing one coordinate in each of $x$ and $y$. For this purpose, it is necessary to obtain bounds on derivatives of the amplitude $A_{\ell}$ and the phase function $\Phi_\ell$.

Recalling the differential operators $\mathcal L_{\delta,\nu}$ and $\mathcal L_{\delta,\nu'}$ defined in \eqref{d:diff-L}, we begin with the following estimates for the derivatives of $S_\ell$.

\begin{lem}\label{lem:bounds-Spm}
    Let $\fr\times \fr'\in \cR_{\delta, \nu}\times \cR_{\delta,\nu'}'$ and  $\ell\in \{1,2\}$. Let $\wt\nu'$ be a unit vector such that $|\wt\nu'-\nu'|\le \epz(\mu'/\mu)^{1/2}$.
    Assume $\fr\times\fr'\in \fR^+_\delta$. Then, for $(x,y)\in 2\fr\times 2\fr'$ and $N\in\N$,
    \begin{align*}
        |(\mathcal L_{\delta, \nu})^NS_\ell(x,y)| + |(\mathcal L_{\delta, \wt\nu'})^NS_\ell(x,y)|  \lesssim_N  \delta^{\frac{1}{2}}\mu^{-\frac{1}{2}}.
    \end{align*}
\end{lem}

\begin{proof}
    Assume $(x,y)\in 2\fr\times 2\fr'$. To prove the lemma, we first obtain the following bound.
    \begin{align}\label{e:bounds-D-1/2}
    |(\mathcal L_{\delta, \nu})^N \cD^{\frac{1}{2}}(x,y)|,\  |(\mathcal L_{\delta, \wt\nu'})^N \cD^{\frac{1}{2}}(x,y)|  \lesssim_N \delta^{\frac{1}{2}}.
    \end{align}

    Let $\mathcal L$ denote either
    $\mathcal L_{\delta,\nu}$ acting in the $x$ variable or
    $\mathcal L_{\delta,\wt\nu'}$ acting in the $y$ variable, as defined in \eqref{d:diff-L}.
    Note that 
    \[
        |\mathcal L\cD(x,y)|+|\mathcal L^2\cD(x,y)|
        \lesssim\delta,
        \qquad
        \mathcal L^j\cD(x,y)=0\quad (j\ge3).
    \]
    The estimate follows from Lemma \ref{lem:pxpy}, \eqref{scale}, and
    $\delta\le \delt\lesssim\mu\mu'$, while the vanishing follows
    because $\cD$ is quadratic in either $x$ or $y$.
    Since $\cD(x,y)\sim\delta$, repeated application of the chain rule
    (or the Fa\`a di Bruno formula) to
    $u\mapsto u^{1/2}$ gives \eqref{e:bounds-D-1/2}.

    By Lemmas \ref{lem:directionab} and \ref{lem:locfab}, together
    with \eqref{dir-con},
    $|\langle \nu,y\rangle|\lesssim(\mu'/\mu)^{1/2}$.
    For $\mathcal L_{\delta,\wt\nu'}$ acting in the $y$ variable, we
    instead use $|\langle\wt\nu',x\rangle|\lesssim1$. Hence, in either
    of the two cases for $\mathcal L$, \eqref{scale}
    gives  $
        |\mathcal L\langle x,y\rangle|
        \lesssim\delta\mu^{-1}
        \lesssim\delta^{1/2}.
   $
    Since the higher derivatives of $\langle x,y\rangle$ vanish, using 
    \eqref{cosS} and \eqref{e:bounds-D-1/2}, we have 
    \[
        |\mathcal L^j\cos S_\ell(x,y)|
        \lesssim_j\delta^{\frac{1}{2}},
        \qquad j\ge1.
    \]
    On the other hand, $S_\ell\in I_\mu$, so
    $\sin S_\ell\sim\mu^{1/2}$. Thus, on the relevant range, i.e.,  $\cos  I_\mu$,  we have 
    $
        |(d/du)^m\arccos u |
        \lesssim_m\mu^{1/2-m}.
    $
    As before, applying the Fa\`a di Bruno formula to
    $S_\ell=\arccos(\cos S_\ell)$, we obtain
    \[
        |\mathcal L^N S_\ell(x,y)|
        \lesssim_N
        \sum_{m=1}^N\mu^{\frac{1}{2}-m}\delta^{\frac{m}{2}}
        \lesssim_N\delta^{\frac{1}{2}}\mu^{-\frac{1}{2}}.
    \]
    Here the last inequality follows since $ \delta^{1/2}/\mu  \lesssim (\mu'/\mu)^{1/2}\ll1.$
    For later use, the same argument applied to
    $u\mapsto(1-u^2)^{-1/2}$ also gives
    \begin{equation}\label{e:bounds-devSk}
        |\mathcal L^N(\sin S_\ell(x,y))^{-1}|
        \lesssim_N\mu^{-\frac{1}{2}},
        \qquad N\in\N_0. \qedhere
    \end{equation}
\end{proof}

Using Lemma \ref{lem:bounds-Spm}, we have the following bounds on derivatives of the amplitude function.

\begin{lem}\label{lem:bdsamp-1}
Let $\fr\times\fr'\in\fR_\delta^+$, where
$\fr\in\cR_{\delta,\nu}$ and $\fr'\in\cR_{\delta,\nu'}'$.
Let $\fs\in\mathfrak S(\fr')$ and $\ell\in\{1,2\}$.
Then
\[
\big|(\mathcal L_{\delta,\nu})^N A_\ell(x,y)\big|
\lesssim_N 1
\]
for every $N\in\N_0$ and $(x,y)\in 2\fr\times2\fr'$.
\end{lem}

\begin{proof}
Assume $(x,y)\in2\fr\times2\fr'$.
By \eqref{b:wtphirr's'}, we have
\[
\big|(\mathcal L_{\delta,\nu})^N
\varphi_{\delta,\fr,\fs}^+(x,y)\big|
\lesssim_N 1.
\]
Moreover, Lemma \ref{lem:bd-derivpsi}, the initial-scale bound
\eqref{e:bd-psi-top}, and Lemma \ref{lem:bounds-Spm}, together with the chain rule, give
\[
\big|(\mathcal L_{\delta,\nu})^N
\big(\psi_{\delta,\ell}^*(x,y,S_\ell(x,y))\big)\big|
\lesssim_N 1.
\]

It remains to control the denominator in \eqref{aell}.
Since $|\partial_t^2\Phi_\ell(x,y,0)|\sim1$, it is
sufficient to prove
\begin{align}\label{e:wtnudev-wcP}
\big|(\mathcal L_{\delta,\nu})^N
\partial_t^2\Phi_\ell(x,y,0)\big|
\lesssim_N 1.
\end{align}
From the definitions of $\tau_\ell$ and $\Phi_\ell$, we have
\[
\partial_t^2\Phi_\ell(x,y,0)
=
(-1)^{\ell+1}\delta^{-\frac{1}{2}}\mu^{\frac{1}{2}}
\frac{\cD^{\frac{1}{2}}(x,y)}{\sin S_\ell(x,y)}.
\]
Thus, \eqref{e:wtnudev-wcP} follows once we prove
\begin{align}\label{e:wtnudev-cP}
\big|(\mathcal L_{\delta,\nu})^N
\big(\cD^{\frac{1}{2}}(x,y)(\sin S_\ell(x,y))^{-1}\big)\big|
\lesssim_N\delta^{\frac{1}{2}}\mu^{-\frac{1}{2}}.
\end{align}
Taking $\mathcal L=\mathcal L_{\delta,\nu}$ in
\eqref{e:bounds-devSk} and combining the resulting estimate with
\eqref{e:bounds-D-1/2} gives
\eqref{e:wtnudev-cP}, and hence \eqref{e:wtnudev-wcP}.

Since $|\partial_t^2\Phi_\ell(x,y,0)|\sim1$, the chain rule
also gives
\[
\big|(\mathcal L_{\delta,\nu})^N
|\partial_t^2\Phi_\ell(x,y,0)|^{-\frac{1}{2}}\big|
\lesssim_N 1.
\]
The desired estimate now follows from \eqref{aell} and
the Leibniz rule.
\end{proof}

We next analyze the phase function $\Phi_\ell$. Since the proof of \eqref{e:L2est-fG} relies on a standard $TT^*$ argument, it is essential to understand the behavior of the mixed Hessian $\partial_x\partial_y^\intercal \Phi_\ell$. To state the result, we introduce some notation. Let us denote
\begin{align*}
    {\mathbf v}_\ell(x,y) = \cos S_\ell(x,y) y - x,\quad \mathbf v'_\ell(x,y) = \cos S_\ell(x,y) x - y.
\end{align*}
We occasionally omit the dependence on $x,y$ to lighten the notation. Using \eqref{pxpy0}, \eqref{pxpy1}, and \eqref{cosS}, note that ${\mathbf v}_\ell = (1/2)\partial_x\cD -(-1)^\ell \sqrt{\cD}\,y$ and $\mathbf v_\ell'= (1/2)\partial_y\cD  -(-1)^\ell  \sqrt{\cD}\,x$. Hence, by \eqref{e:sizeofxDyD},  \eqref{e:sizeofxDyD'},  and the inequality $|\cD| \lesssim \epz\mu\mu'$,
\begin{align}\label{e:sizeofwtab}
    |{\mathbf v}_\ell(x,y)| \sim (\mu\mu')^{\frac{1}{2}},\quad |\mathbf v_\ell'(x,y)| \sim \mu,\quad \ell\in\{1,2\}
\end{align}
on \(\supp\varphi_\delta^+\), which implies that \({\mathbf v}_\ell\)
and \(\mathbf v_\ell'\) are nonzero vectors.
 Moreover, since $|{\mathbf v}_\ell(x,y)- (1/2)\partial_x \cD(x,y)|\lesssim \delta^{1/2}$ and $|\mathbf v'_\ell(x,y)- (1/2)\partial_y \cD(x,y)|\lesssim \delta^{1/2}$, 
\Be 
\label{ab}  \mathbf n_\ell(x,y)   -  \fa(x,y)=O(\delta^{\frac{1}{2}}(\mu\mu')^{-\frac{1}{2}}), \quad   \mathbf n_\ell'(x,y) -  \fb (x,y)=O(\delta^{\frac{1}{2}}\mu^{-1})  .
\Ee
where  
\[ \mathbf n_\ell(x,y) =\frac{{\mathbf v}_\ell(x,y)}{|{\mathbf v}_\ell(x,y)|}, \quad  \mathbf n_\ell'(x,y) =   \frac{{\mathbf v}_\ell'(x,y)}{|{\mathbf v}_\ell'(x,y)|}. \]

We define 
\begin{align*}
    {\mathbf v}_\ell^\perp = {\mathbf v}_\ell' - \frac{\inp{{\mathbf v}_\ell}{{\mathbf v}_\ell'}}{|{\mathbf v}_\ell|^2}{\mathbf v}_\ell ,\qquad {\mathbf v}_\ell'^\perp = {\mathbf v}_\ell - \frac{\inp{{\mathbf v}_\ell}{{\mathbf v}_\ell'}}{|{\mathbf v}_\ell'|^2}{\mathbf v}_\ell',
\end{align*}
which are orthogonal to $\mathbf v_\ell$ and $\mathbf v_\ell'$, respectively. The following result states that the mixed Hessian of $\phi_\ell$ is of rank $1$, and that its image is spanned by ${\mathbf v}_\ell'^\perp$.

\begin{lem}\label{lem:mixedH}
Let $(x,y)\in\cB\times\cB'$ with $\cD(x,y)>0$, and let
$\ell\in\{1,2\}$. Set
$
    \phi_\ell(x,y)
    :=
    \cP\bigl(x,y,S_\ell(x,y)\bigr). 
$
Then
\begin{align}
\label{e:mixedH-unscaled}
    \partial_x\partial_y^\intercal\phi_\ell(x,y)
    &= 
    \frac{(-1)^{\ell}}{\sqrt{\cD(x,y)}\,\sin^3 S_\ell(x,y)}
    \mathbf v_\ell(x,y)
    \mathbf v_\ell'(x,y)^\intercal 
    -\frac{\mathbf I_2}{\sin S_\ell(x,y)}
   ,
\end{align}
where  $\mathbf I_2$ is the $2\times 2$  identity matrix. 
Moreover,
\begin{align}
\label{e:mixedH-v}
    \partial_x\partial_y^\intercal\phi_\ell(x,y)
    \mathbf v_\ell(x,y)
    &=0,
\\
\label{e:mixedH-vperp}
    \partial_x\partial_y^\intercal\phi_\ell(x,y)
    \mathbf v_\ell^\perp(x,y)
    &=
    \frac{(-1)^{\ell}(1-|x|^2)}{\sqrt{\cD(x,y)}\,\sin S_\ell(x,y)}
    \mathbf v_\ell'^\perp(x,y).
\end{align}
\end{lem}

One can prove this lemma following the argument in the proof of \cite[Lemma 2.13]{LR22}. For the reader's convenience, we include its proof in Appendix~\ref{app:mixed-hessian}.

Note that $| \mathbf v_\ell^\perp(x,y)|\sim \mu$ and $|\mathbf v_\ell'^\perp(x,y)|\sim (\mu\mu')^{1/2}$.
 Indeed, this follows since $\mathbf v_\ell -(1/2)\partial_x \cD=O(\delta^{1/2})$ and $\mathbf v_\ell' -(1/2)\partial_y \cD=O(\delta^{1/2})$ and 
$\inp{\partial_x \cD}{\partial_y \cD}=O(\delta)$.   Thus, by \eqref{e:mixedH-vperp}, we have 
    \begin{align*}
        \partial_x\partial_y^\intercal \Phi_\ell(x,y) \mathbf n_{\ell, \perp}(x,y) = L\,
          \delta^{-2}\mu^{\frac{3}{2}} \smp\,   \mathbf n'_{\ell, \perp}(x,y)
    \end{align*}
    with $L=L(x,y,\mu, \mu') $ such that $L\sim 1$, where 
    \[    \mathbf n'_{\ell, \perp}(x,y)=\frac{{\mathbf v}_\ell'^\perp(x,y)}{|{\mathbf v}_\ell'^\perp(x,y)|}, \qquad  \mathbf n_{\ell, \perp}(x,y)=\frac{{\mathbf v_\ell^\perp}(x,y)}{|{\mathbf v}_\ell^\perp(x,y)|}. \] 

 The lemma implies that  
\begin{align}\label{e:mixednunu'}
    |\inp{\nu}{\partial_x\partial_y^\intercal \Phi_\ell(x,y) \nu'}|\sim \delta^{-2}\mu^{\frac{3}{2}}(\mu')^{\frac{1}{2}}, \quad  (x,y)\in 2\fr\times 2\fr'
\end{align}
where $\nu=\theta(\fr)$ and $\nu'=\theta(\fr')$. 
To see this, write $\nu'=  \inp {\nu'}{ \mathbf n_\ell} \mathbf n_\ell  + \inp {\nu'}{\mathbf n_{\ell, \perp}}\mathbf n_{\ell, \perp}$,  $\nu=  \inp \nu{ \mathbf n_\ell'} \mathbf n_\ell'  + \inp \nu{ \mathbf n'_{\ell, \perp}} \mathbf n'_{\ell, \perp}$. Then, by Lemma \ref{lem:mixedH}  it suffices to  show  $|\inp{\nu}{\mathbf n'_{\ell, \perp}}|>c$ and $|\inp{\nu'}{\mathbf n_{\ell, \perp}}|>c$ for some constant $c>0$.  Indeed, using Lemma \ref{lem:locfab} and \eqref{ab},  
we obtain
\begin{align*}
    |\inp{\nu}{\mathbf n'_{\ell,\perp}}|
    &=|\inp{\fa}{\mathbf n'_{\ell,\perp}}|
    +O(\delta^{\frac{1}{2}}\mu^{-1})=|\inp{\fa_{\!\perp}}{\mathbf n'_\ell}|
    +O(\delta^{\frac{1}{2}}\mu^{-1})\\
    &=1+O(\delta^{\frac{1}{2}}(\mu\mu')^{-\frac{1}{2}}).
\end{align*}
For the last equality we used Lemma \ref{lem:alp}, i.e., $\inp{\mathbf n}{\mathbf n'}=O(\delta^{1/2}\mu^{-1})$, and $\mu'\ll \mu$. 
With a sufficiently small $\epz>0$,  this gives $|\inp{\nu}{\mathbf n'_{\ell,\perp}}|> c$, as desired. Similarly, we can obtain $|\inp{\nu'}{\mathbf n_{\ell, \perp}}|>c$. 
Consequently, \eqref{e:mixednunu'} follows. 

Lastly, we control derivatives of the phase function $\Phi_\ell$. The following bounds suffice for our purposes.

\begin{lem}\label{lem:bdsforPhi}
    Let $\theta(\fr)=\nu$, $\theta(\fr')=\nu'$ and let $\fs\in \mathfrak S(\fr')$. Let $\ell\in \{1,2\}$. Assume $\fr\times \fr'\in \fR^+_\delta$. Then,
\begin{align}\label{e:bdsforPhi}
        \big|\inp{\nu}{\partial_x}^N\inp{\nu'}{\partial_y}\Phi_\ell(x,y)\big|\lesssim \delta^{-2}\mu^{\frac{3}{2}}\smp(\delta^{-1}(\mu\mu')^{\frac{1}{2}})^{N-1}
    \end{align}
    for every $(x,y)\in 2\fr\times 2\fr'$ and $N\in \N$.
\end{lem}

\begin{proof}
    Since $\partial_t \cP(x,y,S_\ell) = 0$, by the chain rule we get
    \begin{align}\label{i:pyPhik}
        \partial_y \Phi_\ell(x,y) = \delta^{-\frac{3}{2}}\mu^{\frac{3}{2}}\frac{y \cos S_\ell - x}{\sin S_\ell} = \delta^{-\frac{3}{2}}\mu^{\frac{3}{2}}\frac{{\mathbf v}_\ell(x,y)}{\sin S_\ell}.
    \end{align}
    We claim that, for every $(x,y)\in 2\fr\times 2\fr'$ and $N\in \N$,
    \begin{align}
        \begin{aligned}\label{e:bds-wtaksin}
            \big|\inp{\nu}{\partial_x}^N\inp{{\mathbf v}_\ell(x,y)}{\nu'}\big|&\lesssim \delta^{\frac{1}{2}}(\delta^{-1}(\mu\mu')^{\frac{1}{2}})^N.
        \end{aligned}
    \end{align}
    Taking $\mathcal L=\mathcal L_{\delta,\nu}$ in
    \eqref{e:bounds-devSk} and using \eqref{scale}, we have
    \[
    \big|\inp{\nu}{\partial_x}^N(\sin S_\ell(x,y))^{-1}\big|
    \lesssim_N \mu^{-\frac{1}{2}}
    (\delta^{-1}(\mu\mu')^{\frac{1}{2}})^N.
    \]
    By \eqref{i:pyPhik}, this estimate and \eqref{e:bds-wtaksin} imply \eqref{e:bdsforPhi}. To show the claim, we write 
    \[
    \inp{{\mathbf v}_\ell(x,y)}{\nu'} = 2^{-1}\inp{\partial_x\cD(x,y)}{\nu'} -(-1)^\ell\sqrt{\cD(x,y)}\inp{y}{\nu'},
    \]
    Since $|\inp{y}{\nu'}|\sim 1$, by \eqref{e:bounds-D-1/2}, we get
    \begin{align*}
        \big|\inp{\nu}{\partial_x}^N\big((-1)^\ell\cD^{\frac{1}{2}}(x,y)\inp{y}{\nu'}\big)\big|\lesssim \delta^{\frac{1}{2}}(\delta^{-1}(\mu\mu')^{\frac{1}{2}})^N.
    \end{align*}
    Thus, it is sufficient for us to obtain
    \[
    |\inp{\nu}{\partial_x}^N(\inp{\partial_x\cD(x,y)}{\nu'})|
    \lesssim \delta^{\frac{1}{2}}
    \bigl(\delta^{-1}(\mu\mu')^{\frac{1}{2}}\bigr)^N.
    \]
    However, notice that the statement is obvious when $N\ge 2$, because the derivative vanishes in this case. When $N=1$, we calculate
    \begin{align}\label{i:vpxpxDv'}
        \inp{\nu}{\partial_x}(\inp{\partial_x\cD(x,y)}{\nu'}) = 2\big(\inp{y}{\nu'}\inp{y}{\nu} - \inp{\nu}{\nu'}\big).
    \end{align}
    To bound the term on the right-hand side, we denote
    \begin{align*}
        \nu = b_1\fa(x,y) + c_1\fa_{\!\perp}(x,y),\quad \nu' = b_2\fb(x,y) + c_2\fa'_{\!\perp}(x,y),
    \end{align*}
    with constants $b_1, b_2, c_1$, and $c_2$ such that $|b_1|,\ |b_2|\sim 1$ and $|c_1|,\ |c_2|\lesssim \delta^{1/2}\mu^{-1}$. From this and Lemmas \ref{lem:alp} and \ref{lem:directionab}, the bounds $|\inp{\nu}{\nu'}|\lesssim \delta^{1/2}\mu^{-1}$ and $|\inp{y}{\nu}|\lesssim (\mu'/\mu)^{1/2}$ follow. Substituting these bounds into \eqref{i:vpxpxDv'} gives the desired estimate for the derivative of $\inp{\partial_x\cD(x,y)}{\nu'}$.
\end{proof}

\subsection{The \texorpdfstring{$L^2$}{L2} estimate for the stationary pieces} 
We now prove \eqref{e:fP-dkap} with $\kappa = +$. As noted above, it is sufficient to establish \eqref{e:L2est-fG}. Recall that $\nu$, $\nu'$ are unit vectors such that $\fr\in \cR_{\delta, \nu}$, $\fr'\in \cR_{\delta, \nu'}'$, respectively. Let $L_\fr$, $L_{\fs}$ be affine transformations defined by
\Be
\label{LL}
\begin{aligned}
    L_\fr(x) & = \delta (\mu\mu')^{-\frac{1}{2}}x_1\,\nu+ \delta^{\frac{1}{2}}\mu^{\frac{1}{2}}(\mu')^{-\frac{1}{2}}x_2\,\nu_\sprp  + c_\fr , \\
    L_{\fs}(y) & =\delta \mu^{-1} y_1\,\nu'+ \delta^{\frac{1}{2}}\smp\mu^{-\frac{1}{2}}y_2\,\nu'_\sprp + c_{\fs} .
\end{aligned}
\Ee
Recalling \eqref{phiell} and \eqref{aell}, we then define an operator $\bar{\mathfrak O}_\ell  $ by
\begin{align*}
   \bar{\mathfrak O}_\ell  f(x) = \int A_{\ell}(L_\fr(x),L_{\fs}(y)) e^{i(\lambda\delta^{\frac{3}{2}}\mu^{-\frac{3}{2}})\Phi_\ell(L_\fr(x), L_{\fs}(y))} f(y) dy.
\end{align*}
Changing  variables  $(x,y)\to ( L_\fr(x),  L_{\fs}(y) )$, we have
\begin{align}\label{i:G->wtG}
    \big\|{\mathfrak O}_\ell \big\|_{2\to 2} = \delta^{\frac{3}{2}}\mu^{-\frac{3}{4}}(\mu')^{-\frac{1}{4}}\big\|\bar{\mathfrak O}_\ell \big\|_{2\to 2}.
\end{align}
Also note that the function $A_{\ell}(L_\fr(x),L_{\fs}(y))$ is supported in $[-\epz, \epz]^2\times[-\epz, \epz]^2$.
To estimate  $\bar{\mathfrak O}_\ell  $, freezing $(x_2, y_2)\in [-\epz, \epz]^2$, we define an operator 
\begin{align*}
  \bar{\mathfrak O}_\ell^{x_2,y_2} g(u) = \int  A_\ell^{x_2, y_2}(u,v) e^{i(\lambda\delta^{\frac{3}{2}}\mu^{-\frac{3}{2}})\Phi^{x_2, y_2}_\ell(u,v)} g(v) dv,\quad u,\,v\in \R,
\end{align*}
where
\begin{align*}
    A_\ell^{x_2, y_2}(u,v) &:= A_{\ell}(L_\fr(u, x_2),L_{\fs}(v, y_2)), \\
  \Phi^{x_2, y_2}_\ell(u,v) &:= \Phi_\ell(L_\fr(u, x_2),L_{\fs}(v, y_2)).
\end{align*}
By Lemma \ref{lem:bdsamp-1} and \eqref{e:bdsforPhi}, we have 
\begin{align*}
    |\partial_u^N A_\ell^{x_2, y_2}(u,v)| \lesssim_N 1,\quad |\partial_u^N\partial_v  \Phi_\ell^{x_2, y_2}(u,v)| \lesssim_N 1,
\end{align*}
for $N\in \N$ and $(x_2, y_2)\in [-\epz, \epz] \times [-\epz, \epz]$, with a constant $C_N$ only depending on $N$. Moreover, by \eqref{e:mixednunu'} and \eqref{LL}, we have \(|\partial_u\partial_v\Phi_\ell^{x_2,y_2}(u,v)|\sim1\) on \(\supp A_\ell^{x_2,y_2}\), uniformly in \(x_2,y_2\). Therefore, H\"ormander's oscillatory-integral estimate \cite[Theorem 1.1]{H73}, applied with \(p=2\) and \(N=\lambda\delta^{3/2}\mu^{-3/2}\), gives
\begin{align*}
    \|\bar{\mathfrak O}_\ell^{x_2, y_2} g\|_{L^2}\lesssim \lambda^{-\frac{1}{2}}\delta^{-\frac{3}{4}}\mu^{\frac{3}{4}}\|g\|_{L^2}.
\end{align*}

Since \(x_2\) and \(y_2\) range over fixed intervals, Minkowski's
inequality, the uniform slice estimate above, and Cauchy--Schwarz give
\[
\begin{aligned}
\|\bar{\mathfrak O}_\ell f\|_{L^2_{u,x_2}}
&\le
\int_{-\epz}^{\epz}
\|\bar{\mathfrak O}_\ell^{x_2,y_2}
f(\,\cdot\,,y_2)\|_{L^2_{u,x_2}}\,dy_2\\
&\lesssim
\lambda^{-\frac12}\delta^{-\frac34}\mu^{\frac34}
\|f\|_{L^2_{v,y_2}}.
\end{aligned}
\]
Combining this with \eqref{i:G->wtG} yields \eqref{e:L2est-fG}. This completes the proof.

\medskip

Together with the almost-orthogonality estimates of Section~\ref{sec:alortho}, Proposition~\ref{prop:L2est} supplies the local bounds used in the summation of Section~\ref{sec:conclude}.
 \section{Completion of the proof} 
\label{sec:conclude}

Putting together the estimates obtained in the previous sections, we conclude the proof of Proposition \ref{main-p}. 

\subsection{Proof of  Proposition \ref{main-p}}\label{subsec:proof-main-p}
With the decomposition \eqref{i:decomp-main},  
 the proof of  Proposition \ref{main-p}  is reduced to obtaining the following two propositions.

\begin{prop}\label{prop:main-components1-0}
Under the hypotheses and notation of Proposition~\ref{main-p}, let $\mu,\mu'$ satisfy \eqref{mm} and $\dels\le\delta<\delt$. Set
$B_{1,\infty}(\delta)$
$:=\allowbreak\lambda^{-1/2}\allowbreak\delta^{-1/4}\allowbreak\mu^{-1/4}$. 
 Then, we  have
        \begin{align}
            \|\fP_{\lambda,\delta}^\kappa\|_{1\to \infty}&\lesssim
            B_{1,\infty}(\delta),
            \quad  \kappa\in \{+,-,\circ\},
            \label{1to0}
            \\[3pt]
        \|\fP_{\lambda, \dels}\|_{1\to \infty}&\lesssim
        B_{1,\infty}(\dels). 
         \label{1to0s}
        \end{align}
       \end{prop} 
       
      The $L^1$--$L^\infty$ estimates, which  basically rely on oscillatory integral estimates, are relatively simpler once we have appropriate decompositions.  The key estimates are the following $L^2$ estimates, for which we heavily make use of orthogonality among the operators $\{ \fP_{\delta, \fr,\fr'}^\kappa\}_{\fr\times\fr'\in \fR_{\delta}^\kappa}$,  $\{ \fP_{\dels, \fr,\fr'}\}_{\fr\times\fr'\in \fR_{\dels}^\circ}$ (see 
        \eqref{pdk}, \eqref{pds}). 
        
        \begin{prop} \label{prop:main-components2-2}
       Under the hypotheses and notation of Proposition~\ref{main-p}, let $\mu,\mu'$ satisfy \eqref{mm} and $\dels\le\delta<\delt$. Set 
       \[
       B_2(\delta)
       :=\lambda^{-1}\delta^{\frac{1}{4}}(\mu')^{-\frac{1}{2}}
       \min(\delta^{\frac{1}{4}},(\mu')^{\frac{1}{2}}).
       \]
       Then, we have 
            \begin{align}
      \label{2to2}      \|\fP_{\lambda,\delta}^\kappa\|_{2\to 2}
            &\lesssim
            B_2(\delta),  
      \quad  \kappa\in \{+,-,\circ\}, \\[3pt]            
       \label{2to2s}   \|\fP_{\lambda, \dels}\|_{2\to 2}
       &\lesssim
       B_2(\dels). 
        \end{align} 
\end{prop}

It should be noted that the bound \eqref{2to2s} corresponds to \eqref{2to2} with 
$\delta=\dels$.   
Once we have the estimates in  Propositions  \ref{prop:alortho-1}, \ref{prop:alortho-2},  and  \ref{prop:L2est}, the estimates for $\fP_{\lambda,\delta}^\kappa$ and $\fP_{\lambda,\dels}$ follow essentially the same line of argument.   
The two cases will be handled in the same manner. 

\medskip

Before proceeding, we prove the endpoint estimate \eqref{e:main-fP}, assuming that Proposition \ref{prop:main-components1-0} and Proposition \ref{prop:main-components2-2} hold.  
\begin{proof}[Deduction of Proposition \ref{main-p} from
Propositions \ref{prop:main-components1-0} and
\ref{prop:main-components2-2}]
Interpolation between \eqref{1to0} and \eqref{2to2} gives, for
$\kappa\in \{+,-,\circ\}$, 
\[ 
\|\fP_{\lambda, \delta}^\kappa\|_{q_e' \to q_e} \lesssim \begin{cases}
       \quad  \lambda^{-\frac{4}{5}} \delta^{\frac{1}{20}}\mu^{-\frac{1}{10}}, &  \   (\mu')^2\le \delta, \\
                \lambda^{-\frac{4}{5}}\delta^{\frac{1}{5}}(\mu')^{-\frac{3}{10}}\mu^{-\frac{1}{10}}, &   \  \delta \le (\mu')^2.
    \end{cases}
    \]
 After splitting  the sum $\sum_{ \dels\le \delta < \delt}=\sum_{(\mu')^2\le \delta}+\sum_{ \delta \le (\mu')^2}$, the triangle inequality and summation over dyadic numbers $\delta\in [\dels, \delt]$   gives 
	    \[ \sum_{ \dels\le \delta < \delt}  \|\fP_{\lambda,\delta}^\kappa\|_{q_e' \to q_e} \lesssim    \lambda^{-\frac{4}{5}}\biggl(\frac{\mu'}{\mu}\biggr)^{\frac{1}{20}}\]
    for $\kappa\in \{+,-,\circ\}$. 
 Similarly, interpolation  between \eqref{1to0s} and \eqref{2to2s}  yields 
\begin{align*}
    \|\fP_{\lambda, \dels}\|_{q_e' \to q_e}  \lesssim     
    \begin{cases}
       \quad  \lambda^{-\frac{4}{5}} \dels^{\frac{1}{20}}\mu^{-\frac{1}{10}}, &  \   (\mu')^2\le \dels, \\
                \lambda^{-\frac{4}{5}}\dels^{\frac{1}{5}}(\mu')^{-\frac{3}{10}}\mu^{-\frac{1}{10}}, &   \  \dels \le (\mu')^2.
    \end{cases}
 \end{align*}
In both cases $(\mu')^2\le \dels$  and $ \dels \le (\mu')^2$, $
    \|\fP_{\lambda, \dels}\|_{q_e' \to q_e} \lesssim  \lambda^{-4/5}(\mu'/\mu)^{1/20}.  
$
Indeed, the inequality in the former case follows since $\dels\sim \lambda^{-2/3}\mu$ and $\lambda^{-2/3}\lesssim \mu'$ (see \eqref{mm}).

 Thus,  recalling \eqref{i:decomp-main}, we combine  all these estimates  to get \eqref{e:main-fP}.  This completes the proof. 
\end{proof}

Combining  the results from Sections \ref{sec:alortho} and  \ref{sec:individual}, we prove Propositions \ref{prop:main-components1-0} and  \ref{prop:main-components2-2}. 
We begin by noting that 
\begin{align}
\label{d-pdk}
\fP_{\lambda,\delta}^\kappa&=\sum_{\fr\times\fr'\in \fR^\kappa_{\delta}} \sum_{\fs\in\mathfrak S(\fr')} \fP_{\delta,\fr,\fs}^\kappa,  \quad \kappa\in \{+,-,\circ\}
\\ 
\label{d-pds}
  \fP_{\lambda, \dels}&=\sum_{\fr\times\fr'\in \fR_{\dels}^\circ} \sum_{\fs\in\mathfrak S(\fr')} \fP_{\dels,\fr,\fs},
\end{align} 
which follow from  the decompositions \eqref{pdk}, \eqref{pds}, \eqref{pdts},  and \eqref{pdts'}. 
We first prove Proposition \ref{prop:main-components1-0}. 

 \begin{proof}[Proof of Proposition \ref{prop:main-components1-0}]  We have that, for $\dels\le \delta < \delt$ and $\kappa\in \{+,-,\circ\}$,
\begin{align}\label{e:bds-1infty}
    \|\fP^\kappa_{\delta,\fr,\fs}\|_{1\to \infty}\lesssim \lambda^{-\frac{1}{2}}\delta^{-\frac{1}{4}}\mu^{-\frac{1}{4}}
\end{align}
holds for every pair $(\fr,\fs)$ such that $\fr\times \fr'\in \fR^\kappa_\delta$ and $\fs\in\mathfrak S(\fr')$. 
Indeed, the estimate for $\kappa=\circ,-$ follows from \eqref{e:ptbd-fPcn}, while the estimate for $\kappa=+$ follows from \eqref{p+circ} and \eqref{i:exp-fP+},
since $\fP_{\delta,\fr,\fs}^+=\sum_{\ell=0}^2\fP_{\ell,\delta,\fr,\fs}^+$. 

Similarly, from \eqref{e:ptbd-fP*} we have  the bound 
\Be
\label{e:bds-10} \|\fP_{\dels,\fr,\fs}\|_{1\to \infty}\lesssim \lambda^{-\frac{1}{3}}\mu^{-\frac{1}{2}}
\Ee
for $\fr\times \fr'\in \fR_{\dels}^\circ$ and $\fs\in\mathfrak S(\fr')$.
Since the rectangles $\fr\times \fr'$ in $\fR^\kappa_\delta$ have uniformly bounded overlap, as do the rectangles $\fs\in\mathfrak S(\fr')$, we have
\begin{align*}
    \|\fP_{\lambda,\delta}^\kappa\|_{1\to \infty} &\le C\sup_{\fr\times\fr'\in \fR^\kappa_\delta,\ \fs\in \mathfrak S(\fr')} \|\fP^\kappa_{\delta,\fr,\fs}\|_{1\to \infty} \\
    \|\fP_{\lambda, \dels}\|_{1\to \infty}&\le C\sup_{\fr\times\fr'\in \fR_{\dels}^\circ,\ \fs\in \mathfrak S(\fr')}\|\fP_{\dels,\fr,\fs}\|_{1\to \infty}
\end{align*}
with a constant $C>0$. Combining this with \eqref{e:bds-1infty} and  \eqref{e:bds-10} yields the desired estimates. 
      \end{proof}

\subsection{Global \texorpdfstring{$L^2$}{L2} estimates: Proof of Proposition \ref{prop:main-components2-2}}\label{subsec:global-L2}
In order to prove the estimates \eqref{2to2} and \eqref{2to2s},  we need to exploit orthogonality between $ \fP_{\delta,\fr,\fs}^\kappa$ and $ \fP_{\dels,\fr,\fs}$ appearing in \eqref{d-pdk} and \eqref{d-pds}.  To this end, we combine the estimates in Sections \ref{sec:alortho} and \ref{sec:individual} using Lemma \ref{lem:cotlarstein}.

The proofs of \eqref{2to2} and \eqref{2to2s} are essentially the same. We first prove \eqref{2to2} and  make a brief remark about how to prove \eqref{2to2s}. 

\medskip

Before applying Proposition \ref{prop:alortho-1} to \eqref{d-pdk} and \eqref{d-pds}, we partition the direction nets and the associated tilings into finitely many subfamilies on which the required sparsity conditions hold. It is enough to prove the estimates on each subfamily, since the triangle inequality then recovers the full family with a harmless fixed loss. We may write 
\[  \fR^\kappa_\delta=\bigcup_{(\nu, \nu')\in \Theta\times\Theta'}       \{(\fr,\fr') \in \fR^\kappa_\delta :  \theta(\fr)=\nu,  \quad \theta(\fr')=\nu' \}.     \]  
Since $\Theta$ and $\Theta'$ are $\epz\delta^{1/2}\mu^{-1}$-separated, each can be partitioned into \(O_{K,\epz}(1)\) subfamilies that are $K\delta^{1/2}\mu^{-1}$-separated. Fixing one pair of subfamilies, we may assume
\Be 
\label{nunu} |\nu-\bar\nu|>K\delta^{\frac{1}{2}}\mu^{-1}, \qquad
|\nu'-\bar\nu'|>K\delta^{\frac{1}{2}}\mu^{-1} \Ee
for distinct $\nu,\bar\nu\in\Theta$ and distinct $\nu',\bar\nu'\in\Theta'$.  
Summing over the \(O_{K,\epz}(1)\) pairs of subfamilies causes only a harmless fixed loss.

Similarly, after partitioning $\cR_{\delta,\nu}$ and $\cR_{\delta,\nu'}'$ into finitely many subfamilies, we may assume
\Be 
\label{kfk}  K\fr \cap K\bar \fr=\emptyset,  \quad K\fr' \cap K\bar \fr'=\emptyset \Ee
for each $(\nu,\nu')$, whenever $\fr, \bar\fr \in \cR_{\delta,\nu}$ and
$\fr', \bar\fr' \in \cR_{\delta,\nu'}'$. Here 
$K\fr$, $K\bar \fr$, $K\fr'$, and $K\bar \fr'$ denote the dilates of $\fr$, $\bar\fr$, $\fr'$, and $\bar\fr'$, respectively, by a factor of $K$ about their centers.
This is easy to see. Indeed,  for all  $\nu$,   $\cR_{\delta,\nu}$ can be identified with $\cR_{\delta,\nu_\zc}$ via rotation.  
Clearly, $\cR_{\delta,\nu_\zc}$ can be partitioned into $O(K^2)$ subcollections $\cR$, each satisfying $K\fr \cap K\bar \fr=\emptyset$ for all distinct $\fr,\bar\fr\in \cR$. Transporting these subcollections 
to each $\cR_{\delta,\nu}$ via the corresponding rotation gives the desired subfamilies. One can repeat the same argument for $\cR_{\delta,\nu'}'$. 
This provides a partition of $\bigcup_{(\nu, \nu')\in \Theta\times\Theta'} \cR_{\delta,\nu}\times \cR_{\delta,\nu'}'$ into $O(K^4)$ subfamilies. 
Now, considering intersections of those subfamilies and $\mathfrak R^\kappa_\delta$, we may assume \eqref{kfk} at the expense of an $O(K^4)$ factor in the bound.

After fixing one such subfamily, we retain the notation $\Theta$, $\Theta'$, and $\mathfrak R^\kappa_\delta$; the conditions \eqref{nunu} and \eqref{kfk} continue to hold. Recalling \eqref{d-pdk}, we consider 
\[
\| \sum_{\fr\times\fr'\in \fR^\kappa_\delta} \sum_{\fs\in\mathfrak S(\fr')} \fP_{\delta,\fr,\fs}^\kappa\|_{2\to 2},   \quad \kappa\in \{+,-,\circ\}. \]

    Fix $\fr_\zc\times \fr_\zc' \in \fR^\kappa_\delta$ and $\fs_\zc\in \mathfrak S(\fr_\zc')$.  Let us set 
      \[\mathbf c_{\fr, \fs}=\big\|\fP_{\delta,\fr_\zc,\fs_\zc}^\kappa\big(\fP_{\delta,\fr,\fs}^\kappa\big)^*\big\|_{2\to 2}^{\frac{1}{2}}, \quad   \mathbf d_{\fr, \fs}=\big\|\big(\fP_{\delta,\fr_\zc,\fs_\zc}^\kappa\big)^*\fP_{\delta,\fr,\fs}^\kappa\big\|_{2\to 2}^{\frac{1}{2}}.\]  
  By Lemma \ref{lem:cotlarstein}, the desired estimate \eqref{2to2} follows, once we obtain the following estimates: 
    \begin{align}\label{e:sumPP*-1}
        \sum_{\fr\times \fr'\in \fR^\kappa_\delta}\sum_{\fs\in\mathfrak S(\fr')} \mathbf c_{\fr, \fs} &\lesssim \lambda^{-1} \mu^{\frac{1}{4}}(\mu')^{-\frac{3}{4}} \min(\delta^{\frac{1}{2}}, \mu'), 
   \quad      \kappa\in \{+,-,\circ\}, \\[3pt]
        \label{e:sumP*P-1}
        \sum_{\fr\times \fr'\in \fR^\kappa_\delta}\sum_{\fs\in\mathfrak S(\fr')}  \mathbf d_{\fr, \fs} &\lesssim   \lambda^{-1}\delta^{\frac{1}{2}}(\mu\mu')^{-\frac{1}{4}}
    \end{align}
     with the implicit constants independent of $(\fr_\zc,\fr_\zc',\fs_\zc)$.   
     Indeed, the geometric mean of the right-hand sides in \eqref{e:sumPP*-1} and \eqref{e:sumP*P-1} is
     \[
     \lambda^{-1}\delta^{\frac14}(\mu')^{-\frac12}
     \min\{\delta^{\frac14},(\mu')^{\frac12}\}=B_2(\delta).
     \]
     In what follows 
   we provide the proofs of \eqref{e:sumPP*-1} and \eqref{e:sumP*P-1}, separately. 
    
    \smallskip
     
  \subsection*{Proof of \texorpdfstring{\eqref{e:sumPP*-1}}{the first Cotlar--Stein sum}}
We prove \eqref{e:sumPP*-1} first.   
        For   $\nu'\in\Theta'$,  
we define a sub-collection $\fT(\fr'_\zc,\nu')\subset \mathfrak R^\kappa_\delta$ by 
 \begin{align}
 \label{ftf'}
    \fT(\fr'_\zc,\nu') = \{\fr\times \fr'\in \mathfrak R^\kappa_\delta:   2\fr'\cap 2\fr'_\zc\neq \emptyset, \  \theta(\fr')=\nu' \},
\end{align}
 Consequently,  it is clear that 
\[  \{\fr\times \fr'\in \mathfrak R^\kappa_\delta:   2\fr'\cap 2\fr'_\zc\neq \emptyset\}\subset \bigcup_{\nu'\in \Theta'}   \fT(\fr'_\zc,\nu').\] 

Note that  $\fP_{\delta,\fr_\zc,\fr'_\zc}^\kappa\big(\fP_{\delta,\fr,\fr'}^\kappa\big)^*=0 $  if  $2\fr'_\zc \cap 2\fr'=\emptyset$. 
In fact,     $\fP_{\delta,\fr_\zc,\fs_\zc}^\kappa\big(\fP_{\delta,\fr,\fs}^\kappa\big)^*=0 $  if  $2\fr'_\zc \cap 2\fr'=\emptyset$, 
whenever $\fs_\zc\in \mathfrak S(\fr_\zc')$ and $\fs \in \mathfrak S(\fr')$.  
Thus, it follows that 
\[            \sum_{\fr\times \fr'\in \fR^\kappa_\delta}\sum_{\fs\in\mathfrak S(\fr')} \mathbf c_{\fr, \fs} 
                  \le  \sum_{\nu'\in \Theta'} \sum_{\fr\times\fr'\in \fT (\fr_\zc',\nu')}\sum_{\fs\in\mathfrak S(\fr')} \mathbf c_{\fr, \fs}.\]
        Also note $\fP_{\delta,\fr_\zc,\fs_\zc}^\kappa\big(\fP_{\delta,\fr,\fs}^\kappa\big)^*=0$ if $2\fs_\zc\cap 2\fs=\emptyset$. Thus, 
        setting 
   \[  \mathfrak S(\fr', \fs_\zc ):=\{  \fs\in \mathfrak S(\fr'):   2\fs_\zc\cap\,2\fs\neq \emptyset\},\]   we may replace  $\mathfrak S(\fr')$ with $\mathfrak S(\fr', \fs_\zc )$ 
   and, moreover, $\fT (\fr_\zc',\nu')$ with 
   \[  \fT (\fr_\zc',\nu', \fs_\zc):=  \big\{\fr\times\fr'\in  \fT (\fr_\zc',\nu'): \text{$\exists\,\fs\in \mathfrak S(\fr')$ s.t. $2\fs_\zc\cap\,2\fs\neq \emptyset$}\big\}.\]  
   Therefore, we now have 
        \begin{align*}
            \sum_{\fr\times \fr'\in \fR^\kappa_\delta}\sum_{\fs\in\mathfrak S(\fr')} \mathbf c_{\fr, \fs} 
                  &\le  
                  \sum_{\nu'\in \Theta'} \sum_{\fr\times\fr'\in \fT (\fr_\zc',\nu', \fs_\zc)}\sum_{\fs\in\mathfrak S(\fr', \fs_\zc)} \mathbf c_{\fr, \fs}. 
                  \end{align*}
      
 To simplify the notation, we set   
    \[   \bar{\mathbf c}_{\fr, \fr'}=\sum_{\fs\in \mathfrak S(\fr', \fs_\zc )}  \mathbf c_{\fr, \fs},\qquad    
       \bar{\mathbf C}_{\nu'}= \sum_{\fr\times\fr'\in \fT (\fr_\zc',\nu', \fs_\zc)}  
     \bar{\mathbf c}_{\fr, \fr'}. 
    \]    
  and  
  \[   \Theta' (\nu'_\zc, B):= \{ \nu'\in \Theta' : \ | \nu'-\nu'_\zc| < B\delta^{\frac{1}{2}}(\mu\mu')^{-\frac{1}{2}}  \}.\]  
     Then,    we have 
       \begin{align*}
       \sum_{\fr\times \fr'\in \fR^\kappa_\delta}\sum_{\fs\in\mathfrak S(\fr')}  \mathbf c_{\fr, \fs} \le \sum_{\nu'\in \Theta'}  \bar{\mathbf C}_{\nu'}= I+ I\!I 
       \end{align*}
    with a sufficiently large constant $B>0$, 
    where 
    \[   I:=  \sum_{\nu'\in \Theta' (\nu'_\zc, B)}    \bar{\mathbf C}_{\nu'}, \qquad  I\!I:=\sum_{\nu'\not\in  \Theta' (\nu'_\zc, B)} \bar{\mathbf C}_{\nu'}.\] Note that $|\nu'-\nu'_\zc|\lesssim  (\mu'/\mu)^{1/2}$ by the definition of $\Theta'$. Thus, $I\!I$ vanishes when $\delta\gg (\mu')^2$. 
    
    To estimate $I$, we use Proposition \ref{prop:L2est}, which gives $\mathbf c_{\fr, \fs}\lesssim \lambda^{-1}\delta^{1/2}(\mu\mu')^{-1/4}$. 
    Consequently, we get
    \begin{align*}
     I \lesssim   \sum_{\nu'\in \Theta' (\nu'_\zc, B)} \sum_{\fr\times\fr'\in \fT (\fr_\zc',\nu', \fs_\zc)} \sum_{\fs\in \mathfrak S(\fr', \fs_\zc)} \lambda^{-1}\delta^{\frac{1}{2}}(\mu\mu')^{-\frac{1}{4}}.
    \end{align*}
    
 Recall $\fs_\zc\in \mathfrak S(\fr_\zc')$. Thus, $\fs_\zc$ and $\fs$ are rectangles contained in $2\fr_\zc'$ and $2\fr'$, respectively. Their dimensions are about $\delta^{1/2}\mu^{-1/2}(\mu')^{1/2}\times \delta\mu^{-1}$, whereas the dimensions of $2\fr_\zc'$ and $2\fr'$ are about $\delta^{1/2}\times \delta\mu^{-1}$. Since $\delta^{1/2}\mu^{-1/2}(\mu')^{1/2}\gg \delta\mu^{-1}$, 
it is clear that 
    \begin{align}\label{e:card-s2'}
         \#    \big\{\fs\in \mathfrak S(\fr'): 2\fs_\zc\cap\,2\fs\neq \emptyset\big\}\lesssim 1,  
    \end{align}
    for any $\fr'$.  Moreover, we have 
    \Be
    \label{e:card-rs2'}
      \#      \fT (\fr_\zc',\nu', \fs_\zc)  \lesssim 1 + \delta^{-\frac{1}{2}}(\mu\mu')^{\frac{1}{2}}|\nu'-\nu'_\zc|, 
    \Ee
 To show \eqref{e:card-rs2'}, we consider 
  \[\mathcal  T'(\nu', \fs_\zc)=\big\{ \fr': \theta(\fr')=\nu', \  \text{$\exists\,\fs\in \mathfrak S(\fr')$ s.t. $2\fs_\zc\cap\,2\fs\neq \emptyset$}\big\}.\]
 First note that 
 \[  \# \mathcal  T'(\nu', \fs_\zc) \lesssim 1 + \delta^{-\frac{1}{2}}(\mu\mu')^{\frac{1}{2}}|\nu'-\nu'_\zc|. \] 
 Indeed, we only need to count
the number of $\fr'\in \cR_{\delta, \nu'}'$ which intersect $2\fs_\zc$. The longer sides of $\fr'$ and $\fr_\zc'$ form an angle comparable to $|\nu'-\nu'_\zc|$, and $\fs_\zc\subset2\fr_\zc'$.    
The longer side   of   $\fs_\zc$  has length about $\delta^{1/2}\mu^{-1/2}(\mu')^{1/2}$. Hence,   the number of $\fr'\in \cR_{\delta, \nu'}'$ intersecting $2\fs_\zc$ is $\lesssim  \delta^{1/2}\mu^{-1/2}(\mu')^{1/2} |\nu'-\nu'_\zc|/(\delta\mu^{-1})$, provided that this number is bigger than $1$.  This shows the desired inequality. 
Now, let  $\fr\times\fr'\in  \fT (\fr_\zc',\nu',\fs_\zc)$, then $\fr'\in  \mathcal  T'(\nu', \fs_\zc)$. For a given $\fr'$,  since $\fr\times\fr'\in \mathfrak R_\delta^\kappa$, 
there are only $O(1)$ choices of $\fr$ by \eqref{sharp0} in Lemma \ref{lem:cardinal-TT'}. Combining these together, we therefore obtain the inequality \eqref{e:card-rs2'}.

Since $\ | \nu'-\nu'_\zc| < B\delta^{1/2}(\mu\mu')^{-1/2}$, using the two inequalities \eqref{e:card-s2'} and \eqref{e:card-rs2'}, 
we obtain 
    \begin{align}\label{e:est-sumI}
       I  \lesssim  \sum_{\nu'\in \Theta' (\nu'_\zc, B)} \lambda^{-1}\delta^{\frac{1}{2}}(\mu\mu')^{-\frac{1}{4}}
        \lesssim   \lambda^{-1} \mu^{\frac{1}{4}}(\mu')^{-\frac{3}{4}} \min(\delta^{\frac{1}{2}}, \mu').
       \end{align}
    Indeed, for the second inequality  note that the set $\Theta'$ has diameter $\lesssim (\mu'/\mu)^{1/2}$ and the points in $\Theta'$ are separated by $\sim \delta^{1/2}\mu^{-1} $.  Thus,  
       $
        \#\big\{\nu'\in \Theta': |\nu'-\nu'_\zc| < B\delta^{1/2}(\mu\mu')^{-1/2}\big\}\lesssim  \min(\delta^{-1/2}(\mu\mu')^{1/2}, (\mu/\mu')^{1/2}). 
       $

    Next, we estimate $I\!I$. For the purpose, we may assume  $\delta \lesssim  (\mu')^2$, since $I\!I = 0$ otherwise, as mentioned above.  
	   Since $\ | \nu'-\nu'_\zc| \ge  B\delta^{1/2}(\mu\mu')^{-1/2}$ for sufficiently large $B$,    we may apply  $ii)$ in Proposition \ref{prop:alortho-1}  to $\mathbf c_{\fr, \fs}$  to get  
    \begin{align*}
        I\!I \lesssim   \sum_{\nu'\not\in \Theta' (\nu'_\zc, B)} \sum_{\fr\times\fr'\in \fT (\fr_\zc',\nu', \fs_\zc)} \sum_{\fs\in \mathfrak S(\fr', \fs_\zc)}  \frac{\delta^2\mu^{-\frac{7}{4}}(\mu')^{-\frac{1}{4}}}{(|\nu_\zc - \nu |\lambda^{}\delta\mu^{-1}\smp)^{\frac{N}{2}}}. 
            \end{align*}
   By Lemmas \ref{lem:locfab} and \ref{lem:alp}, we have $|\nu_\sprp+\nu'|$, $|(\nu_\zc)_\sprp+\nu_\zc'|\lesssim \delta^{1/2}\mu^{-1}+\delta\mu^{-3/2}(\mu')^{-1/2}\lesssim\delta^{1/2}\mu^{-1}\ll\delta^{1/2}(\mu\mu')^{-1/2}$. 
	   Since $|\nu'-\nu'_\zc| \ge B\delta^{1/2}(\mu\mu')^{-1/2}$, it follows that $|\nu-\nu_\zc|= |\nu_\sprp-(\nu_\zc)_\sprp|\sim |\nu'-\nu'_\zc|$. Thus, combining the above estimate with \eqref{e:card-s2'} and \eqref{e:card-rs2'},   we have
    \begin{align*}
        I\!I &\lesssim \sum_{\substack{\nu'\in \Theta': |\nu'-\nu'_\zc| \ge B\delta^{\frac{1}{2}}(\mu\mu')^{-\frac{1}{2}}}} \frac{(\lambda\delta^{\frac{3}{2}}\mu^{-\frac{3}{2}})^{-\frac{N}{2}} \delta^2\mu^{-\frac{7}{4}}(\mu')^{-\frac{1}{4}} }{(\delta^{-\frac{1}{2}}(\mu\mu')^{\frac{1}{2}}|\nu_\zc'-\nu'|)^{\frac{N-2}{2}}}. 
    \end{align*}
 Fix \(N_0\ge5\). Since the elements of \(\Theta'\) are separated by
 \(\delta^{1/2}\mu^{-1}\), summation gives
 \[
 I\!I\lesssim
 \frac{\delta^2\mu^{-\frac74}(\mu')^{-\frac14}
 \bigl(\frac{\mu}{\mu'}\bigr)^{\frac12}}
 {(\lambda\delta^{\frac32}\mu^{-\frac32})^{\frac{N_0}{2}}}.
 \]
 Moreover, \(\delta\gtrsim\lambda^{-2/3}\mu\) implies
 \(\lambda\delta^{3/2}\mu^{-3/2}\gtrsim1\). Hence we may weaken the
 factor with exponent \(N_0/2\) to its first power, which gives
    \Be
    \label{e:est-sumI2}    I\!I\lesssim \lambda^{-1}\mu^{\frac{1}{4}}(\mu')^{-\frac{3}{4}} \min(\delta^{\frac{1}{2}}, \mu'). \Ee
    Combining  this and \eqref{e:est-sumI}, we obtain \eqref{e:sumPP*-1}.   \qed

\smallskip

 \subsection*{Proof of \texorpdfstring{\eqref{e:sumP*P-1}}{the second Cotlar--Stein sum}}  
    We now show the estimate  \eqref{e:sumP*P-1}.  As before, for $\nu\in\Theta$, we define 
\Be 
\label{ftf} 
\fT(\fr_\zc,\nu) = \{\fr\times \fr'\in \mathfrak R^\kappa_\delta:  2\fr\cap 2\fr_\zc\neq \emptyset, \  \theta(\fr)=\nu\}.
\Ee
Then, it follows that 
\[
 \{\fr\times \fr'\in \mathfrak R^\kappa_\delta:   2\fr\cap 2\fr_\zc\neq \emptyset\}\subset \bigcup_{\nu\in \Theta}   \fT(\fr_\zc,\nu).
\]

Note $\fs_\zc\in \mathfrak S(\fr_\zc')$ and $\fs \in \mathfrak S(\fr')$. Thus,    $(\fP_{\delta,\fr_\zc,\fs_\zc}^\kappa)^* \fP_{\delta,\fr,\fs}^\kappa=0 $  if  $2\fr_\zc \cap 2\fr=\emptyset$.
     Thus,   we have 
    \[  \sum_{\fr\times \fr'\in \fR^\kappa_\delta}\sum_{\fs\in\mathfrak S(\fr')} \mathbf d_{\fr, \fs} \le \sum_{\nu \in \Theta}\sum_{\fr\times\fr'\in \fT(\fr_\zc,\nu )} \sum_{\fs\in \mathfrak S(\fr')}\mathbf d_{\fr, \fs}=I\!I\!I + I\!V , \]
    where
    \begin{align*}
        I\!I\!I &= \sum_{\nu \in \Theta:\,\nu \neq \nu_\zc}\sum_{\fr\times\fr'\in \fT(\fr_\zc,\nu )} \sum_{\fs\in \mathfrak S(\fr')}\mathbf d_{\fr, \fs}, 
    \\ 
         I\!V &= \sum_{\fr\times\fr'\in \fT(\fr_\zc,\nu_\zc )} \sum_{\fs\in \mathfrak S(\fr')}\mathbf d_{\fr, \fs}.
         \end{align*}
  We first show the estimates for $I\!I\!I$. Since $|\nu'-\nu_\zc'|\sim|\nu-\nu_\zc|$, thanks to \eqref{nunu} with sufficiently large $K$, we  may  apply  $i)$ in Proposition \ref{prop:alortho-1} to  $\mathbf d_{\fr, \fs}$.
  Thus, noting that $ \#\mathfrak S(\fr')\lesssim (\mu/\mu')^{1/2}$, we obtain
  \[
  I\!I\!I
  \lesssim \sum_{\nu \in \Theta:\,\nu \neq \nu_\zc}
  \sum_{\fr\times\fr'\in \fT(\fr_\zc,\nu )}
  \frac{\delta^2\mu^{-\frac{5}{4}}(\mu')^{-\frac{3}{4}}}{(|\nu'-\nu'_\zc|\lambda\delta(\mu')^{-\frac{1}{2}})^{\frac{N}{2}}}.
  \]
   
   We now claim that 
  \[   \#   \fT(\fr_\zc,\nu ) \lesssim |\nu-\nu_\zc|\delta^{-\frac{1}{2}}\mu. \] 
  This can be shown by the same argument as before. Indeed, recall that $\fr, \fr_\zc$ are rectangles of dimensions about $\delta^{1/2}\mu^{1/2}(\mu')^{-1/2}\times \delta(\mu\mu')^{-1/2}$. The angle between their longer sides  is $|\nu-\nu_\zc|$.  Since $\fr$ is an element of the tiling $\mathcal R_{\delta, \nu}$ and 
	  $2\fr\cap 2\fr_\zc\neq \emptyset$, there are at most $O( |\nu-\nu_\zc| \delta^{1/2}\mu^{1/2}(\mu')^{-1/2}/(\delta(\mu\mu')^{-1/2}))$ rectangles $\fr$. 
  Once we fix such a $\fr$, there are only $O(1)$ choices of $\fr'$ such that $\fr\times\fr'  \in \fT(\fr_\zc,\nu )$. Indeed, since 
  $\fr\times\fr' \in \mathfrak R_\delta^\kappa$, this follows from \eqref{sharp0'} in 
  Lemma \ref{lem:cardinal-TT'}. 
 
  Then,  combining the inequalities  with   the inequality $|\nu_\zc-\nu |\sim |\nu_\zc'-\nu'|$,  
    we have 
   \[
  I\!I\!I
  \lesssim \sum_{\nu \in \Theta:\,\nu \neq \nu_\zc}
   \frac{(\lambda\delta^{\frac{3}{2}}\mu^{-1}(\mu')^{-\frac{1}{2}})^{-\frac{N}{2}}
   \delta^2\mu^{-\frac{5}{4}}(\mu')^{-\frac{3}{4}}}{(|\nu-\nu_\zc|\delta^{-\frac{1}{2}}\mu)^{\frac{N-2}{2}}}.
    \]
  Fix \(N_0\ge6\). Summing over the separated elements of \(\Theta\)
  gives
  \[
  I\!I\!I
  \lesssim
  (\lambda\delta^{\frac32}\mu^{-1}(\mu')^{-\frac12})^{-\frac{N_0}{2}}
  \delta^2\mu^{-\frac54}(\mu')^{-\frac34}.
  \]
  Since
  \(\lambda\delta^{3/2}\mu^{-1}(\mu')^{-1/2}\gtrsim1\), we may weaken
  the factor with exponent \(N_0/2\) to its first power, obtaining
  \Be
  \label{e:est-sumI3}
  I\!I\!I 
  \lesssim   \lambda^{-1}\delta^{\frac{1}{2}}(\mu\mu')^{-\frac{1}{4}}.  
  \Ee

   To estimate $I\!V$, we observe that 
	   $\#   \fT(\fr_\zc,\nu_\zc )\lesssim 1$.  Indeed, if $\fr\times\fr' \in \fT(\fr_\zc,\nu_\zc )$, then  $\theta(\fr)=\theta(\fr_\zc)=\nu_\zc$ and  
	   $2\fr\cap2\fr_\zc\neq \emptyset$. The separation condition \eqref{kfk} therefore forces $\fr=\fr_\zc$. Once such a $\fr$ is fixed, by using \eqref{sharp0'} in
Lemma \ref{lem:cardinal-TT'} as before, there are only $O(1)$ $\fr'$ such that $\fr\times\fr'\in\fT(\fr_\zc,\nu_\zc)$.  Moreover, the argument used to prove \eqref{2nd-claim'} and the separation condition \eqref{nunu} force $\theta(\fr')=\theta(\fr'_\zc)=\nu'_\zc$. The estimates established in the proof of \eqref{1st-claim'} then give $K\fr'\cap K\fr'_\zc\neq\emptyset$ when $K$ is chosen sufficiently large. Applying \eqref{kfk} once more yields $\fr'=\fr'_\zc$. Consequently, $
\fT(\fr_\zc,\nu_\zc)=\{\fr_\zc\times\fr'_\zc\}. $
  Therefore,  it remains to show  
 \[   I\!V = \sum_{\fs\in \mathfrak S(\fr'_\zc)}\mathbf d_{\fr_\zc, \fs} \lesssim \lambda^{-1}\delta^{\frac{1}{2}}(\mu\mu')^{-\frac{1}{4}} \]
 with $\fs_\zc\in \mathfrak S(\fr_\zc')$.     To this end, we split the sum 
 \[     I\!V =   \mathbf d_{\fr_\zc, \fs_\zc} + \sum_{\fs\in \mathfrak S(\fr'_\zc): \fs\neq \fs_\zc}\mathbf d_{\fr_\zc, \fs}. \]
Proposition  \ref{prop:L2est} gives $  \mathbf d_{\fr_\zc, \fs_\zc} \lesssim \lambda^{-1}\delta^{1/2}(\mu\mu')^{-1/4}$. 
 Now,  writing $\fs=\fs(\ell,\fr'_\zc)$ and  $\fs_\zc=\fs(\ell_\zc,\fr'_\zc)$,  we use Proposition \ref{prop:L2est} when $1\le|\ell-\ell_\zc|\lesssim 10$ and Proposition \ref{prop:alortho-2} when $|\ell-\ell_\zc|\gg 10$ to get
    \begin{align*}
 \sum_{\fs\in \mathfrak S(\fr'_\zc): \fs\neq \fs_\zc}\mathbf d_{\fr_\zc, \fs}
 &\lesssim
 \sum_{1\le|\ell-\ell_\zc|\lesssim 10}\lambda^{-1}\delta^{\frac{1}{2}}(\mu\mu')^{-\frac{1}{4}} +\sum_{|\ell-\ell_\zc|\gg 10}\frac{ \delta^2\mu^{-\frac{7}{4}}(\mu')^{-\frac{1}{4}} }{(|\ell-\ell_\zc|\lambda\delta^{\frac{3}{2}}\mu^{-\frac{3}{2}})^{\frac{N}{2}}}
\\
&\lesssim \lambda^{-1}\delta^{\frac{1}{2}}(\mu\mu')^{-\frac{1}{4}}.
    \end{align*}
  Therefore, combining this and \eqref{e:est-sumI3} gives \eqref{e:sumP*P-1}.    
   \qed
    
    \medskip
    
For the bottom-scale operators \(\fP_{\dels,\fr,\fs}\) appearing in
\eqref{d-pds}, part {\rm(iii)} of Proposition
\ref{prop:alortho-1} and \eqref{fpdel*} provide exactly the
off-diagonal estimates used above. Moreover,
\[
\lambda^{-\frac43}\mu^{\frac14}(\mu')^{-\frac14}
\sim
\lambda^{-1}\dels^{\frac12}(\mu\mu')^{-\frac14},
\]
so \eqref{e:fP-ast} supplies the same diagonal bound at
\(\delta=\dels\). Hence the analogues of \eqref{e:sumPP*-1} and
\eqref{e:sumP*P-1} hold with \(\dels\) in place of \(\delta\), and
the Cotlar--Stein lemma gives \eqref{2to2s}.

Propositions \ref{prop:main-components1-0} and
\ref{prop:main-components2-2} have now been proved; hence the
preceding deduction establishes Proposition \ref{main-p}. Combining
Proposition \ref{main-p} with \eqref{e:decompm-1st} and
\eqref{e:sumfElk} proves Theorem \ref{thm:core}. The reductions in
Section \ref{sec:prelim} then prove Theorem \ref{thm:decay}, and the
\(TT^\ast\) argument and dyadic annular summation recalled after
Theorem \ref{thm:decay} complete the proof of Theorem \ref{thm:main}.

 \section*{Acknowledgments} 
This work was supported by the National Research Foundation of Korea grants: G-LAMP RS-2024-00443714 and RS-2024-00339824 (Jeong); G-LAMP RS-2023-00301976 and  RS-2024-00342160 (Lee);  and RS-2026-25483775  (Ryu).
 \appendix 

\section{Proof of Lemma \ref{lem:mixedH}}
\label{app:mixed-hessian}

\begin{proof}
Write
\[
    \varepsilon_\ell:=(-1)^{\ell+1},\qquad
    R:=\sqrt{\cD(x,y)},\qquad
    \gamma_\ell:=\cos S_\ell(x,y),\qquad
    s_\ell:=\sin S_\ell(x,y),
\]
and suppress the dependence on $(x,y)$. By \eqref{cosS} and
\eqref{pxpy0}, we have
\[
    \partial_x \gamma_\ell=\frac{\varepsilon_\ell}{R}\mathbf v_\ell,
    \qquad
    \partial_x(s_\ell^{-1})
    =\frac{\varepsilon_\ell \gamma_\ell}{Rs_\ell^3}\mathbf v_\ell.
\]
Since \(\partial_t\cP(x,y,S_\ell)=0\), the chain rule and
the formula for \(\partial_y\cP\) give
\[
    \partial_y\phi_\ell
    =\partial_y\cP(x,y,S_\ell)
    =\frac{\mathbf v_\ell}{s_\ell}.
\]
Differentiating this identity with respect to $x$, we obtain
\begin{align*}
    \partial_x\partial_y^\intercal\phi_\ell
    &=
    -\frac{\mathbf I_2}{s_\ell}
    +
    \frac{\varepsilon_\ell}{Rs_\ell}
    \mathbf v_\ell y^\intercal
    +
    \frac{\varepsilon_\ell \gamma_\ell}{Rs_\ell^3}
    \mathbf v_\ell\mathbf v_\ell^\intercal
\\
    &=
    -\frac{\mathbf I_2}{s_\ell}
    +
    \frac{\varepsilon_\ell}{Rs_\ell^3}
    \mathbf v_\ell
    \bigl(s_\ell^2y+\gamma_\ell\mathbf v_\ell\bigr)^\intercal
\\
    &=
    -\frac{\mathbf I_2}{s_\ell}
    -
    \frac{\varepsilon_\ell}{Rs_\ell^3}
    \mathbf v_\ell\mathbf v_\ell'^\intercal,
\end{align*}
where, by the definition of $\mathbf v_\ell'$, 
$s_\ell^2y+\gamma_\ell\mathbf v_\ell
=y-\gamma_\ell x=-\mathbf v_\ell'$. Since
$-\varepsilon_\ell=(-1)^\ell$, this proves
\eqref{e:mixedH-unscaled}.

Directly from \eqref{cosS} and the definition of $\cD$, we also have
\begin{align*}
    |\mathbf v_\ell|^2
    &=(1-|y|^2)s_\ell^2,
&
    |\mathbf v_\ell'|^2
    &=(1-|x|^2)s_\ell^2,
&
    \inp{\mathbf v_\ell}{\mathbf v_\ell'}
    &=-\varepsilon_\ell R s_\ell^2.
\end{align*}
Applying \eqref{e:mixedH-unscaled} to $\mathbf v_\ell$ therefore gives
\begin{align*}
    \partial_x\partial_y^\intercal\phi_\ell\,
    \mathbf v_\ell
    &=
    -\frac{\mathbf v_\ell}{s_\ell}
    -
    \frac{\varepsilon_\ell}{Rs_\ell^3}
    \mathbf v_\ell
    \inp{\mathbf v_\ell'}{\mathbf v_\ell}
\\
    &=
    -\frac{\mathbf v_\ell}{s_\ell}
    +
    \frac{\mathbf v_\ell}{s_\ell}
    =0.
\end{align*}
This proves \eqref{e:mixedH-v}. Moreover, since
$\mathbf v_\ell^\perp
=\mathbf v_\ell'
-\inp{\mathbf v_\ell}{\mathbf v_\ell'}
|\mathbf v_\ell|^{-2}\mathbf v_\ell$, we have
\[
    \partial_x\partial_y^\intercal\phi_\ell\,
    \mathbf v_\ell^\perp
    =
    \partial_x\partial_y^\intercal\phi_\ell\,
    \mathbf v_\ell'.
\]
Using \eqref{e:mixedH-unscaled} once more,
\begin{align*}
    \partial_x\partial_y^\intercal\phi_\ell\,
    \mathbf v_\ell^\perp
    &=
    -\frac{\mathbf v_\ell'}{s_\ell}
    -
    \frac{\varepsilon_\ell|\mathbf v_\ell'|^2}{Rs_\ell^3}\mathbf v_\ell
\\
    &=
    -\frac{\varepsilon_\ell|\mathbf v_\ell'|^2}{Rs_\ell^3}
    \left(
        \mathbf v_\ell
        -
        \frac{\inp{\mathbf v_\ell}{\mathbf v_\ell'}}{|\mathbf v_\ell'|^2}\mathbf v_\ell'
    \right)
\\
    &=
    -\frac{\varepsilon_\ell|\mathbf v_\ell'|^2}{Rs_\ell^3}\mathbf v_\ell'^\perp.
\end{align*}
The identity
$|\mathbf v_\ell'|^2=(1-|x|^2)s_\ell^2$ now yields
\eqref{e:mixedH-vperp}. Since
\(\Phi_\ell=\delta^{-3/2}\mu^{3/2}\phi_\ell\) by \eqref{phiell},
the corresponding identities for \(\Phi_\ell\) follow by scaling.
\end{proof}

\end{document}